\documentclass[12pt,oneside,reqno,english]{amsart}

\usepackage[T1]{fontenc}
\usepackage[latin9]{inputenc}
\usepackage{geometry}
\usepackage{color}
\usepackage{textcomp}
\usepackage{mathrsfs}
\usepackage{amstext}
\usepackage{amsthm}
\usepackage{amssymb}
\usepackage{graphicx}
\usepackage{setspace}
\usepackage{wasysym}
\makeatletter
\numberwithin{equation}{section}
\numberwithin{figure}{section}
\theoremstyle{plain}
\newtheorem{thm}{\protect\theoremname}[section]
\theoremstyle{plain}
\newtheorem{prop}[thm]{\protect\propositionname}
\theoremstyle{remark}
\newtheorem{notation}[thm]{\protect\notationname}
\theoremstyle{remark}
\newtheorem{rem}[thm]{\protect\remarkname}
\theoremstyle{plain}
\newtheorem{cor}[thm]{\protect\corollaryname}
\theoremstyle{plain}
\newtheorem{lem}[thm]{\protect\lemmaname}
\theoremstyle{remark}
\newtheorem*{acknowledgement*}{\protect\acknowledgementname}

\usepackage{lmodern}
\DeclareFontFamily{OMX}{MnSymbolE}{}
\DeclareFontShape{OMX}{MnSymbolE}{m}{n}{
    <-6>  MnSymbolE5
   <6-7>  MnSymbolE6
   <7-8>  MnSymbolE7
   <8-9>  MnSymbolE8
   <9-10> MnSymbolE9
  <10-12> MnSymbolE10
  <12->   MnSymbolE12}{}
\DeclareSymbolFont{mnlargesymbols}{OMX}{MnSymbolE}{m}{n}
\SetSymbolFont{mnlargesymbols}{bold}{OMX}{MnSymbolE}{b}{n}
\DeclareMathDelimiter{\llangle}{\mathopen}{mnlargesymbols}{'164}{mnlargesymbols}{'164}
\DeclareMathDelimiter{\rrangle}{\mathclose}{mnlargesymbols}{'171}{mnlargesymbols}{'171}

\makeatother

\usepackage{babel}
\providecommand{\acknowledgementname}{Acknowledgement}
\providecommand{\corollaryname}{Corollary}
\providecommand{\lemmaname}{Lemma}
\providecommand{\notationname}{Notation}
\providecommand{\propositionname}{Proposition}
\providecommand{\remarkname}{Remark}
\providecommand{\theoremname}{Theorem}

\begin{document}

\title{Condensation of non-reversible zero-range processes}

\author{Insuk Seo}

\address{Department of Mathematical Sciences and RIMS, Seoul National University
\\
27-212, Gwanak-Ro 1, Gwanak-Gu, Seoul 08826, Republic of Korea.}

\email{insuk.seo@snu.ac.kr}
\begin{abstract}
In this article, we investigate the condensation phenomena for a class
of non-reversible zero-range processes on a fixed finite set. By establishing
a novel inequality bounding the capacity between two sets, and by
developing a robust framework to perform quantitative analysis on
the metastability of non-reversible processes, we prove that the condensed
site of the corresponding zero-range processes approximately behaves
as a Markov chain on the underlying graph whose jump rate is proportional
to the capacity with respect to the underlying random walk. The results
presented in the current paper complete the generalization of the
work of Beltran and Landim \cite{BL3} on reversible zero-range processes,
and that of Landim \cite{Lan2} on totally asymmetric zero-range processes
on a one-dimensional discrete torus.
\end{abstract}


\maketitle

\section{Introduction\label{sec1}}

Metastability is a generic phenomenon that occurs in several of models
in probability theory and statistical physics, such as random perturbations
of dynamical systems \cite{BEGK2,FW,LMS}, low-temperature ferromagnetic
spin systems \cite{BBI,BM,CGOV,LS2,NZ}, the stochastic partial differential
equations \cite{BG}, and the system of sticky particles \cite{BL3,BDG,GRV2,Lan2}.
For an extensive discussion of recent developments in this field,
certain monographs \cite{BdH,OV} can be referred to.

A crucial breakthrough was made by the phenomenal works \cite{BEGK1,BEGK2}
of Bovier, Eckhoff, Gayrard, and Klein. They connected potential theoretic
notions, such as the equilibrium potential and the capacity, and important
quantities related to metastable behavior of the system, such as transition
time and hitting probability. Based on this connection, they established
a robust framework for the quantitative investigation of the metastable
behavior of reversible random dynamics. This framework is now called
the potential theoretic approach, and has been successfully applied
to numerous metastable situations. For a detailed description of this
approach, we recommend referring to \cite{BdH}.

Under the presence of multiple metastable valleys, it is natural to
describe the successive metastable transition as a limiting Markov
chain after suitable time rescaling. Based on the same level of knowledge
of potential theoretic quantities as the potential theoretic approach
introduced above, Beltran and Landim \cite{BL1,BL2} developed a framework
for obtaining this description, and this approach is now known as
the martingale approach. One of the advantages of this approach is
that it can be used for successful analyses of condensing phenomena
for sticky interacting particle systems such as simple inclusion processes
\cite{BDG,GRV2} and zero-range processes \cite{BL3,Lan2}. The current
paper focuses on the latter, i.e., the condensation of sticky zero-range
processes. Specifically, we provide a complete generalization of the
previous results obtained in \cite{BL3,Lan2}, based on a novel methodology
for non-reversible dynamics.

From the perspectives of either the potential theoretic approach or
the martingale approach, investigation of metastability of non-reversible
dynamics is far more complicated and challenging than that of reversible
dynamics, mainly because of two reasons. First, accurate estimation
of the capacity between metastable valleys, which is a crucial step
in both approaches, involves taking into consideration the so-called
flow structure corresponding to the dynamics. For the reversible case,
the estimation of the capacity is carried out via the Dirichlet principle,
which expresses the capacity as a minimization problem in a space
of functions with a certain boundary condition, and the optimizer
is given by the equilibrium potential. For non-reversible processes,
the minimization and maximization problems for the capacity have been
obtained from \cite{GL} and \cite{Slo}, respectively. These principles
are now called the Dirichlet--Thomson principles for non-reversible
processes. These problems are defined in the space of flows, of which
the treatment is far more complicated than that of the space of functions.
Furthermore, these problems require all test flows to be divergence-free.
Divergence-free flows are delicate objects, and hence this additional
restriction is a major source of the technical difficulty of the problem.
In spite of this difficulty, several metastability results for non-reversible
processes have been provided. In \cite{Lan2}, Landim provided a detailed
analysis of the condensation of totally asymmetric zero-range process
on a one-dimensional discrete torus. This result is the first sharp
metastability result for non-reversible dynamics. This work confronts
a min--max problem for the capacity directly, instead of relying
on the Dirichlet--Thomson principles. For the results based on the
Dirichlet--Thomson principles, we recommend referring to \cite{LMS,LS1,LS2}.

Second, the so-called mean jump rate between metastable valleys is
difficult to estimate in the non-reversible case. Mean jump rate is
an essential notion in the Markov chain description of metastability,
in the spirit of the martingale approach, and thus needs to be estimated
precisely. For the reversible case, this mean jump rate is merely
obtained based on the capacities between valleys. By contrast, there
is no clear relationship between the capacity and mean jump rate in
the non-reversible case. In \cite{BL2}, the collapsed chain is introduced
as a potential tool to surmount this difficulty, and this possibility
has been confirmed for totally asymmetric zero-range processes \cite{Lan2}
and cyclic random walks in a potential field \cite{LS1}.

Our achievements in the current article can be divided into three
parts. First, we provide a generalized version of the Dirichlet and
Thomson principles (Theorem \ref{t04}) that is far more convenient
to apply in the asymptotic analysis of capacity than the classical
principle, as we have removed the divergence-free restriction.

The second achievement is the establishment of a general method to
deduce mean jump rate from the capacity estimate. This has been previously
addressed in \cite{LS1}, but the methodology described therein also
relies on the divergence-free flow. We remove this restriction. Consequently,
we can reduce the entire problem to constructions of certain approximating
functions and flows. This reduction is thoroughly explained in Section
\ref{sec6} and is believed to be model-independent. The model-dependent
part is the construction of approximating objects, and this part requires
a deep understanding of the flow structure.

The third achievement is the completion of the dynamical analysis
of the condensation of zero-range processes on a fixed finite set.
There have been several studies on the condensation of zero-range
processes \cite{AGL,BL3,GoL,GSS,JMP,Lan2,LLM}. Condensation is a
wide-spread phenomenon and it indicates that a macroscopically significant
portion of particles is concentrated on one site with dominating probability.
We recommend referring to \cite{EH,Lan2} and the references therein
for examples of condensation. Pertaining to the dynamical analysis
of condensation of zero-range processes, it has been conjectured within
the community that the transition of a condensed site occurs approximately
as a Markov chain whose jump rate is proportional to the capacity
between two sites for the underlying random walk. This has been confirmed
in \cite{BL3} for the reversible zero-range processes with $\alpha>2$,
and in \cite{Lan2} for the totally asymmetric zero-range processes
on the discrete torus with $\alpha>3$, where $\alpha$ is a parameter
that represents the stickiness of constituent particles. In this article,
we extend this result for any non-reversible zero-range process on
any fixed finite set for all $\alpha>2$, and finally verify that
the conjecture holds for this level of generality. This main result
is stated in Theorem \ref{t02}.

\subsubsection*{Organization of the article}

The rest of this paper is organized as follows. We introduce zero-range
processes and relevant notations in Section \ref{sec2}, and state
the main result regarding the condensation of zero-range processes
in Section \ref{sec3}. In Section \ref{sec4}, we introduce adjoint
dynamics and prove the sector condition. This sector condition is
not a crucial ingredient of the proof, but simplifies the proof remarkably
in some scenarios. In Section \ref{sec5}, we review the flow structure
and then formulate the generalized Dirichlet--Thomson principle,
which is one of the main achievements of the article. In Section \ref{sec6},
we develop the general framework for the quantitative analysis of
metastability of non-reversible processes, and prove the main result
stated in Section \ref{sec3}. In Sections \ref{sec7} and \ref{sec8},
we construct several approximating objects playing a central role
in the proof presented in Section \ref{sec6}.

\section{Zero-range processes\label{sec2}}

In this section, the zero-range process and several relevant notions
are introduced. Most notations are inspired from \cite{BL3} and hence
similar to those therein. However, different notations for several
sets and functions are used.

\subsection{\label{sec21}Underlying random walk}

A zero-range process is a system of interacting particles. Herein,
a Markov chain is introduced that describes the underlying movement
of the particles. The zero-range interaction mechanism among them
is explained in the next subsection.

Let $S$ be a finite set with $|S|=\kappa$, and let $\{X(t):t\ge0\}$
be a continuous-time, irreducible Markov chain on $S$, so that the
jump rate from a site $x\in S$ to $y\in S$ is given by $r(x,\,y)$
for some $r:S\times S\rightarrow[0,\,\infty)$. We assume that $r(x,\,x)=0$
for all $x\in S$. The invariant measure of Markov chain $X(\cdot)$
is denote by $m(\cdot)$, namely,
\begin{equation}
\sum_{y\in S}m(x)r(x,\,y)=\sum_{y\in S}m(y)r(y,\,x)\;\;\mbox{for all }x\in S\;.\label{inv}
\end{equation}
Let
\begin{equation}
M_{\star}=\max\{m(x):x\in S\}\;\;,\,S_{\star}=\{x\in S:m(x)=M_{\star}\}\;\text{\;and\;\;}\kappa_{\star}=|S_{\star}|\;.\label{not1}
\end{equation}
That is, $S_{\star}\subseteq S$ represents the set of sites with
maximum mass, with respect to the invariant measure $m(\cdot)$, and
$\kappa_{\star}$ denotes the number of these sites. The normalized
mass is defined by
\begin{equation}
m_{\star}(x)=\frac{m(x)}{M_{\star}}\in(0,\,1]\;\;;\;x\in S\;,\label{not2}
\end{equation}
so that $m_{\star}(x)=1$ for all $x\in S_{\star}$.

It is assumed that $\kappa_{\star}\ge2$, so that the zero-range process
corresponding to this Markov chain, defined in the next subsection,
exhibits the metastable behavior. In particular, in Theorem \ref{t01},
it is observed that under the invariant distribution of the zero-range
dynamics, most particles are concentrated at a site of $S_{\star}$,
with dominating probability.

For $f:S\rightarrow\mathbb{R}$, the generator $L_{X}$ and the Dirichlet
form $D_{X}(\cdot)$ associated with the Markov chain $X(\cdot)$
can be written as
\begin{align*}
 & (L_{X}f)(x)=\sum_{y\in S}r(y,\,x)(f(y)-f(x))\;\;;\;x\in S\;\;\text{and\;}\\
 & D_{X}(f)=\sum_{x\in S}m(x)f(x)(-L_{X}f)(x)=\frac{1}{2}\sum_{x,\,y\in S}m(x)r(x,\,y)\left[f(y)-f(x)\right]^{2}\;,
\end{align*}
respectively.

\subsection{Zero-range process}

A zero-range process is now defined that is an interacting system
of $N$ particles, where the particles follow the law of random walk
$X(\cdot)$ defined above, but interact through the zero-range interaction.

\subsubsection*{Definition of particle systems}

In the study of sticky zero-range process, a parameter $\alpha$ represent
the stickyness of constituent particles. In this article, we assume
that $\alpha>2$. Let $a:\mathbb{N}\rightarrow\mathbb{R}$ (the convention
$\mathbb{N}=\{0,\,1,\,2,\,\cdots\}$ is used) be a function defined
by
\[
a(n)=\begin{cases}
1 & \mbox{if }n=0\\
n^{\alpha} & \mbox{if }n\ge1\;.
\end{cases}
\]
Moreover, let $g:\mathbb{N}\rightarrow\mathbb{R}$ be a function defined
by
\[
g(n)=\begin{cases}
0 & \mbox{if }n=0\\
a(n)/a(n-1) & \mbox{if }n\ge1\;.
\end{cases}
\]
 For $N\in\mathbb{N}$, the set $\mathcal{\mathcal{H}}_{N}\subset\mathbb{N}^{S}$,
representing the set of configuration of $N$ particles on $S$, is
defined by
\[
\mathcal{\mathcal{H}}_{N}=\Bigl\{\eta=(\eta_{x})_{x\in S}\in\mathbb{N}^{S}:\sum_{x\in S}\eta_{x}=N\Bigr\}\;.
\]
The zero-range process $\{\eta_{N}(t):t\ge0\}$ is defined as a continuous-time
Markov chain on $\mathcal{H}_{N}$ associated with the generator
\[
(\mathscr{L}_{N}\mathbf{f})(\eta)=\sum_{x,\,y\in S}g(\eta_{x})r(x,\,y)(\mathbf{f}(\sigma^{x,\,y}\eta)-\mathbf{f}(\eta))\;\;;\;\eta\in\mathcal{H}_{N}\;,
\]
for $\mathbf{f}:\mathcal{H}_{N}\rightarrow\mathbb{R}$, where $\sigma^{x,\,y}\eta\in\mathcal{H}_{N}$
is the configuration obtained from $\eta$ by sending a particle at
site $x$ to $y$. More precisely, $\sigma^{x,\,y}\eta=\eta$ if $\eta_{x}=0$,
and if $\eta_{x}\ge1$, then
\[
(\sigma^{x,\,y}\eta)_{z}=\begin{cases}
\eta_{z}-1 & \mbox{if }z=x\\
\eta_{z}+1 & \mbox{if }z=y\\
\eta_{z} & \mbox{otherwise\;.}
\end{cases}
\]
For convenience $\sigma^{x,\,x}$ is regarded as the identity operator.
It is not difficult to verify that the zero-range process defined
above is irreducible. For $\eta\in\mathcal{H}_{N}$, let $\mathbb{P}_{\eta}^{N}$
be the law of the zero-range process $\eta_{N}(\cdot)$ starting from
$\eta$, and let $\mathbf{\mathbb{E}}_{\eta}^{N}$ be the corresponding
expectation.

In the particle dynamics defined above, each particle interacts only
with the particles at the same site through the function $g(\cdot)$.
Thus, it is called zero-range process. Moreover, in this model $g(n)$
is a decreasing function for $n\ge1$, and hence the movement of particles
is slowed down as the number of particle at the same site becomes
larger. This observation heuristically explains that the particles
are sticky, and this sticky behavior eventually causes their condensation.

\subsubsection*{Invariant measure and partition function}

For $\zeta\in\mathbb{N}^{S_{0}}$, $S_{0}\subseteq S$, let
\begin{equation}
m_{\star}^{\zeta}=\prod_{x\in S_{0}}m_{\star}(x)^{\zeta_{x}}\;\;\mbox{and\;\;}a(\zeta)=\prod_{x\in S_{0}}a(\zeta_{x})\;.\label{rec1}
\end{equation}
Then, the unique invariant measure $\mu_{N}(\cdot)$ on $\mathcal{H}_{N}$
of the zero-range process defined above is given by
\[
\mu_{N}(\eta)=\frac{N^{\alpha}}{Z_{N}}\frac{m_{\star}{}^{\eta}}{a(\eta)}\;\;;\;\eta\in\mathcal{H}_{N}\;,
\]
where $Z_{N}$ represents the partition function that turns $\mu_{N}$
into a probability measure, that is,
\[
Z_{N}=N^{\alpha}\sum_{\eta\in\mathcal{H}_{N}}\frac{m_{\star}^{\eta}}{a(\eta)}\;.
\]
Let
\[
\Gamma_{x}=\sum_{j=0}^{\infty}\frac{m_{\star}(x)^{j}}{a(j)}\;\;\text{for }x\in S\;,\;\;\mbox{and}\;\;\Gamma(\alpha)=\sum_{j=0}^{\infty}\frac{1}{a(j)}\;,
\]
so that $\Gamma_{x}=\Gamma(\alpha)$ for $x\in S_{\star}$. The series
converge because $\alpha>2$. Let now a constant be defined by
\begin{equation}
Z=\kappa_{\star}\Gamma(\alpha)^{\kappa_{\star}-1}\prod_{x\in S\setminus S_{\star}}\Gamma_{x}\;.\label{e20}
\end{equation}
The asymptotic result for the partition function is obtained as follows.
\begin{prop}
\label{e21}We have that
\[
\lim_{N\rightarrow\infty}Z_{N}=Z\;.
\]
\end{prop}

\begin{proof}
See \cite[Proposition 2.1]{BL3}.
\end{proof}

\subsubsection*{Dirichlet form}

For $\mathbf{f},\,\mathbf{g}:\mathcal{H}_{N}\rightarrow\mathbb{R}$,
the inner product $\left\langle \mathbf{f},\,\mathbf{g}\right\rangle _{\mu_{N}}$
is defined by
\[
\left\langle \mathbf{f},\,\mathbf{g}\right\rangle _{\mu_{N}}=\sum_{\eta\in\mathcal{H}_{N}}\mathbf{f}(\eta)\,\mathbf{g}(\eta)\,\mu_{N}(\eta)\;.
\]
Then, for $\mathbf{f}:\mathcal{H}_{N}\rightarrow\mathbb{R}$, the
Dirichlet form associated with the process $\eta_{N}(\cdot)$ is defined
by
\[
\mathscr{D}_{N}(\mathbf{f})=\left\langle \mathbf{f},\,-\mathscr{L}_{N}\mathbf{f}\right\rangle _{\mu_{N}}\;.
\]
By summation by parts, the Dirichlet form can be rewritten as
\[
\mathscr{D}_{N}(\mathbf{f})=\frac{1}{2}\sum_{x,\,y\in S}\mu(\eta)\,g(\eta_{x})\,r(x,\,y)\left[\mathbf{f}(\sigma^{x,\,y}\eta)-\mathbf{f}(\eta)\right]^{2}\;.
\]

\subsection{Equilibrium potential and capacity}

Two potential theoretic notions, namely, equilibrium potential and
capacity, related to the underlying random walk and the associated
zero-range process introduced above will now be explained. Denote
by $\tau_{A}$ the hitting times of the set $A\subset S$, namely
\[
\tau_{A}=\inf\left\{ t:X(t)\in A\right\} \;.
\]
The hitting time $\tau_{\mathcal{A}}$ of a set $\mathcal{A}\subset\mathcal{H}_{N}$
is defined analogously. It should be noted that standard Roman fonts
are used for representing subsets or elements of $S$, and calligraphic
fonts for representing subsets of $\mathcal{H}_{N}$. The configurations
in $\mathcal{H}_{N}$ are denoted by Greek letters.

Let $\mathbf{P}_{x}$, $x\in S$, denote the law of the underlying
Markov chain $X(\cdot)$ starting from a site $x.$ For two disjoint
and non-empty sets $A,\,B\subset S$, the equilibrium potential $h_{A,\,B}:S\rightarrow[0,\,1]$
for the process $X(\cdot)$ is defined by
\[
h_{A,\,B}(x)=\mathbf{P}_{x}[\tau_{A}<\tau_{B}]\;\;;\;x\in S\;.
\]
It is well known that the equilibrium potential $h_{A,B}$ satisfies
\begin{equation}
h_{A,\,B}\equiv1\;\mbox{on}\;A\;,\;\;h_{A,\,B}\equiv0\;\mbox{on}\;B\;,\;\mbox{and}\;\;L_{X}h_{A,\,B}\equiv0\;\mbox{on }(A\cup B)^{c}\;.\label{eqp1}
\end{equation}
Then, the capacity between $A$ and $B$ with respect to the process
$X(\cdot)$ is defined by
\[
\textup{cap}_{X}(A,\,B)=D_{X}(h_{A,\,B})\;.
\]
Thus, by \eqref{eqp1}, the following alternative representation of
capacity is obtained:
\begin{equation}
\textup{cap}_{X}(A,\,B)=-\sum_{x\in A}m(x)\,(L_{X}h_{A,\,B})(x)=\sum_{x\in B}m(x)\,(L_{X}h_{A,\,B})(x)\;.\label{eqp11}
\end{equation}

One can define the equilibrium potential and the capacity for zero-range
processes as well. For two disjoint and non-empty sets $\mathcal{A},\,\mathcal{B}\subset\mathcal{H}_{N}$,
the equilibrium potential is defined by
\[
\mathbf{h}_{\mathcal{A},\,\mathcal{B}}(\eta)=\mathbf{h}_{\mathcal{A},\,\mathcal{B}}^{N}(\eta):=\mathbb{P}_{\eta}^{N}\left[\tau_{\mathcal{A}}<\mathcal{\tau_{\mathcal{B}}}\right]\;\;;\;\eta\in\mathcal{H}_{N}
\]
Here and in the following the notation is simplified by dropping the
dependency on $N$. As in \eqref{eqp1}, the equilibrium potential
$\mathbf{h}_{\mathcal{A},\mathcal{\,B}}$ satisfies
\begin{equation}
\mathbf{h}_{\mathcal{A},\,\mathcal{B}}\equiv1\;\mbox{on}\;\mathcal{A}\;,\;\;\mathbf{h}_{\mathcal{A},\,\mathcal{B}}\equiv0\;\mbox{on}\;\mathcal{B}\;,\;\mbox{and}\;\;\mathscr{L}_{N}\mathbf{\,h}_{\mathcal{A},\,\mathcal{B}}\equiv0\;\mbox{on }(\mathcal{A}\cup\mathcal{B})^{c}\;.\label{eqp2}
\end{equation}
The capacity between $\mathcal{A}$ and $\mathcal{B}$ is defined
by
\begin{equation}
\textup{cap}_{N}(\mathcal{A},\,\mathcal{B})=\mathscr{D}_{N}(\mathbf{h}_{\mathcal{A},\,\mathcal{B}})=-\sum_{\eta\in\mathcal{A}}\mu_{N}(\eta)(\mathscr{L}_{N}\,\mathbf{h}_{\mathcal{A},\mathcal{\,B}})(\eta)=\sum_{\eta\in\mathcal{B}}\mu_{N}(\eta)(\mathscr{L}_{N}\,\mathbf{h}_{\mathcal{A},\mathcal{\,B}})(\eta)\;,\label{cap2}
\end{equation}
where the last two equalities follow from \eqref{eqp2}.

Finally, it should be remarked that, if a set in the definitions above
is a singleton, then the curly brackets will be dropped. For instance,
if $A=\{a\}$ and $B=\{b\}$, then the notation $h_{a,\,b}$ will
be used instead of $h_{\{a\},\,\{b\}}$.

\section{Main result\label{sec3}}

\textcolor{black}{In this section, the main result for the condensation
of non-reversible zero-range processes is presented. This phenomenon
can be understood as metastable behavior; hence, metastable valleys
around the condensed configurations are first defined in Section \ref{sec31}.
Then, in Section \ref{sec32}, the main result describing the metastable
behavior as a limiting Markov chain is presented.}

\subsection{\label{sec31}Metastable valleys}

\subsubsection*{Auxiliary sequences}

Several auxiliary sequences are introduced to concretely describe
metastability. For two sequences $(a_{N})_{N\in\mathbb{N}},\,(b_{N})_{N\in\mathbb{N}}$
of positive real numbers, the notation $a_{N}\ll b_{N}$ implies that
\[
\lim_{N\rightarrow\infty}\frac{b_{N}}{a_{N}}=\infty\;.
\]
A sequence of positive integers $(\pi_{N})_{N\in\mathbb{N}}$ is defined
by
\[
\pi_{N}=\left\lfloor N^{\frac{1}{\alpha}+\frac{1}{2}}\right\rfloor \ll N\;,
\]
where $\left\lfloor x\right\rfloor $ denotes the largest integer
not larger than $x$.

Let now $(\ell_{N})_{N\in\mathbb{N}}$ and $(b_{N}(z))_{N\in\mathbb{N}}$,
$z\in S\setminus S_{\star}$, be sequences of positive integer such
that\begin{equation}
\label{amp3}
\begin{aligned}
&1\ll\ell_{N}\ll\pi_{N}\;,\;\;1\ll b_{N}(z)\;\;\text{for all }z\in S\setminus S_{\star}\;,\;\;\text{and}\\
&\lim_{N\rightarrow\infty}\frac{\ell_{N}^{1+\alpha(\kappa-1)}}{N^{1+\alpha}}\prod_{z\in S\setminus S_{\star}}m_{\star}(z)^{-b_{N}(z)}=0\;.
\end{aligned}
\end{equation}For instance,
\[
\ell_{N}=\left\lfloor N^{\frac{1}{2(\kappa-1)}}\right\rfloor \;\;\text{and\;\;}b_{N}(z)=\left\lfloor \frac{\log N}{-2\kappa\log m_{\star}(z)}\right\rfloor \text{ for }z\in S\setminus S_{\star}\;,
\]
satisfy all the assumptions above

\subsubsection*{Metastable valleys}

For each $x\in S_{\star}$, the metastable valley representing the
set of configurations such that most particles are condensed at the
vertex $x$ is defined by
\[
\mathcal{E}_{N}^{x}=\left\{ \eta\in\mathcal{H}_{N}:\eta_{x}\ge N-\ell_{N}\;\mbox{and\;}\eta_{z}\le b_{N}(z)\;\mbox{for all }z\in S\setminus S_{\star}\right\} \;.
\]
For a non-empty set $A\subseteq S_{\star}$, define
\[
\mathcal{E}_{N}(A)=\bigcup_{x\in A}\mathcal{E}_{N}^{x}\;,
\]
and let
\[
\mathcal{E}_{N}=\mathcal{E}_{N}(S_{\star})\;\;\text{and\;\;}\Delta_{N}=\mathcal{H}_{N}\setminus\mathcal{E}_{N}\;.
\]
It should be emphasized that the definitions of the invariant measure
as well as the metastable valleys are identical to the reversible
zero-range process considered in \cite{BL3}. Hence the following
result on invariant measure is immediate from \cite[display (3.2)]{BL3}.
\begin{thm}
\label{t01}The invariant measure $\mu_{N}(\cdot)$ is concentrated
on the valleys defined above, in the sense that
\[
\lim_{N\rightarrow\infty}\mu_{N}(\mathcal{E}_{N}^{x})=\frac{1}{\kappa_{\star}}\text{ for all }x\in S_{\star}\;\;\text{and\;\;}\lim_{N\rightarrow\infty}\mu_{N}(\Delta_{N})=0\;.
\]
\end{thm}

This theorem explains the static condensation of zero-range process.
The main concern here is the dynamical analysis of this condensation
behavior. Suppose that almost all particles are condensed at a certain
site of $S_{\star}$. Then, after a sufficiently long time, the particles
are moved and condensed at another site, and this will be sequentially
repeated. This is a type of metastable behavior; hence, its analysis
lies on the framework of Beltran and Landim \cite{BL1, BL2}.

\subsection{\label{sec32}Condensation of zero-range processes}

The standard method for expressing the ing behavior in terms of the
convergence to a Markov chain under the presence of multiple metastable
valleys is the martingale approach developed in \cite{BL1, BL2} and
enhanced in \cite{LLM}. The main result of this study is explained
in the spirit of this approach to metastability. To this end, a projection
function $\Psi:\mathcal{H}_{N}\rightarrow S_{\star}\cup\{\mathfrak{0}\}$
is first defined by
\[
\Psi(\eta)=\begin{cases}
x & \mbox{if }x\in\mathcal{E}_{N}^{x}\;,\\
\mathfrak{0} & \mbox{if }x\in\Delta_{N}\;.
\end{cases}
\]
A projection of the zero-range process $\eta_{N}(\cdot)$ is then
defined by
\[
Y_{N}(t)=\Psi(\eta_{N}(t))\;.
\]
It should be noticed that the (non-Markov) process $Y_{N}(\cdot)$
on $S_{\star}\cup\{\mathfrak{0}\}$ and represents the valley at which
the zero-range process is staying at time $t$. The null state $\mathfrak{0}$
indicates the state at which the zero-range process does not exhibit
condensation. Then, metastability can be represented in terms of the
convergence of $Y_{N}(\cdot)$ to a Markov chain $Y(\cdot)$ on $S_{\star}\cup\{\mathfrak{0}\}$
defined below.

\subsubsection*{Limiting Markov chain describing metastable behavior}

Define a constant by
\begin{equation}
I_{\alpha}=\int_{0}^{1}u^{\alpha}(1-u)^{\alpha}du\;.\label{ialpha}
\end{equation}
The following remark on notation, which is valid throughout the paper,
is now in order.
\begin{notation}
\label{nott}The notation $u,\,v\in T$ or $\{u,\,v\}\subset T$ for
some set $T$ automatically implies that $u$ and $v$ are \textit{different}
elements of the set $T$.
\end{notation}

Let $\left\{ Y(t):t\ge0\right\} $ be a Markov chain on $S_{\star}\cup\{\mathfrak{0}\}$,
whose jump rate is given by
\begin{equation}
a(x,\,y)=\frac{1}{M_{\star}\,\Gamma(\alpha)\,I_{\alpha}}\,\textup{cap}_{X}(x,\,y)\;\;;\;x,\,y\in S_{\star}\;,\label{eax}
\end{equation}
and $a(x,\,y)=0$ otherwise. As the capacity is symmetric, it is easy
to verify that the invariant measure $\mu(\cdot)$ of this process
is given by
\begin{equation}
\mu(x)=\begin{cases}
1/\kappa_{\star} & \mbox{if }x\in S_{\star}\;,\\
0 & \mbox{if }x=\mathfrak{0}\;.
\end{cases}\label{mux}
\end{equation}
The Markov chain $Y(\cdot)$ is a long-range process in the sense
that $a(x,\,y)>0$ for all $x,\,y\in S_{\star}$. Let $\mathbf{Q}_{x}$
denote the law of the Markov chain $Y(\cdot)$ starting at $x\in S_{\star}$.
\begin{notation}
\label{rem32}Let $\{\widehat{Y}(t):t\ge0\}$ denote the Markov chain
on $S_{\star}$ with jump rate $a(\cdot,\,\cdot)$. This process is
obtained by neglecting the null-state $\mathfrak{0}$ in $Y(\cdot)$.
Then, one can verify by simple algebra that $\widehat{Y}(\cdot)$
is an irreducible Markov chain and is reversible with respect to the
invariant measure $\mu(\cdot)$ conditioned on $S_{\star}$.
\end{notation}

\subsubsection*{Main result}

The main result of this study is the following theorem, which describes
the metastable transition of the condensation of the zero-range processes
in a precise manner. For $t\ge0$, define
\[
W_{N}(t)=Y_{N}(N^{1+\alpha}t)\;.
\]

\begin{thm}
\label{t03}For all $x\in S_{\star}$ and for all $(\eta_{N})_{N\in\mathbb{N}}$
such that $\eta_{N}\in\mathcal{E}_{N}^{x}$ for all $N$, the finite
dimensional distributions of the process $W_{N}(\cdot)$ under $\mathbb{P}_{\eta_{N}}^{N}$
converges to that of the law $\mathbf{Q}_{x}$, as $N$ tends to infinity.
\end{thm}

The proof of this theorem is given in Section \ref{sec6}. It should
be stressed that, in the theorem above, the convergence of the finite
dimensional distributions can be replaced with that in the soft topology
\cite{Lan1}.

\section{\label{sec4}Adjoint dynamics and sector condition}

For the investigation of the metastability of non-reversible processes,
numerous computations are involved with both the original and the
adjoint dynamics simultaneously. In particular, the Dirichlet and
Thomson principles stated in Theorem \ref{dt} highlight this fact.
Accordingly, the notations related to the adjoint chains and the symmetrized
chains of the non-reversible zero-range processes are introduced in
Section \ref{sec41}. The sector condition for zero-range processes
is proved in Section \ref{sec42}

\subsection{\label{sec41}Adjoint dynamics and symmetrized dynamics}

For two sites $x,\,y\in S$, let
\[
r^{*}(x,\,y)=r(y,\,x)\,m(y)/m(x)\;.
\]
The adjoint generators of $L_{X}$ with respect to $L^{2}(m)$ is
define by, for all $f:S\rightarrow\mathbb{R}$,
\[
(L_{X}^{*}\,f)(x)=\sum_{y\in S}r^{*}(x,\,y)\,(f(y)-f(x))\;\;;\;x\in S\;.
\]
Analogously, the adjoint generator of $\mathscr{L}_{N}$ with respect
to $L^{2}(\mu_{N})$ is defined by, for all $\mathbf{f}:\mathcal{H}_{N}\rightarrow\mathbb{R}$,
\[
(\mathscr{L}_{N}^{*}\,\mathbf{f})(\eta)=\sum_{x,y\in S}g(\eta_{x})\,r^{*}(x,\,y)\,(\mathbf{f}(\sigma^{x,\,y}\eta)-\mathbf{f}(\eta))\;\;;\;\eta\in\mathcal{H}_{N}\;.
\]
Then, the processes generated by $L_{X}^{*}$ and $\mathscr{L}_{N}^{*}$
are denoted by $\{X^{*}(t):t\ge0\}$ and $\{\eta_{N}^{*}(t):t\ge0\}$,
respectively, and are called the adjoint dynamics. It should be noted
that the original and adjoint dynamics share the invariant measure
and the Dirichlet form.

The equilibrium potentials $h_{A,\,B}^{*},\,\mathbf{h}_{\mathcal{A},\,\mathcal{B}}^{*}$,
and the capacities $\textup{cap}_{X}^{*}(A,\,B)$, $\textup{cap}_{N}^{*}(\mathcal{A},\,\mathcal{B})$
for these adjoint dynamics are defined as before. It is known from
\cite[display (2.4)]{GL} that although the equilibrium potentials
for the original dynamics and the adjoint dynamics are quite different,
the corresponding capacities are the same, i.e., $\textup{cap}_{X}^{*}(A,\,B)=\textup{cap}_{X}(A,\,B)$
and $\textup{cap}_{N}^{*}(\mathcal{A},\mathcal{\,B})=\textup{cap}_{N}(\mathcal{A},\,\mathcal{B})$
for all $A,\,B\subset S$ and $\mathcal{A},\,\mathcal{B}\subset\mathcal{H}_{N}$.

Another process of interest is the symmetrized zero-range process
$\{\eta_{N}^{s}(t):t\ge0\}$ on $\mathcal{H}_{N}$ with generator
$\mathscr{L}_{N}^{s}=(1/2)(\mathscr{L}_{N}+\mathscr{L}_{N}^{*})$.
One can verify that this process is reversible with respect to the
invariant measure $\mu_{N}(\cdot)$. Hence, the process $\eta_{N}^{s}(\cdot)$
is that considered in \cite{BL3}. Let $\textup{cap}_{N}^{s}(\cdot,\,\cdot)$
denote the capacity with respect to the process $\eta_{N}^{s}(\cdot)$.

\subsection{\label{sec42}Sector condition for the zero-range processes}

In \cite{BL3}, several estimates were obtained in the context of
reversible zero-range processes, and can be employed in this study
using the so-called sector condition for the zero-range process $\eta_{N}(\cdot)$,
which is proved in Proposition \ref{psector} below. In particular,
Corollary \ref{csector} provides \textit{rough} estimates of the
capacity $\textup{cap}_{N}(\cdot,\,\cdot)$ via the estimates in \cite{BL3}
for the symmetrized capacity $\textup{cap}_{N}^{s}(\cdot,\,\cdot)$.
It should be emphasized that for \textit{sharp} estimates, an entirely
new idea is required.

For $u\in S$, let $\omega^{u}=(\omega_{x}^{u})_{x\in S}\in\mathcal{H}_{1}$
be the configuration with one particle at site $x$, namely,
\[
\omega_{x}^{u}=\begin{cases}
1 & \mbox{if }x=u\\
0 & \mbox{otherwise.}
\end{cases}
\]
For $u\in S$ and $\eta\in\mathcal{H}_{N}$, let $\eta+\omega^{u}\in\mathcal{H}_{N+1}$
be the configuration obtained from $\eta$ by adding a particle at
site $u$. The configuration $\eta-\omega^{u}\in\mathcal{H}_{N-1}$
can be defined similarly, provided that $\eta_{u}\ge1$. Remark that,
for $u\in S$ and $\eta\in\mathcal{H}_{N}$ such that $\eta_{u}\ge1$,
we have
\begin{equation}
\mu_{N}(\eta)\,g(\eta_{u})=a_{N}\,\mu_{N-1}(\eta-\omega^{u})\,m(u)\;,\label{u1}
\end{equation}
where $a_{N}$ is defined by
\[
a_{N}=\frac{N^{\alpha}\,Z_{N-1}}{(N-1)^{\alpha}\,Z_{N}\,M_{\star}}\;.
\]
By Proposition \ref{e21}, it is immediate that
\begin{equation}
\lim_{N\rightarrow\infty}a_{N}=M_{\star}^{-1}\;.\label{u2}
\end{equation}

\begin{rem}
\label{rem41}Henceforth, all constants are assumed to depend on the
set $S$, the underlying random walk $X(\cdot)$, and the parameter
$\alpha$. Later on, dependency on a new parameter $\epsilon$ will
be additionally allowed, and this will be explicitly stated.
\end{rem}

\begin{prop}[Sector condition for zero-range processes]
\label{psector}There exists a constant $C_{0}>0$ such that for
all $\mathbf{f},\,\mathbf{g}:\mathcal{H}_{N}\rightarrow\mathbb{R}$,
we have
\[
\left\langle \mathbf{g},\,-\mathcal{\mathscr{L}}_{N}\mathbf{f}\right\rangle _{\mu_{N}}^{2}\le C_{0}\mathscr{\,D}_{N}(\mathbf{f})\mathscr{\,D}_{N}(\mathbf{g})\;.
\]
\end{prop}

\begin{proof}
By \eqref{u1} and the change of variable $\eta-\omega^{x}=\zeta$,
we have\begin{equation}
\label{ep41}
\begin{aligned}
\left\langle \mathbf{g},\,-\mathcal{\mathscr{L}}_{N}\mathbf{f}\right\rangle _{\mu_{N}}&=\sum_{\eta\in\mathcal{H}_{N}}\,\sum_{x,\,y\in S}\mu_{N}(\eta)\,g(\eta_{x})\,r(x,\,y)\left[\mathbf{f}(\eta)-\mathbf{f}(\sigma^{x,\,y}\eta)\right]\mathbf{g}(\eta)\\&=a_{N}\,\sum_{\zeta\in\mathcal{H}_{N-1}}\mu_{N-1}(\zeta)\,A(\mathbf{f},\,\mathbf{g};\zeta)\;,
\end{aligned}
\end{equation}where
\[
A(\mathbf{f},\,\mathbf{g};\zeta)=\sum_{x,\,y\in S}m(x)\,r(x,\,y)\left[\mathbf{f}(\zeta+\omega^{x})-\mathbf{f}(\zeta+\omega^{y})\right]\mathbf{g}(\zeta+\omega^{x})\;.
\]
By \eqref{inv},
\[
A(\mathbf{f},\,\mathbf{f};\zeta)=\frac{1}{2}\sum_{x,\,y\in S}m(x)\,r(x,\,y)\left[\mathbf{f}(\zeta+\omega^{x})-\mathbf{f}(\zeta+\omega^{y})\right]^{2}\;.
\]
Therefore, the Dirichlet form can be rewritten as
\begin{equation}
\mathscr{D}_{N}(\mathbf{f})=\frac{a_{N}}{2}\sum_{\zeta\in\mathcal{H}_{N-1}}\,\sum_{x,\,y\in S}\mu_{N-1}(\zeta)\,m(x)\,r(x,\,y)\left[\mathbf{f}(\zeta+\omega^{x})-\mathbf{f}(\zeta+\omega^{y})\right]^{2}\;.\label{edr}
\end{equation}

For $\zeta\in\mathcal{H}_{N-1}$, let
\[
\overline{\mathbf{g}}(\zeta)=\frac{1}{\kappa}\sum_{z\in S}\mathbf{g}(\zeta+\omega^{z})\;,
\]
where $\kappa=|S|$. By \eqref{inv}, it holds that
\begin{align*}
\sum_{x,\,y\in S}m(x)\,r(x,\,y)\left[\mathbf{f}(\zeta+\omega^{x})-\mathbf{f}(\zeta+\omega^{y})\right] & =0\;.
\end{align*}
From this identity, it follows that, for all $\lambda>0$, \begin{equation}
\label{ee3}
\begin{aligned}
\left|A(\mathbf{f},\,\mathbf{g};\zeta)\right|=&\,\Bigl|\sum_{x,\,y\in S}m(x)\,r(x,\,y)\left[\mathbf{f}(\zeta+\omega^{x})-\mathbf{f}(\zeta+\omega^{y})\right]\left[\mathbf{g}(\zeta+\omega^{x})-\overline{\mathbf{g}}(\zeta)\right]\Bigr|\\\le&\,\frac{\lambda}{2}\sum_{x,\,y\in S}m(x)\,r(x,\,y)\left[\mathbf{f}(\zeta+\omega^{x})-\mathbf{f}(\zeta+\omega^{y})\right]^{2}\\&\;+\frac{1}{2\lambda}\sum_{x,\,y\in S}m(x)\,r(x,\,y)\left[\mathbf{g}(\zeta+\omega^{x})-\overline{\mathbf{g}}(\zeta)\right]^{2}\;.
\end{aligned}
\end{equation}Let $E\subseteq S\times S$ be defined by $E=\{(x,\,y):r(x,\,y)>0\}$,
and let
\[
C_{1}=\min_{(x,\,y)\in E}m(x)r(x,\,y)\;\;,\;C_{2}=\max_{(x,\,y)\in E}m(x)r(x,\,y)\;.
\]
To each $u,\,v\in S$, a \textit{canonical path}
\[
u=z_{1}(u,\,v),\,z_{2}(u,\,v),\,\cdots,\,z_{k(u,\,v)}(u,\,v)=v
\]
is assigned such that
\[
(z_{i}(u,\,v),\,z_{i+1}(u,\,v))\in E\;\;\text{for all\;\;}1\le i\le k(u,\,v)-1\;.
\]
Here, it can be assumed that all $z_{i}(u,\,v)$, $1\le i\le k(u,\,v)$,
are different; hence, $k(u,\,v)\le\kappa$. The existence of such
a path is ensured by the irreducibility of $X(\cdot)$. Then, the
last summation of \eqref{ee3} can be bounded above by
\begin{equation}
C_{2}(\kappa-1)\sum_{x\in S}\left[\mathbf{g}(\zeta+\omega^{x})-\overline{\mathbf{g}}(\zeta)\right]^{2}=\frac{C_{2}(\kappa-1)}{\kappa}\sum_{u,\,v\in S}\left[\mathbf{g}(\zeta+\omega^{u})-\mathbf{g}(\zeta+\omega^{v})\right]^{2}\;.\label{ee1}
\end{equation}
By the Cauchy--Schwarz inequality, and the fact that $k(u,\,v)\le\kappa$,
the last summation can be bounded above by \begin{equation} \label{ee2} \begin{aligned}
& (\kappa-1)\,\sum_{u,\,v\in S}\,\sum_{i=1}^{k(u,\,v)-1}\left[\mathbf{g}(\zeta+\mathfrak{\omega}^{z_{i}(u,\,v)})-\mathbf{g}(\zeta+\omega^{z_{i+1}(u,\,v)})\right]^{2}\\
&\le\kappa^{2}(\kappa-1)\,\sum_{(x,\,y)\in E}\left[\mathbf{g}(\zeta+\omega^{x})-\mathbf{g}(\zeta+\omega^{y})\right]^{2}\\
&\le\frac{\kappa^{3}}{C_{1}}\,\sum_{(x,\,y)\in E}m(x)\,r(x,\,y)\left[\mathbf{g}(\zeta+\omega^{x})-\mathbf{g}(\zeta+\omega^{y})\right]^{2}\;.
\end{aligned} \end{equation} By \eqref{ee1} and \eqref{ee2}, there exists a constant $C>0$ such
that \begin{equation} \label{ee5} \begin{aligned}&\sum_{x,\,y\in S}m(x)\,r(x,\,y)\left[\mathbf{g}(\zeta+\omega^{x})-\overline{\mathbf{g}}(\zeta)\right]^{2}\\
&\le C\,\sum_{x,\,y\in S}m(x)\,r(x,\,y)\left[\mathbf{g}(\zeta+\omega^{x})-\mathbf{g}(\zeta+\omega^{y})\right]^{2}\;. \end{aligned} \end{equation} By \eqref{ep41}, \eqref{ee3}, and \eqref{ee5},
\[
\left|\left\langle \mathbf{g},\,-\mathcal{\mathscr{L}}_{N}\mathbf{f}\right\rangle _{\mu_{N}}\right|\le\lambda\mathcal{\mathscr{D}}_{N}(\mathbf{f})+\frac{C}{2\lambda}\mathcal{\mathscr{D}}_{N}(\mathbf{g})\;.
\]
The proof can be completed by optimizing over $\lambda>0$.
\end{proof}
Henceforth, the constant $C_{0}$ will always be used to denote the
constant appearing in Proposition \ref{psector}. The following corollary
is immediate from \cite[Lemmata 2.5 and 2.6]{GL}.
\begin{cor}
\label{csector}For any two disjoint, non-empty subsets $\mathcal{A},\,\mathcal{B}$
of $\mathcal{H}_{N}$, it holds that
\[
\textup{cap}_{N}^{s}(\mathcal{A},\,\mathcal{B})\le\textup{cap}_{N}(\mathcal{A},\,\mathcal{B})\le C_{0}\,\textup{cap}_{N}^{s}(\mathcal{A},\,\mathcal{B})\;.
\]
\end{cor}

\section{\label{sec5}Generalized Dirichlet-Thomson Principles}

The major technical difficulty in the quantitative analysis of metastability,
in the spirit of the potential theoretic analysis of Bovier, Eckhoff,
Gayrard, and Klein \cite{BEGK1, BEGK2} or the martingale approach
of Beltran and Landim \cite{BL1, BL2}, is the sharp estimates of
the capacities between metastable valleys. For reversible processes,
the notable observation is that the Dirichlet principle expresses
the capacity as the infimum of a variational formula for a class of
functions whose minimum is achieved by the equilibrium potential between
valleys. Hence, the sharp upper bound of the capacity can be immediately
obtained if we are able to find a test function of the Dirichlet principle
that accurately approximates the equilibrium potential between valleys.
Moreover, the lower bound of capacity can be usually obtained by the
dimension reduction technique \cite{BEGK1, BEGK2} or by the Thomson
principle for reversible Markov chains \cite{LMT}. The reader is
referred to the recent monograph \cite{BdH} for a comprehensive discussion
on this matter.

Recently, a sharp analysis of the capacity for several non-reversible
dynamics has also been obtained in \cite{LMS, LS1, LS2}, based on
the Dirichlet principle \cite{GL} and Thomson principle \cite{Slo}
for non-reversible dynamics. These principles are stated in Theorem
\ref{dt}. As they express the capacity as infimum and supremum, respectively,
of certain variational formulas, and the optimizers for these formulas
are explicitly known, the same strategy as in the reversible case
can be used. However, this program is notoriously complicated because
these principles require the divergence-free flow as a test flow.
Obtaining a divergence-free flow approximating the optimal flow requires
deep intuition about the underlying processes as well as highly complicated
computations. In view of this difficulty, one of the main achievement
of this study is Theorem \ref{t04} that removes the divergence-free
restriction for the test flow, and in turn allows the use of any flow
to bound the capacity from below and above. To explain this new result,
the flow structure for the zero-range process is first explained in
Section \ref{sec51}. Then, in Section \ref{sec52}, the Dirichlet
and Thomson principles for non-reversible Markov chains are reviewed
and the proposed generalization is developed.

\subsection{\label{sec51}Flow structure}

Herein, the flow structure is interpreted in terms of the non-reversible
zero-range processes considered in this study. The reader is referred
to \cite{GL, Slo, LS1} for a summary of the general theory in the
context of Markov chains and to \cite{LMS} for diffusion processes.

The flow structure is constructed on a directed graph whose vertex
set is $\mathcal{H}_{N}$. Two configurations $\eta,\,\zeta\in\mathcal{H}_{N}$
are called adjacent and denoted as $\eta\sim\zeta$, if $\zeta$ can
be obtained by a legitimate jump of a particle in the configuration
$\eta$ or vice versa. That is, $\eta\sim\zeta$ if there exists $\xi\in\mathcal{H}_{N-1}$
and $x,\,y\in S$ satisfying $r(x\,,y)+r(y,\,x)>0$ such that $\eta=\xi+\mathfrak{\omega}^{x}$
and $\zeta=\xi+\mathfrak{\omega}^{y}$. It should be noted that $\eta\sim\zeta$
if and only if $\zeta\sim\eta$. Finally, the set of directed edges
is definied by
\[
\mathcal{H}{}_{N}^{\otimes}=\{(\eta,\,\zeta)\in\mathcal{H}_{N}\times\mathcal{H}_{N}:\eta\sim\zeta\}\;.
\]
It should be remarked that $(\eta,\,\zeta)\in\mathcal{H}{}_{N}^{\otimes}$
if and only if $(\zeta,\,\eta)\in\mathcal{H}{}_{N}^{\otimes}$; however,
these two elements must be distinguished.

The conductance, adjoint conductance, and symmetrized conductance
between $\eta$ and $\zeta=\sigma^{x,\,y}\eta$ for some $x,\,y$
satisfying $r(x,y)+r(y,x)>0$ are defined by
\begin{align}
 & c_{N}(\eta,\,\zeta)=\mu_{N}(\eta)\,g(\eta_{x})\,r(x,\,y)\;,\nonumber \\
 & c_{N}^{*}(\eta,\,\zeta)=\mu_{N}(\eta)\,g(\eta_{x})\,r^{*}(x,\,y)\;,\label{cond1}\\
 & c_{N}^{s}(\eta,\,\zeta)=(1/2)\left[c_{N}(\eta,\,\zeta)+c_{N}^{*}(\eta,\,\zeta)\right]\;,\nonumber
\end{align}
respectively. If $\eta=\xi+\omega^{x}$ and $\zeta=\xi+\mathfrak{\omega}^{y}$
for some $\xi\in\mathcal{H}_{N-1}$, then, by \eqref{u1}, these conductances
can be written as \begin{equation} \label{cond2} \begin{aligned}
&c_{N}(\eta,\,\zeta)=a_{N}\,\mu_{N-1}(\xi)\,m(x)\,r(x,\,y)\;,\\
&c_{N}^{*}(\eta,\,\zeta)=a_{N}\,\mu_{N-1}(\xi)\,m(y)\,r(y,\,x)\;. \end{aligned} \end{equation}From these expressions, it is apparent that $c_{N}^{*}(\eta,\,\zeta)=c_{N}(\zeta,\,\eta)$
and $c_{N}^{s}(\eta,\,\zeta)=c_{N}^{s}(\zeta,\,\eta)>0$. Thus, $c_{N}^{s}(\cdot,\,\cdot)$
is a symmetric, positive function on $\mathcal{H}{}_{N}^{\otimes}$.

An anti-symmetric real-valued function on $\mathcal{H}{}_{N}^{\otimes}$
is called a \textit{flow}, i.e., $\phi:\mathcal{H}{}_{N}^{\otimes}\rightarrow\mathbb{R}$
is a flow if and only if $\phi(\eta,\,\zeta)=-\phi(\zeta,\,\eta)$
for all $(\eta,\,\zeta)\in\mathcal{H}{}_{N}^{\otimes}$. Let $\mathfrak{F}_{N}$
denote the set of flows on $\mathcal{H}{}_{N}^{\otimes}$. On this
set, an inner product is defined by,
\begin{equation}
\left\llangle \phi,\,\psi\right\rrangle =\left\llangle \phi,\,\psi\right\rrangle _{\mathfrak{F}_{N}}=\frac{1}{2}\sum_{(\eta,\,\zeta)\in\mathcal{H}{}_{N}^{\otimes}}\frac{\phi(\eta,\,\zeta)\,\psi(\eta,\,\zeta)}{c_{N}^{s}(\eta,\,\zeta)}\;\;;\;\phi,\,\psi\in\mathfrak{F}_{N}\;.\label{inner}
\end{equation}
The flow norm is defined by $\left\Vert \phi\right\Vert ^{2}=\left\Vert \phi\right\Vert _{\mathfrak{F}_{N}}^{2}:=\left\llangle \phi,\,\phi\right\rrangle $
for $\phi\in\mathcal{\mathfrak{F}}_{N}$.

Another important notion related to flows is \textit{divergence}.
For each $\eta\in\mathcal{H}_{N}$ and $\phi\in\mathfrak{F}_{N}$,
the divergence of the flow $\phi$ at $\eta$ is defined by
\[
(\mbox{div}\,\phi)(\eta):=\sum_{\zeta:\eta\sim\zeta}\phi(\eta,\,\zeta)\;.
\]
The divergence of $\phi$ on a set $\mathcal{A}\subseteq\mathcal{H}_{N}$
is defined by
\[
(\mbox{div}\,\phi)(\mathcal{A})=\sum_{\eta\in\mathcal{A}}\,(\mbox{div}\,\phi)(\eta)\;.
\]
A flow $\phi\in\mathfrak{F}_{N}$ is called divergence-free at $\eta$
if $(\mbox{div}\,\phi)(\eta)=0$. It is called divergence-free on
$\mathcal{A}\subseteq\mathcal{H}_{N}$ if $(\mbox{div}\,\phi)(\eta)=0$
for all $\eta\in\mathcal{A}$.

For $\mathbf{f}:\mathcal{H}_{N}\rightarrow\mathbb{R}$ and for $(\eta,\,\zeta)\in\mathcal{H}{}_{N}^{\otimes}$,
the objects $\Phi_{\mathbf{f}}=\Phi_{\mathbf{f}}^{N}$, $\Phi_{\mathbf{f}}^{*}=\Phi_{\mathbf{f}}^{N,*}$,
and $\Psi_{\mathbf{f}}=\Psi_{\mathbf{f}}^{N}$ are defined by
\begin{align}
\Phi_{\mathbf{f}}(\eta,\,\zeta) & =\mathbf{f}(\eta)\,c_{N}(\eta,\,\zeta)-\mathbf{f}(\zeta)\,c_{N}(\zeta,\,\eta)\;,\nonumber \\
\Phi_{\mathbf{f}}^{*}(\eta,\,\zeta) & =\mathbf{f}(\eta)\,c_{N}(\zeta,\,\eta)-\mathbf{f}(\zeta)\,c_{N}(\eta,\,\zeta)\;,\label{flow}\\
\Psi_{\mathbf{f}}(\eta,\,\zeta) & =c_{N}^{s}(\eta,\,\zeta)\left[\mathbf{f}(\eta)-\mathbf{f}(\zeta)\right]=(1/2)(\Phi_{\mathbf{f}}+\Phi_{\mathbf{f}}^{*})(\eta,\,\zeta)\;.\nonumber
\end{align}
It is elementary to verify that these objects, as functions on $\mathcal{H}{}_{N}^{\otimes}$,
are anti-symmetric; hence, they are flows. The following properties
for these flows are well known and will be frequently used later.
\begin{prop}
\label{pflow}With notations as above, the followings hold.
\begin{enumerate}
\item For all $\mathbf{f}:\mathcal{H}_{N}\rightarrow\mathbb{R}$ and $\eta\in\mathcal{H}_{N}$,
\[
(\textup{div}\,\Phi_{\mathbf{f}})(\eta)=-\mu_{N}(\eta)\,(\mathscr{L}_{N}^{*}\mathbf{\,f})(\eta)\;\;\mbox{and\;\;}(\textup{div}\,\Phi_{\mathbf{f}}^{*})(\eta)=-\mu_{N}(\eta)\,(\mathscr{L}_{N}\mathbf{\,f})(\eta)\;.
\]
Therefore, for disjoint non-empty subsets $\mathcal{A},\,\mathcal{B}$
of $\mathcal{H}_{N}$, the flows $\Phi_{\mathbf{h}_{\mathcal{A},\mathcal{\,B}}^{*}}$
and $\Phi_{\mathbf{h}_{\mathcal{A},\,\mathcal{B}}}^{*}$ are divergence-free
on $(\mathcal{A}\cup\mathcal{B})^{c}$.
\item For all $\mathbf{f},\,\mathbf{g}:\mathcal{H}_{N}\rightarrow\mathbb{R}$,
\[
\left\llangle \Psi_{\mathbf{f}},\,\Phi_{\mathbf{g}}\right\rrangle =\left\langle -\mathscr{L}_{N}\mathbf{f},\,\mathbf{g}\right\rangle _{\mu_{N}}\;\;\text{and}\;\;\left\llangle \Psi_{\mathbf{f}},\,\Phi_{\mathbf{g}}^{*}\right\rrangle =\left\langle -\mathscr{L}_{N}^{*}\mathbf{f},\,\mathbf{g}\right\rangle _{\mu_{N}}\;.
\]
\item For all $\mathbf{f}:\mathcal{H}_{N}\rightarrow\mathbb{R}$ and $\phi\in\mathfrak{F}_{N}$,
\[
\left\llangle \Psi_{\mathbf{\mathbf{f}}},\,\phi\right\rrangle =\sum_{\eta\in\mathcal{H}_{N}}\mathbf{f}(\eta)(\mbox{\textup{div} }\phi)(\eta)\;.
\]
\item For all $\mathbf{f}:\mathcal{H}_{N}\rightarrow\mathbb{R}$, it holds
that $||\Psi_{\mathbf{f}}||^{2}=\mathcal{\mathscr{D}}_{N}(\mathbf{f})$.
Therefore,
\[
\bigl\Vert\Psi_{\mathbf{h}_{\mathcal{A},\mathcal{\,B}}}\bigr\Vert^{2}=\bigl\Vert\Psi_{\mathbf{h}_{\mathcal{A},\,\mathcal{B}}^{*}}\bigr\Vert^{2}=\textup{cap}_{N}(\mathcal{A},\,\mathcal{B})\;.
\]
\end{enumerate}
\end{prop}

\begin{proof}
The proof follows by elementary algebra. The reader is referred to
\cite{Slo} for the proof.
\end{proof}

\subsection{\label{sec52}Generalization of the Dirichlet--Thomson principles}

Several classes of functions and flows are defined to explain the
Dirichlet and the Thomson principles for non-reversible Markov chains
and their generalizations. Fix two disjoint non-empty subsets $\mathcal{A},\,\mathcal{B}$
of $\mathcal{H}_{N}$ and $a,\,b\in\mathbb{R}$. In the definitions
below, the dependency on $N$ will be neglected as there is no risk
of confusion.
\begin{itemize}
\item Let $C_{a,\,b}(\mathcal{A},\,\mathcal{B})$ be the class of real-valued
functions $\mathbf{f}$ on $\mathcal{H}_{N}$ satisfying $\mathbf{f}|_{\mathcal{A}}\equiv a$
and $\mathbf{f}|_{\mathcal{B}}\equiv b$, i.e.,
\[
C_{a,\,b}(\mathcal{A},\,\mathcal{B})=\left\{ \mathbf{f}:\mathcal{H}_{N}\rightarrow\mathbb{R}:\mathbf{f}(\eta)=a,\,\forall\eta\in\mathcal{A}\;\mbox{and}\;\mathbf{f}(\eta)=b,\,\forall\eta\in\mathcal{B}\right\} \;.
\]
\item Let $\mathfrak{S}_{a}(\mathcal{A},\,\mathcal{B})\subset\mathcal{\mathfrak{F}}_{N}$
be the set of flows whose divergence on $\mathcal{A}$ is $a$, i.e.,
\[
\mathfrak{\mathfrak{S}}_{a}(\mathcal{A},\,\mathcal{B})=\left\{ \phi\in\mathcal{\mathfrak{F}}_{N}:(\textup{div }\phi)(\mathcal{A})=a\right\}
\]
\item Let $\mathfrak{DF}_{a}(\mathcal{A},\,\mathcal{B})\subset\mathfrak{\mathfrak{S}}_{a}(\mathcal{A},\,\mathcal{B})$
be the set of divergence-free flows from $\mathcal{A}$ to $\mathcal{B}$
of strength $a$, i.e.,
\[
\mathfrak{DF}_{a}(\mathcal{A},\,\mathcal{B})=\left\{ \phi\in\mathfrak{S}_{a}(\mathcal{A},\,\mathcal{B}):(\textup{div }\phi)(\eta)=0\;\mbox{for all }\eta\in(\mathcal{A}\cup\mathcal{B})^{c}\right\}
\]
It should be noticed that $\phi\in\mathfrak{DF}_{a}(\mathcal{A},\,\mathcal{B})$
implies that $(\textup{div }\phi)(\mathcal{B})=-a$.
\end{itemize}
In the following theorem, the Dirichlet and the Thomson principles
are stated for non-reversible Markov chains.
\begin{thm}
\label{dt}Let $\mathcal{A}$ and $\mathcal{B}$ be two disjoint and
non-empty subsets of $\mathcal{H}_{N}$. Then, the capacity between
$\mathcal{A}$ and $\mathcal{B}$ satisfies the following variational
formulas:
\begin{align}
\textup{cap}_{N}(\mathcal{A},\,\mathcal{B}) & =\inf_{\mathbf{f}\in C_{1,\,0}(\mathcal{A},\,\mathcal{B}),\;\phi\in\mathfrak{DF}_{0}(\mathcal{A},\,\mathcal{B})}\,||\Phi_{\mathbf{f}}-\phi||^{2}\label{DP}\\
 & =\sup_{\mathbf{g}\in C_{0,\,0}(\mathcal{A},\,\mathcal{B}),\;\psi\in\mathfrak{DF}_{1}(\mathcal{A},\,\mathcal{B})}\,\frac{1}{||\Phi_{\mathbf{g}}-\psi||^{2}}\;.\label{TP}
\end{align}
Furthermore, the unique optimizer of the first variational formula
is
\[
(\mathbf{f}_{0},\,\phi_{0})=\left(\frac{\mathbf{h}_{\mathcal{A},\,\mathcal{B}}+\mathbf{h}_{\mathcal{A},\,\mathcal{B}}^{*}}{2},\,\frac{\Phi_{\mathbf{h}_{\mathcal{A},\mathcal{\,B}}^{*}}-\Phi_{\mathbf{h}_{\mathcal{A},\mathcal{\,B}}}^{*}}{2}\right)\;,
\]
and of the second variational formula is
\[
(\mathbf{g}_{0},\,\psi_{0})=\left(\frac{\mathbf{h}_{\mathcal{A},\mathcal{\,B}}^{*}-\mathbf{h}_{\mathcal{A},\mathcal{\,B}}}{2\,\textup{cap}_{N}(\mathcal{A},\,\mathcal{B})},\,\frac{\Phi_{\mathbf{h}_{\mathcal{A},\mathcal{\,B}}^{*}}+\Phi_{\mathbf{h}_{\mathcal{A},\,\mathcal{B}}}^{*}}{2\,\textup{cap}_{N}(\mathcal{A},\,\mathcal{B})}\right)\;.
\]
\end{thm}

In the previous theorem, the first variational formula \eqref{DP}
was established in \cite{GL} and is called the Dirichlet principle.
The second formula \eqref{TP} was developed in \cite{Slo} and is
known as the Thomson principle.

To use Dirichlet principle, a test function $\mathbf{f}\in C_{1,\,0}(\mathcal{A},\,\mathcal{B})$
and a test flow $\phi\in\mathfrak{DF}_{0}(\mathcal{A},\,\mathcal{B})$
should be suitably chosen so that the upper bound $\textup{cap}_{N}(\mathcal{A},\,\mathcal{B})\le||\Phi_{\mathbf{f}}-\phi||^{2}$
is obtained. The sharpness of this bound is closely related to the
fact that $(\mathbf{f,\,\phi)}$ approximates $(\mathbf{f}_{0},\,\phi_{0})$.
Indeed, determining $\mathbf{f}\in C_{1,\,0}(\mathcal{A},\,\mathcal{B})$
that approximates the optimizer $\mathbf{f}_{0}$ especially when
$\mathcal{A}$ and $\mathcal{B}$ are metastable valleys is not usually
a difficult task. As mentioned earlier, the technical obstacle appears
in the construction of a test flow $\phi\in\mathfrak{DF}_{0}(\mathcal{A},\,\mathcal{B})$
properly approximating the optimizing flow $\phi_{0}$. The requirement
that $(\textup{div }\phi)(\eta)=0$\textit{ }for all $\eta\in(\mathcal{A}\cup\mathcal{B})^{c}$
is a severe restriction in usual applications. In particular, the
underlying graph for the flow structure, which is $\mathcal{H}_{N}^{\otimes}$
for the present model, is complicated, and the problem is thus particularly
diffucult. The following Theorem eliminates this divergence-free restriction
and hence widens the range of potential applications.
\begin{thm}
\label{t04}Let $\mathcal{A}$ and $\mathcal{B}$ be two disjoint
and non-empty subsets of $\mathcal{H}_{N}$, and let $\varepsilon$
be any real number.
\begin{enumerate}
\item For $\mathbf{f}\in C_{1,\,0}(\mathcal{A},\,\mathcal{B})$ and $\phi\in\mathfrak{\mathfrak{S}}_{\varepsilon}(\mathcal{A},\,\mathcal{B})$,
we have that
\begin{equation}
\textup{cap}_{N}(\mathcal{A},\,\mathcal{B})\le||\Phi_{\mathbf{f}}-\phi||^{2}+2\varepsilon+2\sum_{\eta\in(\mathcal{A}\cup\mathcal{B})^{c}}\mathbf{\mathbf{h}_{\mathcal{A},\mathcal{\,B}}}(\eta)\,(\mbox{\textup{div} }\phi)(\eta)\;.\label{DPG}
\end{equation}
\item For $\mathbf{g}\in C_{0,\,0}(\mathcal{A},\,\mathcal{B})$ and $\psi\in\mathfrak{\mathfrak{S}}_{1+\varepsilon}(\mathcal{A},\,\mathcal{B})$,
we have that
\begin{equation}
\textup{cap}_{N}(\mathcal{A},\,\mathcal{B})\ge\frac{1}{||\Phi_{\mathbf{g}}-\psi||^{2}}\Bigl[1-2\varepsilon-2\sum_{\eta\in(\mathcal{A}\cup\mathcal{B})^{c}}\mathbf{\mathbf{h}_{\mathcal{A},\,\mathcal{B}}}(\eta)\,(\mbox{\textup{div} }\psi)(\eta)\Bigr]\;.\label{TPG}
\end{equation}
\end{enumerate}
\end{thm}

\begin{proof}
By (2) of Proposition \ref{pflow} and the fact that $\mathscr{L}_{N}\mathbf{\,h}_{\mathcal{A},\mathcal{\,B}}\equiv0$
on $(\mathcal{A}\cup\mathcal{B})^{c}$, we have
\[
\left\llangle \Psi_{\mathbf{\mathbf{h}_{\mathcal{A},\,\mathcal{B}}}},\,\Phi_{\mathbf{f}}\right\rrangle =\left\langle -\mathscr{L}_{N}\mathbf{\mathbf{\,h}_{\mathcal{A},\,\mathcal{B}}},\,\mathbf{f}\right\rangle _{\mu_{N}}=-\sum_{\eta\in\mathcal{A\cup B}}\mu_{N}(\eta)\,(\mathcal{\mathscr{L}}_{N}\mathbf{\,h}_{\mathcal{A},\mathcal{\,B}})(\eta)\mathbf{f}(\eta)\;.
\]
Hence, by \eqref{cap2},
\begin{equation}
\left\llangle \Psi_{\mathbf{h}_{\mathcal{A},\mathcal{\,B}}},\,\Phi_{\mathbf{f}}\right\rrangle =a\,\textup{cap}_{N}(\mathcal{A},\,\mathcal{B})\;\;\mbox{for all }\mathbf{f}\in C_{a,\,0}(\mathcal{A},\,\mathcal{B})\;.\label{ob1}
\end{equation}
Moreover, by (3) of Proposition \ref{pflow},
\begin{equation}
\left\llangle \Psi_{\mathbf{\mathbf{h}_{\mathcal{A},\mathcal{\,B}}}},\,\phi\right\rrangle =b+\sum_{\eta\in(\mathcal{A}\cup\mathcal{B})^{c}}\mathbf{\mathbf{h}_{\mathcal{A},\mathcal{\,B}}}(\eta)\,(\mbox{div }\phi)(\eta)\;\;\mbox{for all }\phi\in\mathfrak{\mathfrak{S}}_{b}(\mathcal{A},\mathcal{\,B})\;.\label{ob2}
\end{equation}

For part (1), let $\mathbf{f}\in C_{1,\,0}(\mathcal{A},\,\mathcal{B})$
and $\phi\in\mathfrak{\mathfrak{S}}_{\varepsilon}(\mathcal{A},\,\mathcal{B})$.
Then, by \eqref{ob1} and \eqref{ob2},
\begin{equation}
\left\llangle \Phi_{\mathbf{f}}-\phi,\,\Psi_{\mathbf{h}_{\mathcal{A},\mathcal{\,B}}}\right\rrangle =\textup{cap}_{N}(\mathcal{\mathcal{A}},\,\mathcal{B})-\Bigl[\varepsilon+\sum_{\eta\in(\mathcal{A}\cup\mathcal{B})^{c}}\mathbf{h}_{\mathcal{A},\,\mathcal{B}}(\eta)\,(\mbox{div }\phi)(\eta)\Bigr]\;,\label{e451}
\end{equation}
Furthermore, by the Cauchy-Schwarz inequality and (4) of Proposition
\ref{pflow},
\begin{equation}
\left\llangle \Phi_{\mathbf{f}}-\phi,\,\Psi_{\mathbf{h}_{\mathcal{A},\mathcal{\,B}}}\right\rrangle ^{2}\le\left\Vert \Phi_{\mathbf{f}}-\phi\right\Vert ^{2}\left\Vert \Psi_{\mathbf{h}_{\mathcal{A},\mathcal{\,B}}}\right\Vert ^{2}=\left\Vert \Phi_{\mathbf{f}}-\phi\right\Vert ^{2}\textup{cap}_{N}(\mathcal{\mathcal{A}},\,\mathcal{B})\;.\label{e452}
\end{equation}
By \eqref{e451} and \eqref{e452},
\[
\left\Vert \Phi_{\mathbf{f}}-\phi\right\Vert ^{2}\textup{cap}_{N}(\mathcal{\mathcal{A}},\,\mathcal{B})\ge\textup{cap}_{N}(\mathcal{\mathcal{A}},\,\mathcal{B})^{2}-2\textup{\,cap}_{N}(\mathcal{\mathcal{A}},\,\mathcal{B})\Bigl[\varepsilon+\sum_{\eta\in(\mathcal{A}\cup\mathcal{B})^{c}}\mathbf{h}_{\mathcal{A},\mathcal{\,B}}(\eta)\,(\mbox{div }\phi)(\eta)\Bigr]\;.
\]
Thus, part (1) is proved. The proof of part (2) is similar. For $\mathbf{g}\in C_{0,0}(\mathcal{A},\,\mathcal{B})$
and $\psi\in\mathfrak{\mathfrak{S}}_{1+\varepsilon}(\mathcal{A},\,\mathcal{B})$,
again by \eqref{ob1}, \eqref{ob2}, we have
\[
\left\llangle \Phi_{\mathbf{g}}-\psi,\,\Psi_{\mathbf{h}_{\mathcal{A},\,\mathcal{B}}}\right\rrangle =1-\varepsilon-\sum_{\eta\in(\mathcal{A}\cup\mathcal{B})^{c}}\mathbf{h}_{\mathcal{A},\mathcal{\,B}}(\eta)\,(\mbox{div }\psi)(\eta)\;.
\]
Hence, by computations as before, the proof of part (2) is completed.
\end{proof}
Based on this theorem, one can prove that $\textup{cap}_{N}(\mathcal{A},\,\mathcal{B})\simeq\,a_{N}$
for some sequence $(a_{N})_{N\in\mathbb{N}}$ as follows. The essential
part is to determine $\mathbf{f}\in C_{1,\,0}(\mathcal{A},\,\mathcal{B})$
and $\phi\in\mathfrak{\mathfrak{S}}_{\varepsilon_{N}}(\mathcal{A},\,\mathcal{B})$
where $\varepsilon_{N}\ll a_{N}$, such that
\begin{equation}
\sum_{\eta\in(\mathcal{A}\cup\mathcal{B})^{c}}\mathbf{\mathbf{h}_{\mathcal{A},\mathcal{\,B}}}(\eta)(\mbox{\textup{div} }\phi)(\eta)\ll a_{N}\;\;\text{and\;}\;||\Phi_{\mathbf{f}}-\phi||^{2}\simeq a_{N}\;.\label{dpg1}
\end{equation}
Then, by (1) of Theorem \ref{t04}, we obtain $\textup{cap}_{N}(\mathcal{A},\,\mathcal{B})\apprle a_{N}$.
In a similar manner, \eqref{TPG} is used to obtain $\textup{cap}_{N}(\mathcal{A},\,\mathcal{B})\apprge a_{N}$
and completes the estimate. Moreover, the first condition of \eqref{dpg1}
is valid if
\begin{equation}
\sum_{\eta\in(\mathcal{A}\cup\mathcal{B})^{c}}\left|(\mbox{\textup{div} }\phi)(\eta)\right|\ll a_{N}\;.\label{dpg2}
\end{equation}
This condition usually holds if $\mathcal{A}$ and $\mathcal{B}$
contain all the valleys. Otherwise, \eqref{dpg2} is not easy to verify,
and the summation in \eqref{dpg1} involving equilibrium potentials
should be handled directly. This can be achieved using a general argument
presented in Lemma \ref{lem85}.

\section{\label{sec6}Metastability of non-reversible zero-range processes}

In this section Theorem \ref{t02} is proved. Most arguments presented
here are not model-dependent; the special feature of zero-range dynamics
hardly plays a role. The model-dependent part, which is the construction
processes for approximating objects, is postponed to Sections \ref{sec7} and \ref{sec8}.

\subsection{\label{sec61}Brief review of the martingale approach to metastability}

A summary of the general results obtained in \cite{BL1, BL2, LLM}
regarding the metastability is first presented. In \cite{BL1, BL2},
Beltran and Landim demonstrated that, up to several technical estimates,
obtaining the sharp asymptotics for the so-called \textit{mean jump
rates} between metastable valleys is crucial and sufficient for describing
the metastable or ing behavior in terms of the convergence to the
Markov chain, after a suitable time rescaling. This approach is called
the martingale approach to metastability. The mode of convergence
for this original work is the soft topology developed in \cite{Lan1}.
Recently, Landim, Loulakis and Mourragui in \cite{LLM} showed that
the finite dimensional convergence can be proved by establishing an
additional estimate. For the present model, this estimate corresponds
to \eqref{H3} below. In this subsection, these results are briefly
summarized in terms of non-reversible zero-range processes.

The trace chain of the zero-range process $\eta_{N}(\cdot)$ on the
set $\mathcal{E}_{N}$ is first defined. For $t\ge0$, let
\[
T^{\mathcal{E}_{N}}(t)=\int_{0}^{t}\mathbf{1}\left\{ \eta_{N}(s)\in\mathcal{E}_{N}\right\} ds\;,
\]
which represents the amount of time for which the zero-range process
stays in one of the valleys up to time $t$. Let $S^{\mathcal{E}_{N}}(t)$
be the generalized inverse of $T^{\mathcal{E}_{N}}(t)$, i.e.,
\[
S^{\mathcal{E}_{N}}(t)=\sup\left\{ s\ge0:T^{\mathcal{E}_{N}}(s)\le t\right\} \;.
\]
The trace chain of $\eta_{N}(\cdot)$ on $\mathcal{E}_{N}$ is defined
by $\eta_{N}^{\mathcal{E}_{N}}(t)=\eta_{N}(S^{\mathcal{E}_{N}}(t))$,
$t\ge0$. Then, it is known that $\eta_{N}^{\mathcal{E}_{N}}(\cdot)$
is a Markov chain on $\mathcal{E}_{N}$ with stationary measure $\mu_{N}(\cdot)/\mu_{N}(\mathcal{E}_{N})$.
For two configurations $\eta,\,\zeta\in\mathcal{E}_{N}$, let $j_{N}(\eta,\,\zeta)$
be the jump rate between $\eta$ and $\zeta$ for the chain $\eta_{N}^{\mathcal{E}_{N}}(\cdot)$.
Finally, for $x,\,y\in S_{\star}$, the mean jump rate between two
valleys $\mathcal{E}_{N}^{x}$ and $\mathcal{E}_{N}^{y}$ is defined
by
\[
r_{N}(x,\,y)=\frac{1}{\mu_{N}(\mathcal{E}_{N}^{x})}\,\sum_{\eta\in\mathcal{E}_{N}^{x}}\,\sum_{\zeta\in\mathcal{E}_{N}^{y}}\mu_{N}(\eta)\,j_{N}(\eta,\,\zeta)\;.
\]

For each $x\in S_{\star}$, let $\xi_{N}^{x}\in\mathcal{H}_{N}$ be
the configuration such that all particles are concentrated at site
$x$. In addition, define
\begin{align*}
 & \breve{\mathcal{E}}_{N}^{x}=\mathcal{E}_{N}\setminus\mathcal{E}_{N}^{x}=\mathcal{E}_{N}(S_{\star}\setminus\{x\})\;\;;\;x\in S_{\star}\\
 & \breve{\mathcal{E}}_{N}^{x,\,y}=\mathcal{E}_{N}\setminus\left(\mathcal{E}_{N}^{x}\cup\mathcal{E}_{N}^{y}\right)=\mathcal{E}_{N}(S_{\star}\setminus\{x,\,y\})\;\;;\;x,\,y\in S_{\star}\;.
\end{align*}
The following theorem is proved in \cite[Theorem 2.1]{BL2} and in
\cite[Proposition 1.1]{LLM}.
\begin{thm}
Suppose that
\begin{align}
 & \lim_{N\rightarrow\infty}N^{1+\alpha}\,r_{N}(x,\,y)=a(x,\,y)\;\;\mbox{for all}\;x,\,y\in S_{\star}\;,\label{H0}\\
 & \lim_{N\rightarrow\infty}\sup_{\eta\in\mathcal{E}_{N}^{x},\,\eta\neq\xi_{N}^{x}}\frac{\textup{cap}_{N}(\mathcal{E}_{N}^{x},\,\breve{\mathcal{E}}_{N}^{x})}{\textup{cap}_{N}(\eta,\,\xi_{N}^{x})}=0\;\;\mbox{for all }x\in S_{\star}\;\text{,\,and}\label{H1}\\
 & \lim_{N\rightarrow\infty}\frac{\mu_{N}(\Delta_{N})}{\mu_{N}(\mathcal{E}_{N}^{x})}=0\;\;\mbox{for all }x\in S_{\star}\;.\label{H2}
\end{align}
Then, for all $x\in S_{\star}$ and for all sequences $(\eta_{N})_{N=1}^{\infty}$
such that $\eta_{N}\in\mathcal{E}_{N}^{x}$ for all $N$, the process
$W_{N}(\cdot)$ under $\mathbb{P}_{\eta_{N}}^{N}$ converges to $\mathbf{Q}_{x}$
with respect to the soft topology developed in \cite{Lan1}. In addition,
suppose that
\begin{equation}
\lim_{\delta\rightarrow0}\limsup_{N\rightarrow\infty}\max_{\eta\in\mathcal{E}_{N}^{x}}\sup_{2\delta\le s\le3\delta}\mathbb{P}_{\eta}^{N}\left[\eta_{N}(N^{1+\alpha}\,s)\in\Delta_{N}\right]=0\;\;\mbox{for all }x\in S_{\star}\;.\label{H3}
\end{equation}
Then, the finite dimensional distributions of the process $W_{N}(\cdot)$
under $\mathbb{P}_{\eta_{N}}^{N}$ converge to the those of $\mathbf{Q}_{x}$.
\end{thm}

It should be remarked that the conditions \eqref{H0}, \eqref{H1} and \eqref{H2}
are called (\textbf{H0}), (\textbf{H1}) and (\textbf{H2}), respectively,
in \cite{BL1, BL2}. By Theorem \ref{t01}, the condition \eqref{H2}
is immediate. The condition \eqref{H1} has been verified in \cite[Section 6]{BL3}
for \textit{reversible} zero-range process, namely,
\begin{equation}
\lim_{N\rightarrow\infty}\sup_{\eta\in\mathcal{E}_{N}^{x},\,\eta\neq\xi_{N}^{x}}\frac{\textup{cap}_{N}^{s}(\mathcal{E}_{N}^{x},\,\breve{\mathcal{E}}_{N}^{x})}{\textup{cap}_{N}^{s}(\eta,\,\xi_{N}^{x})}=0\;.\label{gg1}
\end{equation}
Hence, \eqref{H1} is an immediate consequence of this result and
the sector condition, i.e., Corollary \ref{csector}.

The condition \eqref{H3} will now be investigated. The proof is similar
to that in \cite[Example 4.2]{LLM}.
\begin{prop}
The condition \eqref{H3} holds.
\end{prop}

\begin{proof}
By \cite[Lemma 2.4]{LLM}, it suffices to verify that for all $x\in S_{\star}$,
we have
\begin{align}
 & \lim_{N\rightarrow\infty}\sup_{\eta\in\mathcal{E}_{N}^{x}}\mathbb{P}_{\eta}^{N}\left[\tau_{\xi_{N}^{x}}>N^{1+\alpha}\,\delta\right]=0\;\;\text{for all \ensuremath{\delta>0\;\text{and}}}\label{fd1}\\
 & \lim_{\delta\rightarrow0}\limsup_{N\rightarrow\infty}\sup_{\delta<t<3\delta}\mathbb{P}_{\xi_{N}^{x}}^{N}\left[\eta_{N}(N^{1+\alpha}\,t)\in\Delta_{N}\right]=0\;.\label{fd2}
\end{align}
For \eqref{fd1}, by the Markov inequality and \cite[Proposition 6.2]{BL2}
with $g\equiv1$, we have
\begin{equation}
\mathbb{P}_{\eta}^{N}\left[\tau_{\xi_{N}^{x}}>N^{1+\alpha}\,\delta\right]\le\frac{1}{N^{1+\alpha}\,\delta}\mathbb{E}_{\eta}^{N}\left[\tau_{\xi_{N}^{x}}\right]\le\frac{1}{N^{1+\alpha}\,\delta}\frac{1}{\textup{cap}_{N}(\eta,\,\xi_{N}^{x})}\;.\label{fd3}
\end{equation}
It follows from \cite[Theorem 2.2]{BL3} that
\begin{equation}
\textup{cap}_{N}^{s}(\mathcal{E}_{N}^{x},\,\breve{\mathcal{E}}_{N}^{x})\ge C\,N^{-(1+\alpha)}\;.\label{fd4}
\end{equation}
By \eqref{fd3}, \eqref{fd4} and  Corollary \ref{csector}, it can
be concluded that
\[
\mathbb{P}_{\eta}^{N}\left[\tau_{\xi_{N}^{x}}>N^{1+\alpha}\,\delta\right]\le\frac{C}{\delta}\frac{\textup{cap}_{N}^{s}(\mathcal{E}_{N}^{x},\,\breve{\mathcal{E}}_{N}^{x})}{\textup{cap}_{N}^{s}(\eta,\,\xi_{N}^{x})}\;.
\]
Hence, the estimate of \eqref{fd1} follows from \eqref{gg1}.

The second requirement \eqref{fd2} is now considered. Remark from
the definition of $\mu_{N}$ that we have $\mu_{N}(\xi_{N}^{x})=Z_{N}^{-1}$.
Hence, for $t>0$,
\[
\mathbb{P}_{\xi_{N}^{x}}^{N}\left[\eta_{N}(N^{1+\alpha}\,t)\in\Delta_{N}\right]\le\frac{\mathbb{P}_{\mu_{N}}^{N}\left[\eta_{N}(N^{1+\alpha}\,t)\in\Delta_{N}\right]}{\mu_{N}(\xi_{N}^{x})}=\frac{\mu_{N}(\Delta_{N})}{\mu_{N}(\xi_{N}^{x})}=Z_{N}\,\mu_{N}(\Delta_{N})\;.
\]
Hence, \eqref{fd2} is obtained by Proposition \ref{e21} and Theorem \ref{t01}.
\end{proof}
Therefore, the proof of Theorem \ref{t03} is reduced to the asymptotic
estimate \eqref{H0}, and this estimate is the core of the entire
problem. In particular, this estimate is reduced to the estimate of
the capacity between metastable valleys in the reversible case, owing
to the following identity for reversible Markov chains:
\begin{equation}
r_{N}(x,\,y)=\frac{1}{2}\left[\textup{cap}_{N}(\mathcal{E}_{N}^{x},\,\breve{\mathcal{E}}_{N}^{x})+\textup{cap}_{N}(\mathcal{E}_{N}^{y},\,\breve{\mathcal{E}}_{N}^{y})-\textup{cap}_{N}(\mathcal{E}_{N}^{x}\cup\mathcal{E}_{N}^{y},\,\breve{\mathcal{E}}_{N}^{x,\,y})\right]\;\;;\;x,\,y\in S_{\star}\;.\label{er1}
\end{equation}
Unfortunately, this relation is no longer valid in the non-reversible
case; hence, the estimation of $r_{N}(x,\,y)$ becomes a more delicate
task. The general strategy developed in \cite{LS1} can be summarized
as follows.
\begin{enumerate}
\item Let the mean holding rate be defined by
\[
\lambda_{N}(x)=\sum_{y\in S_{\star}\setminus\{x\}}r_{N}(x,\,y)\;.
\]
Then, the estimate of $\lambda_{N}(x)$ follows from the capacity
estimate. More precisely, it is known from \cite[display (A.8)]{BL2}
that
\begin{equation}
\lambda_{N}(x)=\frac{\textup{cap}_{N}(\mathcal{E}_{N}^{x},\,\breve{\mathcal{E}}_{N}^{x})}{\mu_{N}(\mathcal{E}_{N}^{x})}\;.\label{e610}
\end{equation}
Then, by Theorem \ref{t01}, it suffices to estimate $\textup{cap}_{N}(\mathcal{E}_{N}^{x},\,\breve{\mathcal{E}}_{N}^{x})$
to obtain the sharp asymptotics of $\lambda_{N}(x)$. This estimate
of the capacity between valleys is obtained in Corollary \ref{c02}
below.
\item The second step is to compute the sharp asymptotics of $r_{N}(x,\,y)/\lambda_{N}(x)$
using the so-called collapsed chain introduced in \cite{GL}. This
can be briefly explained as follows. Fix a point $x\in S_{\star}$.
Then, the collapsed chain is obtained from the original chain $\eta_{N}(\cdot)$
by carefully collapsing the set $\mathcal{E}_{N}^{x}$ into a single
point $\mathfrak{o}$. The precise definition and basic properties
of the collapsed chain are presented in Section \ref{sec64}. If $\overline{\mathbb{P}}_{\mathfrak{o}}^{N}$
denotes the law of this collapsed chain starting from $\mathfrak{o}$,
then it has been proven in \cite[Proposition 4.2]{BL2} that
\begin{equation}
\frac{r_{N}(x,\,y)}{\lambda_{N}(x)}=\overline{\mathbb{P}}_{\mathfrak{o}}^{N}\bigl[\tau_{\mathcal{E}_{N}^{y}}<\tau_{\breve{\mathcal{E}}_{N}^{x,\,y}}\bigr]\;.\label{e611}
\end{equation}
The right-hand side of the previous equality can be regarded as the
value of the equilibrium potential between $\mathcal{E}_{N}^{y}$
and $\breve{\mathcal{E}}_{N}^{x,\,y}$ at the collapsed state $\mathfrak{o}$,
with respect to the collapsed chain. The estimate of this value is
based on the capacity estimate for the collapsed chain, the sector
condition of the collapsed chain, and a careful investigation of the
relation between the original and the collapsed chain. This argument
is explained in detail in Sections \ref{sec64}, \ref{sec65}, and \ref{sec66}.
\end{enumerate}
The proof of Theorem \ref{t03} based on this strategy is also given
in Section \ref{sec66}.

\subsection{\label{sec62}Capacity estimates}

Herein, the main capacity estimates are provided. To this end, certain
potential-theoretic notations for the Markov chain $\widehat{Y}(\cdot)$
on $S_{\star}$ of Notation \ref{rem32} should be first introduced.
Notice here that the Markov chain $\widehat{Y}(\cdot)$ describes
the limiting metastable behavior of the present model.

\subsubsection*{Limiting Markov chain $\widehat{Y}(\cdot)$}

For $f:S_{\star}\rightarrow\mathbb{R}$, the generator of the Markov
chain $\widehat{Y}(\cdot)$ on $S_{\star}$ can be written as
\begin{equation}
(\mathfrak{L}_{Y}f)(x)=\sum_{y\in S_{\star}\setminus\{x\}}\,\frac{\textup{cap}_{X}(x,\,y)}{M_{\star}\,\Gamma(\alpha)\,I_{\alpha}}\left[f(y)-f(x)\right]\;\;;\;x\in S_{\star}\;.\label{hu3}
\end{equation}
The invariant measure for $\widehat{Y}(\cdot)$ is the uniform measure
$\mu(\cdot)$ on $S_{\star}$, i.e.,
\[
\mu(x)=1/\kappa_{\star}\;\;\text{for all }\;x\in S_{\star}\;.
\]
Thus, for $f:S_{\star}\rightarrow\mathbb{R}$, the Dirichlet form
can be written as
\[
\mathfrak{D}_{Y}(f)=\sum_{x\in S_{\star}}\mu(x)\,f(x)\left[-(\mathfrak{L}_{Y}\,f)(x)\right]=\frac{1}{2}\sum_{x,\,y\in S_{\star}}\frac{\textup{cap}_{X}(x,\,y)}{M_{\star}\,\Gamma(\alpha)\,I_{\alpha}\,\kappa_{\star}}\left[f(y)-f(x)\right]^{2}\;
\]
Let $\widehat{\mathbf{Q}}_{x}$ denote the law of chain $\widehat{Y}(\cdot)$
starting from $x\in S_{\star}$. Then, for two disjoint non-empty
sets $A,\,B\subseteq S_{\star}$, the equilibrium potential and capacity
between them with respect to the chain $\widehat{Y}(\cdot)$ are defined
by
\begin{equation}
\mathfrak{h}_{A,B}(x)=\mathbf{\widehat{Q}}_{x}(\tau_{A}<\tau_{B})\;\;\text{for}\;x\in S_{\star}\;\;\text{and\;\;}\textup{cap}_{Y}(A,\,B)=\mathfrak{D}_{Y}(\mathfrak{h}_{A,\,B})\;,\label{hu4}
\end{equation}
respectively.

\subsubsection*{Main capacity estimates}

The main capacity estimates are now stated.
\begin{thm}
\label{t02}For disjoint, non-empty subsets $A,\,B$ of $S_{\star}$,
we have that
\[
\lim_{N\rightarrow\infty}N^{1+\alpha}\,\textup{cap}_{N}(\mathcal{E}_{N}(A),\,\mathcal{E}_{N}(B))=\textup{cap}_{Y}(A,\,B)\;.
\]
\end{thm}

The proof of this result is given in the next subsection. In addition,
if $(A,\,B)$ is a partition of $S_{\star}$, i.e., $A\cup B=S_{\star}$,
the equilibrium potential $\mathfrak{h}_{A,\,B}$ for $\widehat{Y}(\cdot)$
becomes the indicator function on $A$; hence, the following corollary
is obtained.
\begin{cor}
\label{c02}Suppose that two disjoint, non-empty subsets $A,\,B$
of $S_{\star}$ satisfy $A\cup B=S_{\star}$. Then,
\[
\lim_{N\rightarrow\infty}N^{1+\alpha}\,\textup{cap}_{N}(\mathcal{E}_{N}(A),\,\mathcal{E}_{N}(B))=\frac{1}{M_{\star}\,\kappa_{\star}\,\Gamma(\alpha)\,I_{\alpha}}\,\sum_{x\in A,\,y\in B}\textup{cap}_{X}(x,\,y)\;.
\]
\end{cor}

\subsection{\label{sec63}Approximation of optimal flows and proof of Theorem
\ref{t02}}

Herein, the proof of Theorem \ref{t02} is provided based on the generalized
Dirichlet and Thomson principles of Theorem \ref{t04}. Several technical
details about the construction of approximations of equilibrium potentials
and optimal flows are postponed to Sections \ref{sec7} and \ref{sec8}.

Another parameter $\epsilon>0$ that denotes a sufficiently small
number is now introduced. In particular, we shall assume that $\epsilon\in(0,\,\epsilon_{0})$
where $\epsilon_{0}$ is a sufficiently small number to be introduced
in Lemma \ref{bg}.
\begin{notation}
Henceforth the constant term $C$ will be allowed to depend on this
new parameter $\epsilon$. Furthermore, $o_{N}(1)$ is used for representing
a term that vanishes as $N$ tends to $\infty$, and for $o_{\epsilon}(1)$
expressing a term that vanishes as $\epsilon$ tends to $0$. It should
be noted that the term $o_{N}(1)$ may depend on $\epsilon$, whereas
the term $o_{\epsilon}(1)$ does not depend on $N$. These dependencies
of the constant term $C$ and the error term $o_{N}(1)$ on the parameter
$\epsilon$ do not incur any technical problem, as we always take
$N\rightarrow\infty$ first and then $\epsilon\rightarrow0$.
\end{notation}

Throughout this subsection, let two disjoint non-empty subsets $A,\,B$
of $S_{\star}$ be fixed. In Section \ref{sec7}, for $\epsilon\in(0,\,\epsilon_{0})$
and sufficiently large $N\in\mathbb{N}$, two real-valued functions
$\mathbf{V}_{A,\,B}=\mathbf{V}_{A,\,B}^{N,\,\epsilon}$ and $\mathbf{V}_{A,\,B}^{*}=\mathbf{V}_{A,\,B}^{*,N,\,\epsilon}$
on $\mathcal{H}_{N}$ are constructed that approximate the equilibrium
potentials $\mathbf{h}_{\mathcal{E}_{N}(A),\mathcal{\,E}_{N}(B)}$
and $\mathbf{h}_{\mathcal{E}_{N}(A),\,\mathcal{E}_{N}(B)}^{*}$, respectively.
Furthermore, in Section \ref{sec75} the following properties of these
approximating objects are verified.
\begin{prop}
\label{p63}For $\epsilon\in(0,\,\epsilon_{0})$ and sufficiently
large $N\in\mathbb{N}$, there are two functions $\mathbf{V}_{A,\,B}$
and $\mathbf{V}_{A,\,B}^{*}$ satisfying the following properties:
\begin{enumerate}
\item For all $x\in S_{\star}$ and $\eta\in\mathcal{E}_{N}^{x}$, it holds
that $\mathbf{V}_{A,\,B}(\eta)=\mathbf{V}_{A,\,B}^{*}(\eta)=\mathfrak{h}_{A,\,B}(x)$.
That is, these functions are constant and equal to $\mathfrak{h}_{A,\,B}(x)$
on each valley $\mathcal{E}_{N}^{x}$, $x\in S_{\star}$.
\item It holds that
\[
N^{1+\alpha}\mathscr{D}_{N}(\mathbf{V}_{A,\,B})\le\left(1+o_{N}(1)+o_{\epsilon}(1)\right)\,\textup{cap}_{Y}(A,\,B)\;.
\]
The same inequality holds for $\mathscr{D}_{N}(\mathbf{V}_{A,B}^{*})$
as well.
\end{enumerate}
\end{prop}

The next step is to construct test flow to approximate $\Phi_{\mathbf{h}_{\mathcal{E}_{N}(A),\mathcal{\,E}_{N}(B)}}^{*}$
and $\Phi_{\mathbf{h}_{\mathcal{E}_{N}(A),\mathcal{\,E}_{N}(B)}^{*}}$.
Of course, the natural candidates are $\Phi_{\mathbf{V}_{A,\,B}}^{*}$
and $\Phi_{\mathbf{V}_{A,\,B}^{*}}$. One may expect that Theorem
\ref{t04} may be used to estimate the capacity based on these objects.
Unfortunately, a technical issue arises around the saddle tube defined
in Section \ref{sec8}, at which the divergence of these flows is
not negligible. This problem is resolved by a systematic correction
procedure developed in Section \ref{sec8}, which cleans out the non-negligible
flow and in turn allows the application of Theorem \ref{t04}. The
consequences of this correction procedure can be summarized as follows.
\begin{prop}
\label{p64}For $\epsilon\in(0,\,\epsilon_{0})$ and sufficiently
large $N\in\mathbb{N}$, there exists a flow $\Phi_{A,\,B}=\Phi_{A,\,B}^{N,\,\epsilon}\in\mathfrak{F}_{N}$
satisfying the following properties.
\begin{enumerate}
\item The flow $\Phi_{A,\,B}$ approximates $\Phi_{\mathbf{V}_{A,\,B}}^{*}$
in the sense that
\[
\left\Vert \Phi_{A,\,B}-\Phi_{\mathbf{V}_{A,\,B}}^{*}\right\Vert ^{2}=\left(o_{N}(1)+o_{\epsilon}(1)\right)N^{-(1+\alpha)}\;.
\]
\item The divergence of $\Phi_{A,\,B}$ is negligible on $\Delta_{N}$ in
the sense that
\[
\sum_{\eta\in\Delta_{N}}\left|(\textup{div }\Phi_{A,\,B})(\eta)\right|=o_{N}(1)\,N^{-(1+\alpha)}\;.
\]
\item The divergence of $\Phi_{A,B}$ is negligible on $\mathcal{E}_{N}^{x}$,
$x\in S_{\star}\setminus(A\cup B)$, in the sense that
\begin{align*}
 & (\textup{div }\Phi_{A,\,B})(\mathcal{E}_{N}^{x})=o_{N}(1)\,N^{-(1+\alpha)}\;\;\text{and}\\
 & \sum_{\eta\in\mathcal{E}_{N}^{x}}\mathbf{\mathbf{h}}_{\mathcal{E}_{N}(A),\,\mathcal{E}_{N}(B)}(\eta)\,(\textup{div }\Phi_{A,\,B})(\eta)=o_{N}(1)\,N^{-(1+\alpha)}\;.
\end{align*}
\item The divergence of $\Phi_{A,B}$ satisfies
\begin{align*}
 & (\textup{div }\Phi_{A,B})(\mathcal{E}_{N}(A))=\left(1+o_{N}(1)\right)N^{-(1+\alpha)}\,\textup{cap}_{Y}(A,\,B)\;\;\mbox{and}\\
 & (\textup{div }\Phi_{A,B})(\mathcal{E}_{N}(B))=-\left(1+o_{N}(1)\right)N^{-(1+\alpha)}\,\textup{cap}_{Y}(A,\,B)\;.
\end{align*}
\end{enumerate}
There also exists $\Phi_{A,\,B}^{*}=\Phi_{A,\,B}^{*,\,N,\,\epsilon}\in\mathfrak{F}_{N}$
that approximates $\Phi_{\mathbf{V}_{A,\,B}^{*}}$ and satisfies the
four properties above.
\end{prop}

In particular, by (2) and (3) of the previous proposition, we have
the following estimate that enables the application of the generalized
Dirichlet and Thomson principles.
\begin{lem}
\label{lem66}We have that
\[
\sum_{\eta\in(\mathcal{E}_{N}(A\cup B))^{c}}\mathbf{\mathbf{h}}_{\mathcal{E}_{N}(A),\,\mathcal{E}(B)}(\eta)\,(\textup{div }\Phi_{A,\,B})(\eta)=o_{N}(1)\,N^{-(1+\alpha)}\;.
\]
\end{lem}

\begin{proof}
The summation on the left-hand side can be divided as
\[
\sum_{\eta\in\Delta_{N}}+\sum_{x\notin A\cup B}\,\,\sum_{\eta\in\mathcal{E}_{N}^{x}}\;.
\]
As $|\mathbf{\mathbf{h}}_{\mathcal{E}_{N}(A),\,\mathcal{E}(B)}|\le1$,
the absolute value of the first summation is $o_{N}(1)\,N^{-(1+\alpha)}$
by (2) of Proposition \ref{p64}. The second summation is $o_{N}(1)\,N^{-(1+\alpha)}$
by the second estimate of (3) of Proposition \ref{p64}.
\end{proof}
Theorem \ref{t02} may now be proved.
\begin{proof}[Proof of Theorem \ref{t02}]
The upper bound of the capacity is first considered. Let \begin{equation} \label{ec1} \begin{aligned}
&\mathbf{f}=\frac{\mathbf{V}_{A,\,B}+\mathbf{V}_{A,\,B}^{*}}{2}\in C_{1,\,0}(\mathcal{E}_{N}(A),\,\mathcal{\mathcal{E}}_{N}(B))\;\;\mbox{and}\\
&\phi=\frac{\Phi_{A,\,B}^{*}-\Phi_{A,\,B}}{2}\in\mathfrak{\mathfrak{S}}_{\alpha_{N}}(\mathcal{E}_{N}(A),\,\mathcal{\mathcal{E}}_{N}(B))\;\;\text{for some }\alpha_{N}=o_{N}(1)\,N^{-(1+\alpha)}\;.
\end{aligned} \end{equation} Note that $\alpha=o_{N}(1)N^{-(1+\alpha)}$ follows from part (4)
(for $\Phi_{A,\,B}$ and $\Phi_{A,\,B}^{*}$) of Proposition \ref{p64}.
Then, by part (1) of Theorem \ref{t04} and Lemma \ref{lem66}, we
have
\begin{equation}
\textup{cap}_{N}(\mathcal{E}_{N}(A),\,\mathcal{E}_{N}(B))\le\left\Vert \Phi_{\mathbf{f}}-\phi\right\Vert ^{2}+o_{N}(1)\,N^{-(1+\alpha)}.\label{e630}
\end{equation}
Let
\begin{equation}
\Phi_{A,\,B}=\Phi_{\mathbf{V}_{A,\,B}}^{*}+\Theta_{N}\;\;\mbox{and}\;\;\Phi_{A,\,B}^{*}=\Phi_{\mathbf{V}_{A,\,B}^{*}}+\Theta_{N}^{*}\;.\label{ec2}
\end{equation}
Then, we have
\begin{equation}
\Phi_{\mathbf{f}}-\phi=\Phi_{(\mathbf{V}_{A,\,B}+\mathbf{V}_{A,\,B}^{*})/2}-\frac{\Phi_{\mathbf{V}_{A,\,B}^{*}}-\Phi_{\mathbf{V}_{A,\,B}}^{*}}{2}+\frac{\Theta_{N}-\Theta_{N}^{*}}{2}=\Psi_{\mathbf{V}_{A,\,B}}+\frac{\Theta_{N}-\Theta_{N}^{*}}{2}\;.\label{e631}
\end{equation}
By (1) and (2) of Proposition \ref{p63}, \begin{equation} \label{e632} \begin{aligned}
&\left\Vert \Psi_{\mathbf{V}_{A,\,B}}\right\Vert ^{2}=\mathscr{D}_{N}(\mathbf{V}_{A,B})\le\left(1+o_{N}(1)+o_{\epsilon}(1)\right)N^{-(1+\alpha)}\,\textup{cap}_{Y}(A,\,B)\;,\\&\left\Vert \frac{\Theta_{N}-\Theta_{N}^{*}}{2}\right\Vert ^{2}=\left(o_{N}(1)+o_{\epsilon}(1)\right)N^{-(1+\alpha)}\;.\end{aligned} \end{equation}Therefore, by \eqref{e631},   \eqref{e632} and the triangle inequality,
\begin{equation}
\left\Vert \Phi_{\mathbf{f}}-\phi\right\Vert ^{2}\le\left(1+o_{N}(1)+o_{\epsilon}(1)\right)N^{-(1+\alpha)}\,\textup{cap}_{Y}(A,\,B)\;.\label{e64up}
\end{equation}
By \eqref{e630} and \eqref{e64up}, we obtain the upper bound
\begin{equation}
\textup{cap}_{N}(\mathcal{E}_{N}(A),\,\mathcal{E}_{N}(B))\le\left(1+o_{N}(1)+o_{\epsilon}(1)\right)N^{-(1+\alpha)}\,\textup{cap}_{Y}(A,\,B)\;.\label{e64u}
\end{equation}

To obtain the lower bound, part (2) of Theorem \ref{t04} is used.
To this end, let\begin{equation} \label{ec3} \begin{aligned}
&\mathbf{g}=\frac{\mathbf{V}_{A,\,B}^{*}-\mathbf{V}_{A,\,B}}{2\,N^{-(1+\alpha)}\,\textup{cap}_{Y}(A,\,B)}\in C_{0,\,0}(\mathcal{E}_{N}(A),\,\mathcal{\mathcal{E}}_{N}(B))\;,\;\mbox{and}\\
&\psi=\frac{\Phi_{A,\,B}^{*}+\Phi_{A,\,B}}{2\,N^{-(1+\alpha)}\,\textup{cap}_{Y}(A,\,B)}\in\mathfrak{\mathfrak{S}}_{1+o_{N}(1)}(\mathcal{E}_{N}(A),\,\mathcal{\mathcal{E}}_{N}(B))\;.
\end{aligned} \end{equation}Then, by Theorem \ref{t04} and Lemma \ref{lem66}, we have
\begin{equation}
\textup{cap}_{N}(\mathcal{E}_{N}(A),\,\mathcal{E}_{N}(B))\ge\frac{1}{\left\Vert \Phi_{\mathbf{g}}-\psi\right\Vert ^{2}}\left(1+o_{N}(1)+o_{\epsilon}(1)\right)\;.\label{e644}
\end{equation}
As we can write
\[
\Phi_{\mathbf{g}}-\psi=-\frac{1}{N^{-(1+\alpha)}\,\textup{cap}_{Y}(A,\,B)}\left[\Psi_{\mathbf{V}_{A,\,B}}-\frac{\Theta_{N}+\Theta_{N}^{*}}{2}\right]\;,
\]
by similar computations as in the upper bound, we obtain
\begin{equation}
||\Phi_{\mathbf{g}}-\psi||^{2}\le\left(1+o_{N}(1)+o_{\epsilon}(1)\right)\frac{1}{N^{-(1+\alpha)}\,\textup{cap}_{Y}(A,\,B)}\;.\label{e645}
\end{equation}
Combining \eqref{e644} and \eqref{e645}, we have
\begin{equation}
\textup{cap}_{N}(\mathcal{E}_{N}(A),\,\mathcal{E}_{N}(B))\ge\left(1+o_{N}(1)+o_{\epsilon}(1)\right)N^{-(1+\alpha)}\,\textup{cap}_{Y}(A,\,B)\;.\label{e64lo}
\end{equation}
By the upper bound \eqref{e64u} and the lower bound \eqref{e64lo},
the proof is completed.
\end{proof}
By a careful reading of the previous proof, the estimate obtained
in Proposition \ref{p63} can be strengthened as follows.
\begin{cor}
\label{cor66}We have that
\[
\mathscr{D}_{N}(\mathbf{V}_{A,\,B})=\left(1+o_{N}(1)+o_{\epsilon}(1)\right)N^{-(1+\alpha)}\,\textup{cap}_{Y}(A,\,B)\;.
\]
\end{cor}

\subsection{\label{sec64}Collapsed chain}

The importance of the collapsed chain in the context of metastability
has been noticed in \cite{GL, BL2, LS1}. The collapsed chain can
be regarded as a special case of lumped Markov chain (cf. \cite[Section 9.3]{BdH}).
In \cite{GL} the collapsed chain was used for establishing the Dirichlet
principle for non-reversible Markov chains on countable spaces. In
\cite{BL2}, the relation \eqref{e611} between the mean jump rate
and the collapsed chain was obtained, opening up the possibility of
rigorous investigation of metastability of non-reversible processes.
In \cite[Section 8]{LS1}, a method of estimating the right-hand side
of \eqref{e611} was obtained and applied to a cyclic random walk
in a potential field. In this method, the construction of divergence-free
flows was assumed, which is not true in the present case. Accordingly,
in this study the method is properly modified to obtain the sharp
asymptotics of the mean jump rate. To explain this process, certain
well known results on the collapsed chains are presented in this section,
in the context of the zero-range processes. All the proofs are elementary
and given in \cite[Section 8.2]{LS1}.

\subsubsection*{Definition of collapsed chain}

Let $x\in S_{\star}$ and $\overline{\mathcal{H}}_{N}=(\mathcal{H}_{N}\setminus\mathcal{E}_{N}^{x})\cup\{\mathfrak{o}\}$,
where $\mathfrak{o}$ is a new single point. We can regard $\overline{\mathcal{H}}_{N}$
as the set obtained from $\mathcal{H}_{N}$ by collapsing the set
$\mathcal{E}_{N}^{x}$ into a single point $\mathfrak{o}$. Let $R_{N}(\cdot,\,\cdot)$
be the jump rate of the chain $\eta_{N}(\cdot)$, i.e.,
\[
R_{N}(\eta,\,\zeta)=\begin{cases}
g(\eta_{x})r(x,\,y) & \mbox{if }\zeta=\sigma^{x,\,y}\eta\mbox{ for some }x,\,y\in S,\\
0 & \mbox{otherwise\;.}
\end{cases}
\]
The corresponding jump rate on $\overline{\mathcal{H}}_{N}$ is defined
by
\begin{align*}
 & \overline{R}_{N}(\eta,\,\zeta)=R_{N}(\eta,\,\zeta)\;\;\mbox{for all\;}\eta,\,\zeta\in\mathcal{H}_{N}\setminus\mathcal{E}_{N}^{x}\;,\\
 & \overline{R}_{N}(\eta,\,\mathfrak{o})=\sum_{\zeta\in\mathcal{E}_{N}^{x}}R_{N}(\eta,\,\zeta)\;\;\mbox{for all\;}\eta\in\mathcal{H}_{N}\setminus\mathcal{E}_{N}^{x}\;,\;\text{and}\\
 & \overline{R}_{N}(\mathfrak{o},\,\zeta)=\frac{1}{\mu_{N}(\mathcal{E}_{N}^{x})}\sum_{\eta\in\mathcal{E}_{N}^{x}}\mu_{N}(\eta)R_{N}(\eta,\,\zeta)\;\;\mbox{for all\;}\zeta\in\mathcal{H}_{N}\setminus\mathcal{E}_{N}^{x}\;.
\end{align*}
Then, the collapsed chain $\{\overline{\eta}_{N}(t):t\ge0\}$ is defined
as the Markov chain on $\overline{\mathcal{H}}_{N}$ whose jump rate
is $\overline{R}_{N}(\cdot,\,\cdot)$. Let $\mathscr{\overline{L}}_{N}$
denote the generator of the collapsed chain $\overline{\eta}_{N}(\cdot)$,
and let $\mathscr{\overline{L}}_{N}^{\,*}$ and $\overline{\mathscr{L}}_{N}^{\,s}$
denote the generators of the adjoint chain and the symmetrized chain
of the collapsed chain, respectively. Let $\overline{\mathscr{D}}_{N}(\cdot)$
be the Dirichlet form associated with these generators. Denote by
$\overline{\mathbb{P}}_{\eta}^{N}$, $\eta\in\overline{\mathcal{H}}_{N}$,
the law of chain $\overline{\eta}_{N}(\cdot)$ starting from $\eta$.
\begin{lem}
The Markov chain $\overline{\eta}_{N}(\cdot)$ is irreducible on $\overline{\mathcal{H}}_{N}$,
and the unique invariant measure $\overline{\mu}_{N}(\cdot)$ is given
by
\[
\overline{\mu}_{N}(\eta)=\mu_{N}(\eta)\mbox{\;\;for\;\;}\eta\in\mathcal{H}_{N}\setminus\mathcal{E}_{N}^{x}\mbox{\;,\;\;and\;\;}\overline{\mu}_{N}(\mathfrak{o})=\mu_{N}(\mathcal{E}_{N}^{x})\;.
\]
\end{lem}

The proof is based on elementary computations and is left to the reader.

\subsubsection*{Flow structure of collapsed chains}

As $\overline{\eta}_{N}(\cdot)$ is a Markov chain on $\overline{\mathcal{H}}_{N}$,
the flow structure can be induced, and then potential theory can be
developed in the same manner as in Section \ref{sec51}. The flow
structure of the Markov chain $\overline{\eta}_{N}(\cdot)$ as well
as its relation to that of the original chain are summarized now.

The conductance between $\eta,\,\zeta\in\overline{\mathcal{H}}_{N}$
of the collapsed chain is defined by
\[
\overline{c}_{N}(\eta,\,\zeta)=\overline{\mu}_{N}(\eta)\,\overline{R}_{N}(\eta,\,\zeta)\;.
\]
It can be verified that $\overline{c}_{N}(\eta,\,\zeta)=c_{N}(\eta,\,\zeta)$
if $\eta,\,\zeta\neq\mathfrak{o}$, and that
\[
\overline{c}_{N}(\eta,\,\mathfrak{o})=\sum_{\zeta\in\mathcal{E}_{N}^{x}}c_{N}(\eta,\,\zeta)\;\;\mbox{and\;\;}\overline{c}_{N}(\mathfrak{o},\,\zeta)=\sum_{\eta\in\mathcal{E}_{N}^{x}}c_{N}(\eta,\,\zeta)\;.
\]
This linearity is fundamental in the relation between the original
and collapsed chain. Define the symmetrized conductance by
\[
\overline{c}_{N}^{\,s}(\eta,\,\zeta)=(1/2)\,(\,\overline{c}_{N}(\eta,\,\zeta)+\overline{c}_{N}(\zeta,\,\eta))\;\;;\;\eta,\,\zeta\in\overline{\mathcal{H}}_{N}\;,
\]
and the edge set by
\[
\overline{\mathcal{H}}_{N}^{\otimes}=\left\{ (\eta,\,\zeta)\in\overline{\mathcal{H}}_{N}\times\overline{\mathcal{H}}_{N}:\overline{c}_{N}^{\,s}(\eta,\,\zeta)>0\right\} \;.
\]
Then, flows are defined by an anti-symmetric real-valued function
on $\overline{\mathcal{H}}_{N}^{\otimes}$, and the set of flows is
denoted by $\overline{\mathfrak{F}}_{N}$. The inner product and the
norm on the flow structure can be defined in the same manner as before
and are denoted by $\llangle\cdot,\,\cdot\rrangle_{\mathcal{C}}$
and $||\cdot||_{\mathcal{C}}$, respectively. The divergence of a
flow is also defined similarly.

\subsubsection*{Collapsed objects: flows, functions, equilibrium potential, and capacity}

For $\phi\in\mathfrak{F}_{N}$, the collapsed flow $\overline{\phi}\in\overline{\mathfrak{F}}_{N}$
is defined by
\begin{align}
 & \overline{\phi}(\eta,\,\zeta)=\phi(\eta,\,\zeta)\;\;\text{\ensuremath{\forall}\ensuremath{\eta,\,\zeta\neq\mathfrak{o}\;}},\;\;\overline{\phi}(\eta,\,\mathfrak{o})=\sum_{\zeta\in\mathcal{E}_{N}^{x}}\phi(\eta,\,\zeta)\;,\text{and}\;\;\overline{\phi}(\mathfrak{o},\,\zeta)=\sum_{\eta\in\mathcal{E}_{N}^{x}}\phi(\eta,\,\zeta)\;.\label{eff1}
\end{align}
Then, the following results are known.
\begin{lem}
\label{lem67}For $\phi\in\mathfrak{F}_{N}$, the flow norm of $\overline{\phi}$
satisfies $||\overline{\phi}||_{\mathcal{C}}\le||\phi||$, and the
equality holds if and only if,
\begin{align}
 & \phi(\eta,\,\zeta)=0\mbox{\;\;for all\;}\;(\eta,\,\zeta)\in\mathcal{H}_{N}^{\otimes}\;\mbox{such that }\mbox{\ensuremath{\eta},\,\ensuremath{\zeta}\ensuremath{\in}}\mathcal{E}_{N}^{x}\;\mbox{and}\label{hu1}\\
 & \frac{\phi(\eta,\,\zeta)}{c_{N}^{s}(\eta,\,\zeta)}=\frac{\phi(\eta,\,\zeta')}{c_{N}^{s}(\eta,\,\zeta')}\mbox{ \;for all}\;(\eta,\,\zeta),\,(\eta,\,\zeta')\in\mathcal{H}_{N}^{\otimes}\;\mbox{such that }\zeta,\,\zeta'\in\mathcal{E}_{N}^{x}\;.\label{hu2}
\end{align}
\end{lem}

\begin{proof}
See \cite[Lemma 8.2]{LS1}.
\end{proof}
\begin{lem}
\label{lem68}For $\phi\in\mathfrak{F}_{N}$, the divergence of $\overline{\phi}$
satisfies
\[
(\textup{div }\overline{\phi})(\eta)=\begin{cases}
(\textup{div }\phi)(\eta) & \mbox{if }\eta\neq\mathfrak{o}\;,\\
(\textup{div }\phi)(\mathcal{E}_{N}^{x}) & \mbox{if }\eta=\mathfrak{o\;.}
\end{cases}
\]
\end{lem}

\begin{proof}
See \cite[display (8.7)]{LS1}.
\end{proof}
The collapse of functions is now considered. Suppose that a function
$\mathbf{f}:\mathcal{H}_{N}\rightarrow\mathbb{R}$ satisfies $\mathbf{f}(\eta)=a$
for all $\eta\in\mathcal{E}_{N}^{x}$, for some $a\in\mathbb{R}$.
Then, the collapsed function $\overline{\mathbf{f}}:\overline{\mathcal{H}}_{N}\rightarrow\mathbb{R}$
is defined by
\[
\overline{\mathbf{f}}(\eta)=\begin{cases}
\mathbf{f}(\eta) & \mbox{if }\eta\neq\mathfrak{o}\;,\\
\mathbf{f}(a) & \mbox{if }\eta=\mathfrak{o\;.}
\end{cases}
\]
As in \eqref{flow}, for $\mathbf{g}:\overline{\mathcal{H}}_{N}\rightarrow\mathbb{R}$,
three flows can be defined by
\begin{align}
\overline{\Phi}_{\mathbf{g}}(\eta,\,\zeta) & =\mathbf{g}(\eta)\,\overline{c}_{N}(\eta,\,\zeta)-\mathbf{g}(\zeta)\,\overline{c}_{N}(\zeta,\,\eta)\;,\nonumber \\
\overline{\Phi}_{\mathbf{g}}^{\,*}(\eta,\,\zeta) & =\mathbf{g}(\eta)\,\overline{c}_{N}(\zeta,\,\eta)-\mathbf{g}(\zeta)\,\overline{c}_{N}(\eta,\,\zeta)\;,\label{cflow}\\
\overline{\Psi}_{\mathbf{g}}(\eta,\,\zeta) & =\overline{c}_{N}^{\,s}(\eta,\,\zeta)\left[\mathbf{g}(\eta)-\mathbf{g}(\zeta)\right]\;.\nonumber
\end{align}

\begin{lem}
\label{lem69}Suppose that a function $\mathbf{f}:\mathcal{H}_{N}\rightarrow\mathbb{R}$
is constant on $\mathcal{E}_{N}^{x}$ so that the collapsed function
$\overline{\mathbf{f}}$ can be defined. Then,
\[
\overline{\Phi_{\mathbf{f}}}=\overline{\Phi}_{\overline{\mathbf{f}}}\;,\;\;\overline{\Phi_{\mathbf{f}}^{*}}=\overline{\Phi}_{\overline{\mathbf{f}}}^{\,*}\;,\;\;\mbox{and\;\;}\overline{\Psi_{\mathbf{f}}}=\overline{\Psi}_{\overline{\mathbf{f}}}\;,
\]
where $\overline{\Phi_{\mathbf{f}}}$, $\overline{\Phi_{\mathbf{f}}^{*}}$
and $\overline{\Psi_{\mathbf{f}}}$ represent the collapsed flows
of $\Phi_{\mathbf{f}}$, $\Phi_{\mathbf{f}}^{*}$, and $\Psi_{\mathbf{f}}$,
in the sense of \eqref{eff1}, respectively.
\end{lem}

\begin{proof}
See \cite[Lemma 8.3]{LS1}.
\end{proof}
For two disjoint non-empty subsets $\mathcal{A}$ and $\mathcal{B}$
of $\overline{\mathcal{H}}_{N}$, the equilibrium potential $\overline{\mathbf{h}}_{\mathcal{A},\mathcal{\,B}}$
and the capacity $\overline{\textup{cap}}_{N}(\mathcal{A},\,\mathcal{B})$
for the collapsed chain $\overline{\eta}_{N}(\cdot)$ can be defined.
Moreover, $\overline{C}_{a,\,b}(\mathcal{A},\,\mathcal{B})$ and $\mathfrak{\overline{S}}_{a}(\mathcal{A},\,\mathcal{B})$
are defined as the natural collapsed versions of $C_{a,\,b}(\mathcal{A},\,\mathcal{B})$
and $\mathfrak{\mathfrak{S}}_{a}(\mathcal{A},\,\mathcal{B})$ introduced
in Section \ref{sec52}, respectively.
\begin{rem}
If $\mathcal{A},\,\mathcal{B}\subset\mathcal{H}_{N}\setminus\mathcal{E}_{N}^{x}$,
then the equilibrium potential $\mathbf{h}_{\mathcal{A},\,\mathcal{B}}$
with respect to the original dynamics can be considered. In this case,
$\mathbf{h}_{\mathcal{A},\,\mathcal{B}}$ conditioned on $\mathcal{E}_{N}^{x}$
may not be a constant function; hence, the collapsed function of $\mathbf{h}_{\mathcal{A},\,\mathcal{B}}$
may not be defined. Thus, we should not regard $\overline{\mathbf{h}}_{\mathcal{A},\mathcal{\,B}}$
as a collapse of function $\mathbf{h}_{\mathcal{A},\,\mathcal{B}}$
in the sense explained above.
\end{rem}

Let $\overline{\textup{cap}}_{N}^{\,s}(\cdot,\,\mathcal{\cdot})$
denote the capacity corresponding to the dynamics associated with
the generator $\overline{\mathscr{L}}_{N}^{\,s}$. It is known that
the sector condition of the collapsed chain is inherited from the
original chain; therefore, Corollary \ref{csector} is valid for the
collapsed dynamics as well.
\begin{lem}
\label{lem612}For two disjoint non-empty subsets $\mathcal{A},\,\mathcal{B}$
of $\overline{\mathcal{H}}_{N}$, we have
\[
\overline{\textup{cap}}_{N}^{\,s}(\mathcal{A},\,\mathcal{B})\le\textup{\ensuremath{\overline{\textup{cap}}}}_{N}(\mathcal{A},\,\mathcal{B})\le C_{0}\,\textup{\ensuremath{\overline{\textup{cap}}}}_{N}^{\,s}(\mathcal{A},\,\mathcal{B})\;.
\]
\end{lem}

\begin{proof}
See \cite[Lemma 8.6]{LS1}.
\end{proof}

\subsection{\label{sec65}Capacity estimates for collapsed chains}

The capacity estimates for the collapsed chain are an essential ingredient
for estimating the mean jump rate. It should be first noted that the
capacity for the collapsed process is easy to compute when one of
the sets involved is $\{\mathfrak{o}\}$.
\begin{lem}
\label{lem613}For all non-empty subsets $\mathcal{A}$ of $\mathcal{H}_{N}\setminus\mathcal{E}_{N}^{x}$,
\[
\textup{\ensuremath{\overline{\textup{cap}}}}_{N}(\mathcal{A},\,\{\mathfrak{o}\})=\textup{cap}_{N}(\mathcal{A},\,\mathcal{E}_{N}^{x})\;.
\]
\end{lem}

\begin{proof}
See \cite[display (3.10)]{GL} or \cite[Theorem 9.7]{BdH}.
\end{proof}
A similar result holds for $\textup{\ensuremath{\overline{\textup{cap}}}}_{N}(\mathcal{A},\,\mathcal{B})$
as well when either $\mathcal{A}$ or $\mathcal{B}$ contains $\mathfrak{o}$.
However, to the best of the author's knowledge, there is no trivial
method for comparing $\textup{cap}_{N}(\mathcal{A},\,\mathcal{B})$
and $\textup{\ensuremath{\overline{\textup{cap}}}}_{N}(\mathcal{A},\,\mathcal{B})$
when $\mathfrak{o}\notin\mathcal{A}\cup\mathcal{B}$. In view of this
remark, the following estimate is a non-trivial result.
\begin{prop}
\label{p613}For two disjoint and non-empty subsets $A$ and $B$
of $S_{\star}\setminus\{x\}$ that satisfy $A\cup B=S_{\star}\setminus\{x\}$,
it holds that
\[
\textup{\ensuremath{\overline{\textup{cap}}}}_{N}(\mathcal{E}_{N}(A),\,\mathcal{E}_{N}(B))=\left(1+o_{N}(1)+o_{\epsilon}(1)\right)\,N^{-(1+\alpha)}\,\textup{cap}_{Y}(A,\,B)\;.
\]
\end{prop}

\begin{rem}
In fact, the condition $A\cup B=S_{\star}\setminus\{x\}$ in the previous
proposition is redundant. However, the more general result without
this condition is not required in the present study. Moreover, its
proof is more complicated and is thus omitted here.
\end{rem}

Proposition \ref{p613} is proved at the end of the this subsection.
Throughout this subsection, we fix two sets $A,\,B$ satisfying the
condition of Proposition \ref{p613}. Recall the functions $\mathbf{V}_{A,\,B}$
and $\mathbf{V}_{A,\,B}^{*}$ from Proposition \ref{p63} and the
flows $\Phi_{A,\,B}$ and $\Phi_{A,\,B}^{*}$ from Proposition \ref{p64}.
It should be noted that $\mathbf{V}_{A,\,B}(\eta)=\mathfrak{h}_{A,\,B}(x)$
for all $\eta\in\mathcal{E}_{N}^{x}$; thus the collapsed function
$\overline{\mathbf{V}}_{A,\,B}:=\overline{\mathbf{V}_{A,\,B}}$, satisfying
$\overline{\mathbf{V}}_{A,\,B}(\mathfrak{o})=\mathfrak{h}_{A,\,B}(x)$,
can be defined.
\begin{lem}
\label{lem615} It holds that
\[
\bigl\Vert\overline{\Psi}_{\,\overline{\mathbf{V}}_{A,\,B}}\bigr\Vert_{\mathcal{C}}^{2}=\left(1+o_{N}(1)+o_{\epsilon}(1)\right)N^{-(1+\alpha)}\textup{\,cap}_{Y}(A,\,B)\;.
\]
\end{lem}

\begin{proof}
By Lemma \ref{lem69} and Lemma \ref{lem67}, we obtain

\[
\bigl\Vert\overline{\Psi}_{\,\overline{\mathbf{V}}_{A,\,B}}\bigr\Vert_{\mathcal{C}}^{2}=\bigl\Vert\overline{\Psi_{\mathbf{V}_{A,\,B}}}\bigr\Vert_{\mathcal{C}}^{2}=\bigl\Vert\Psi_{\mathbf{V}_{A,\,B}}\bigr\Vert^{2}\;,
\]
where the second equality follows since by the elementary calculation
we are able to check that the flow $\Psi_{\mathbf{V}_{A,\,B}}$ satisfies
the equality conditions \eqref{hu1} and \eqref{hu2} of Lemma \ref{lem67}.
It now suffices to invoke Corollary \ref{cor66} to finish the proof.

Let $\overline{\Phi}_{A,\,B}:=\overline{\Phi_{A,\,B}}$ be the collapsed
flow of $\Phi_{A,\,B}$.
\end{proof}
\begin{lem}
\label{lem617}For two disjoint and non-empty subsets $A$ and $B$
of $S_{\star}\setminus\{x\}$ that satisfy $A\cup B=S_{\star}\setminus\{x\}$,
it holds that
\[
\sum_{\eta\in\mathcal{\overline{H}}_{N}\setminus\mathcal{E}_{N}(A\cup B)}\overline{\mathbf{\mathbf{h}}}_{\mathcal{E}_{N}(A),\,\mathcal{E}_{N}(B)}(\eta)(\textup{\mbox{div }}\overline{\Phi}_{A,\,B})(\eta)=o_{N}(1)\,N^{-(1+\alpha)}\;.
\]
\end{lem}

\begin{proof}
As $\mathcal{\overline{H}}_{N}\setminus\mathcal{E}_{N}(A\cup B)=\Delta_{N}\cup\{\mathfrak{o}\}$,
Lemma \ref{lem68} implies that the absolute value of the left-hand
side is bounded above by
\[
\sum_{\eta\in\Delta_{N}}\left|(\textup{div }\overline{\Phi}_{A,\,B})(\eta)\right|+\left|(\textup{div }\overline{\Phi}_{A,\,B})(\mathfrak{o})\right|=\sum_{\eta\in\Delta_{N}}\left|(\textup{div }\Phi_{A,\,B})(\eta)\right|+\left|(\textup{div }\Phi_{A,\,B})(\mathcal{E}_{N}^{x})\right|\;.
\]
The last expression is $o_{N}(1)\,N^{-(1+\alpha)}$ by (2) and (3)
of Proposition \ref{p64}.
\end{proof}

\begin{proof}[Proof of Proposition \ref{p613}]
The proof is similar to that of Theorem \ref{t02}. We start by recalling
the functions $\mathbf{f},\,\mathbf{g}$ and the flows $\phi,\,\psi$
from \eqref{ec1} and \eqref{ec3}. Then, by the definition of the
collapsing procedure, it can be verified that
\[
\overline{\mathbf{f}}\in\overline{C}_{1,\,0}(\mathcal{E}_{N}(A),\,\mathcal{\mathcal{E}}_{N}(B))\;\;\text{and\;\;}\overline{\phi}\in\mathfrak{\mathfrak{\overline{S}}}_{\alpha_{N}}(\mathcal{E}_{N}(A),\,\mathcal{\mathcal{E}}_{N}(B))\;,
\]
where $\alpha_{N}=o_{N}(1)\,N^{-(1+\alpha)}$, and that
\begin{equation}
\overline{\Phi}_{\,\mathbf{\overline{f}}}-\overline{\phi}=\overline{\Psi}_{\mathbf{\overline{V}}_{A,\,B}}-\frac{\overline{\Theta}_{N}^{\,*}-\overline{\Theta}_{N}}{2}\;,\label{e6131}
\end{equation}
where $\overline{\Theta}_{N}$ and $\overline{\Theta}_{N}^{\,*}$
are the collapsed flows of $\Theta_{N}$ and $\Theta_{N}^{*}$ defined
in \eqref{ec2}, respectively. By Lemma \ref{lem67}, we have
\begin{equation}
\bigl\Vert\overline{\Theta}_{N}\bigr\Vert_{\mathcal{C}}^{2}=\left(o_{N}(1)+o_{\epsilon}(1)\right)N^{-(1+\alpha)}\;\;\mbox{and}\;\;\bigl\Vert\overline{\Theta}_{N}^{\,*}\bigr\Vert_{\mathcal{C}}^{2}=\left(o_{N}(1)+o_{\epsilon}(1)\right)N^{-(1+\alpha)}\;.\label{e6132}
\end{equation}
Thus, by Theorem \ref{t04}, Lemma \ref{lem615}, and Lemma \ref{lem617},
we have that
\begin{equation}
\textup{\ensuremath{\overline{\textup{cap}}}}_{N}(\mathcal{E}_{N}(A),\,\mathcal{E}_{N}(B))\le\left(1+o_{N}(1)+o_{\epsilon}(1)\right)N^{-(1+\alpha)}\,\textup{cap}_{Y}(A,\,B)\;.\label{epp1}
\end{equation}
For the reversed inequality, it suffices to take
\[
\overline{\mathbf{g}}\in\overline{C}_{0,\,0}(\mathcal{E}_{N}(A),\,\mathcal{\mathcal{E}}_{N}(B))\;\;\text{and\;\;}\overline{\psi}\in\mathfrak{\overline{\mathfrak{S}}}_{1+o_{N}(1)}(\mathcal{E}_{N}(A),\,\mathcal{\mathcal{E}}_{N}(B))\;,
\]
and then repeat the same arguments so that we obtain
\begin{equation}
\textup{\ensuremath{\overline{\textup{cap}}}}_{N}(\mathcal{E}_{N}(A),\,\mathcal{E}_{N}(B))\ge\left(1+o_{N}(1)+o_{\epsilon}(1)\right)N^{-(1+\alpha)}\,\textup{cap}_{Y}(A,\,B)\;.\label{epp2}
\end{equation}
By \eqref{epp1} and \eqref{epp2}, the proof is completed.
\end{proof}

\subsection{\label{sec66}Estimate of mean jump rate and proof of Theorem \ref{t03}}

In view of \eqref{e611}, the probability $\overline{\mathbb{P}}_{\mathfrak{o}}^{N}[\tau_{\mathcal{E}_{N}^{y}}<\tau_{\breve{\mathcal{E}}_{N}^{x,\,y}}]$
should be estimated to obtain the sharp asymptotics of the mean jump
rate $r_{N}(x,\,y)$. This estimate follows from the following proposition.
\begin{prop}
\label{p618}For two disjoint and non-empty subsets $A,\,B$ of $S_{\star}\setminus\{x\}$
satisfying $A\cup B=S_{\star}\setminus\{x\}$, we have that
\[
\lim_{N\rightarrow\infty}\overline{\mathbb{P}}_{\mathfrak{o}}^{N}\left[\tau_{\mathcal{E}_{N}(A)}<\tau_{\mathcal{E}_{N}(B)}\right]=\mathfrak{h}_{A,\,B}(x)\;.
\]
\end{prop}

\begin{proof}
The proof relies on Proposition \ref{p613} and Lemma \ref{lem615}.
Recall the equilibrium potential $\overline{\mathbf{h}}_{\mathcal{E}_{N}(A),\mathcal{\,E}_{N}(B)}$
between $\mathcal{E}_{N}(A)$ and $\mathcal{E}_{N}(B)$, with respect
to the collapsed chain $\overline{\eta}_{N}(\cdot)$. Then, by Proposition
\ref{p613},
\begin{equation}
\left\Vert \,\overline{\Psi}_{\,\overline{\mathbf{h}}_{\mathcal{E}_{N}(A),\,\mathcal{E}_{N}(B)}}\right\Vert _{\mathcal{C}}^{2}=\overline{\textup{cap}}_{N}(\mathcal{E}_{N}(A),\mathcal{\,E}_{N}(B))=\left(1+o_{N}(1)+o_{\epsilon}(1)\right)N^{-(1+\alpha)}\textup{\,cap}_{Y}(A,\,B)\;.\label{ek1}
\end{equation}
By Lemma \ref{lem615},
\begin{equation}
\left\Vert \,\overline{\Psi}_{\overline{\mathbf{V}}_{A,\,B}}\right\Vert _{\mathcal{C}}^{2}=\left(1+o_{N}(1)+o_{\epsilon}(1)\right)N^{-(1+\alpha)}\textup{\,cap}_{Y}(A,\,B)\;.\label{ek2}
\end{equation}
Finally, by \eqref{e6131} and \eqref{e6132}, \begin{equation} \label{e4511} \begin{aligned}
&\left\llangle \,\overline{\Psi}_{\,\overline{\mathbf{V}}_{A,\,B}},\,\overline{\Psi}_{\,\overline{\mathbf{h}}_{\mathcal{E}_{N}(A),\,\mathcal{E}_{N}(B)}}\right\rrangle _{\mathcal{C}}\\
&=\left\llangle \,\overline{\Phi}_{\mathbf{\,\overline{f}}}-\overline{\phi},\,\overline{\Psi}_{\,\overline{\mathbf{h}}_{\mathcal{E}_{N}(A),\,\mathcal{E}_{N}(B)}}\right\rrangle _{\mathcal{C}}+\left(o_{N}(1)+o_{\epsilon}(1)\right)N^{-(1+\alpha)}\;,
\end{aligned} \end{equation}where $\overline{\mathbf{f}}$ and $\overline{\phi}$ are the objects
defined in the proof of Proposition \ref{p613}. By the same computation
as in \eqref{e451}, \begin{equation} \label{e4512} \begin{aligned}
&\left\llangle \,\overline{\Phi}_{\mathbf{\,\overline{f}}}-\overline{\phi},\,\overline{\Psi}_{\,\overline{\mathbf{h}}_{\mathcal{E}_{N}(A),\,\mathcal{E}_{N}(B)}}\right\rrangle _{\mathcal{C}}\\
&=\overline{\textup{cap}}_{N}(\mathcal{E}_{N}(A),\,\mathcal{E}_{N}(B))-\sum_{\eta\in\mathcal{\overline{H}}_{N}\setminus\mathcal{E}_{N}(A\cup B)}\overline{\mathbf{h}}_{\mathcal{E}_{N}(A),\,\mathcal{E}_{N}(B)}(\eta)\,(\textup{div }\overline{\phi})(\eta)
\end{aligned} \end{equation}Thus, by \eqref{e4511}, \eqref{e4512}, and Proposition \ref{p613},
\begin{equation}
\left\llangle \,\overline{\Psi}_{\overline{\mathbf{V}}_{A,\,B}},\,\overline{\Psi}_{\overline{\mathbf{h}}_{\mathcal{E}_{N}(A),\,\mathcal{E}_{N}(B)}}\right\rrangle _{\mathcal{C}}=\left(1+o_{N}(1)+o_{\epsilon}(1)\right)N^{-(1+\alpha)}\,\textup{cap}_{Y}(A,\,B)\label{ek3}
\end{equation}

Define $\mathbf{u}=\overline{\mathbf{h}}_{\mathcal{E}_{N}(A),\,\mathcal{E}_{N}(B)}-\overline{\mathbf{V}}_{A,\,B}$.
Then, by \eqref{ek1}, \eqref{ek2} and \eqref{ek3}, \begin{equation} \label{ekk1} \begin{aligned}
\bigl\Vert\,\overline{\Psi}_{\mathbf{u}}\bigr\Vert_{\mathcal{C}}^{2}&=\bigl\Vert\,\overline{\Psi}_{\,\overline{\mathbf{h}}_{\mathcal{E}_{N}(A),\,\mathcal{E}_{N}(B)}}\bigr\Vert_{\mathcal{C}}^{2}+\bigl\Vert\,\overline{\Psi}_{\,\overline{\mathbf{V}}_{A,\,B}}\bigr\Vert_{\mathcal{C}}^{2}-2\left\llangle \,\overline{\Psi}_{\,\overline{\mathbf{V}}_{A,\,B}},\,\overline{\Psi}_{\,\overline{\mathbf{h}}_{\mathcal{E}_{N}(A),\mathcal{\,E}_{N}(B)}}\right\rrangle _{\mathcal{C}}\\
&=\left(o_{N}(1)+o_{\epsilon}(1)\right)N^{-(1+\alpha)}\;.
\end{aligned} \end{equation} As $\mathbf{u}(\mathfrak{o})=\overline{\mathbf{h}}_{\mathcal{E}_{N}(A),\,\mathcal{E}_{N}(B)}(\mathfrak{o})-\mathfrak{h}_{A,\,B}(x)$
and $\mathbf{u}(\eta)=0$ for all $\eta\in\mathcal{E}_{N}(A\cup B)$,
we can write
\[
\mathbf{u}=\left(\,\overline{\mathbf{h}}_{\mathcal{E}_{N}(A),\,\mathcal{E}_{N}(B)}(\mathfrak{o})-\mathfrak{h}_{A,\,B}(x)\right)\mathbf{u}_{0}
\]
for some $\mathbf{u}_{0}\in\overline{C}_{1,\,0}(\{\mathfrak{o}\},\,\mathcal{E}_{N}(A\cup B))$.
Thus,
\begin{equation}
\left\Vert \overline{\Psi}_{\mathbf{u}}\right\Vert _{\mathcal{C}}^{2}=\overline{\mathscr{D}}_{N}(\mathbf{u})=\left(\,\overline{\mathbf{h}}_{\mathcal{E}_{N}(A),\,\mathcal{E}_{N}(B)}(\mathfrak{o})-\mathfrak{h}_{A,\,B}(x)\right)^{2}\,\overline{\mathscr{D}}_{N}(\mathbf{u}_{0})\;.\label{ekk2}
\end{equation}
By the Dirichlet principle for \textit{reversible} dynamics, Lemma
\ref{lem612}, Lemma \ref{lem613}, and Theorem \ref{t02}, we have
\begin{align}
\overline{\mathscr{D}}_{N}(\mathbf{u}_{0}) & \ge\overline{\textup{cap}}_{N}^{\,s}(\mathfrak{o},\,\mathcal{E}_{N}(A\cup B))\ge C_{0}^{-1}\,\overline{\textup{cap}}_{N}(\mathfrak{o},\,\mathcal{E}_{N}(A\cup B))\nonumber \\
 & =C_{0}^{-1}\,\textup{cap}_{N}(\mathcal{E}_{N}^{x},\,\mathcal{E}_{N}(A\cup B))\label{ekk3}\\
 & =C_{0}^{-1}\left(1+o_{N}(1)+o_{\epsilon}(1)\right)N^{-(1+\alpha)}\textup{\,cap}_{Y}(x,\,A\cup B)\nonumber
\end{align}
Therefore, by \eqref{ekk1}, \eqref{ekk2} and \eqref{ekk3},
\[
\left[\,\overline{\mathbf{h}}_{\mathcal{E}_{N}(A),\,\mathcal{E}_{N}(B)}(\mathfrak{o})-\mathfrak{h}_{A,\,B}(x)\right]^{2}\le o_{N}(1)+o_{\epsilon}(1)\;.
\]
Thus, the proof is completed by taking $\limsup_{N\rightarrow\infty}$
on both sides and then letting $\epsilon\rightarrow0$.
\end{proof}
The following proposition completes the proof of Theorem \ref{t03}.
\begin{prop}
\label{p619}For all $x,\,y\in S_{\star}$, we have that
\[
\lim_{N\rightarrow\infty}N^{1+\alpha}\,r_{N}(x,\,y)=a(x,\,y)\;.
\]
\end{prop}

\begin{proof}
By Theorem \ref{t01}, \eqref{e610}, and Corollary \ref{c02}, we
have
\begin{equation}
\lambda_{N}(x)=\frac{\textup{cap}_{N}(\mathcal{E}_{x}^{N},\,\breve{\mathcal{E}}_{x}^{N})}{\mu(\mathcal{E}_{x}^{N})}=\left(1+o_{N}(1)\right)N^{-(1+\alpha)}\,\frac{1}{M_{\star}\,\Gamma(\alpha)\,I_{\alpha}}\,\sum_{y\in S_{\star}\setminus\{x\}}\textup{cap}_{X}(x,\,y)\;.\label{e6191}
\end{equation}
Recall from \eqref{hu4} the definition of $\mathfrak{h}_{y,\,S_{\star}\setminus\{x,\,y\}}$
and from \eqref{hu3} the definition of chain $\widehat{Y}(\cdot)$.
Write $\tau=\inf\left\{ t:\widehat{Y}(t)\neq\widehat{Y}(0)\right\} $.
Then, one can observe that
\[
\mathfrak{h}_{y,\,S_{\star}\setminus\{x,\,y\}}(x)=\widehat{\mathbf{Q}}_{x}\left(\widehat{Y}(\tau)=y\right)=\,\frac{\textup{cap}_{X}(x,\,y)}{\sum_{y\in S_{\star}\setminus\{x\}}\textup{cap}_{X}(x,\,y)}\;.
\]
Thus, by Proposition \ref{p618},
\begin{equation}
\frac{r_{N}(x,\,y)}{\lambda_{N}(x)}=\left(1+o_{N}(1)\right)\mathfrak{h}_{y,\,S_{\star}\setminus\{x,\,y\}}(x)=\left(1+o_{N}(1)\right)\,\frac{\textup{cap}_{X}(x,\,y)}{\sum_{y\in S_{\star}\setminus\{x\}}\textup{cap}_{X}(x,\,y)}\;.\label{e6192}
\end{equation}
The proof is completed by multiplying \eqref{e6191} and \eqref{e6192}.
\end{proof}

\section{\label{sec7}Approximation of equilibrium potentials}

The construction of the approximation of the equilibrium potential
between valleys, in reversible set-up was carried out in \cite{BL3}.
The corresponding construction is presented in this section. The following
comments are valid throughout the remaining of the paper.
\begin{itemize}
\item The dependency on $N$ and $\epsilon$ will be ignored for the subsets
of $\mathcal{H}_{N}$, functions on $\mathcal{H}_{N}$, and flows
on $\mathcal{H}_{N}^{\otimes}$ when there is no risk of confusion.
For instance, the notation $\mathcal{E}^{x}$ will be used instead
of $\mathcal{E}_{N}^{x}$ or $\mathcal{D}^{x}$ instead of $\mathcal{D}_{N,\,\epsilon}^{x}$
(cf. \eqref{dx}). Of course, the sets such as $\mathcal{H}_{N}$
or $\mathcal{H}_{N\,,S_{0}}$ defined in Section \ref{sec72}, the
subscript is retained to stress the dependency, as $N$ is occasionally
replaced with some other number such as $N\epsilon$.
\item We shall assume that $N$ is sufficiently large so that $N\epsilon>\pi_{N}>\ell_{N}$.
Recall that $\pi_{N}=\lfloor N^{\frac{1}{\alpha}+\frac{1}{2}}\rfloor\ll N$.
For notational simplicity, it will be assumed that $N\epsilon$ is
an integer. Of course, all the arguments are valid without this assumption.
\end{itemize}
This section is organized as follows. In Section \ref{sec72}, several
basic properties of invariant measure that are frequently used are
investigated. In Section \ref{sec71}, a global geometry of $\mathcal{H}_{N}$
is presented that is well-suited for describing the metastability
of non-reversible zero-range processes. In Section \ref{sec73} certain
auxiliary functions are introduced that play a fundamental role in
the construction of test flows. In Section \ref{sec74} the approximation
of equilibrium potential on tubes is constructed, and is finally extended
into a global object in Section \ref{sec75}, thus completing the
construction. The proof of Proposition \ref{p63} is also given in
that subsection as well.

\subsection{\label{sec72}Estimates related to invariant measure}

For a non-empty set $S_{0}\subseteq S$ and $k\in\mathbb{N}$, let
$\mathcal{H}_{k,\,S_{0}}$ be the set of particle configuration on
$S_{0}$ with $k$ particles, i.e.,
\[
\mathcal{H}_{k,\,S_{0}}=\Bigl\{\zeta=(\zeta_{x})_{x\in S_{0}}\in\mathbb{N}^{S_{0}}:\sum_{x\in S_{0}}\xi_{x}=k\Bigr\}\;.
\]
Using the notations introduced in \eqref{rec1}, let
\[
Z_{k,\,S_{0}}=k^{\alpha}\,\sum_{\zeta\in\mathcal{H}_{k,\,S_{0}}}\,\prod_{x\in S_{0}}\frac{m_{\star}(x)^{\zeta_{x}}}{a(\zeta_{x})}=k^{\alpha}\sum_{\zeta\in\mathcal{H}_{k\,,S_{0}}}\frac{m_{\star}^{\zeta}}{a(\zeta)}\;.
\]
Hence, $\mathcal{H}_{N,\,S}=\mathcal{H}_{N}$ and $Z_{N,\,S}=Z_{N}$.
By the same principle as in Proposition \ref{e21}, the following
result is obtained
\begin{lem}
\label{le71}For all non-empty set $S_{0}\subseteq S$, we have
\[
\lim_{k\rightarrow\infty}Z_{k,\,S_{0}}=|S_{0}\cap S_{\star}|\,\Gamma(\alpha)^{|S_{0}\cap S_{\star}|-1}\prod_{x\in S_{0}\setminus S_{\star}}\Gamma_{x}\;.
\]
\end{lem}

In this subsection, let us fix a sequence $(d_{N})_{N\in\mathbb{N}}$
of positive integer satisfying $1\ll d_{N}\ll N$, namely $\lim_{N\rightarrow\infty}d_{N}=+\infty$
and $\lim_{N\rightarrow\infty}d_{N}/N=0$. In the applications, $d_{N}$
is either $\pi_{N}$ or $\ell_{N}$.
\begin{lem}
\label{le72}For all non-empty sets $S_{0}\subseteq S$, we have
\[
\lim_{N\rightarrow\infty}\sum_{k=0}^{d_{N}}\,\,\sum_{\zeta\in\mathcal{H}_{k,\,S_{0}}}\frac{m_{\star}^{\zeta}}{a(\zeta)}=\Gamma(\alpha)^{|S_{0}\cap S_{\star}|}\prod_{x\in S_{0}\setminus S_{\star}}\Gamma_{x}\;.
\]
\end{lem}

\begin{proof}
By Lemma \ref{le71},
\[
\sum_{k=d_{N}+1}^{\infty}\,\,\sum_{\zeta\in\mathcal{H}_{k,\,S_{0}}}\frac{m_{\star}^{\zeta}}{a(\zeta)}=\sum_{k=d_{N}+1}^{\infty}\frac{Z_{k,\,S_{0}}}{k^{\alpha}}\le C\sum_{k=d_{N}+1}^{\infty}\frac{1}{k^{\alpha}}=o_{N}(1)\;.
\]
Therefore, it suffices to verify that
\[
\sum_{k=0}^{\infty}\,\,\sum_{\zeta\in\mathcal{H}_{k,\,S_{0}}}\frac{m_{\star}^{\zeta}}{a(\zeta)}=\Gamma(\alpha)^{|S_{0}\cap S_{\star}|}\prod_{x\in S_{0}\setminus S_{\star}}\Gamma_{x}=\prod_{x\in S_{0}}\Gamma_{x}\;.
\]
This is obvious because if we express $\Gamma(\alpha)$ and $\Gamma_{x}$
as infinite series, and expand the right-hand side, then the left-hand
side is obtained.
\end{proof}
For $d<k$, let
\[
\mathcal{H}_{k,\,S_{0}}(d)=\left\{ \zeta\in\mathcal{H}_{k,\,S_{0}}:\zeta_{x}<k-d\;\;\text{for all }x\in S_{0}\cap S_{\star}\right\} \;.
\]

\begin{lem}
\label{le73}For a non-empty set $S_{0}\subseteq S$ and sufficiently
large $N$, we have
\[
\sum_{\zeta\in\mathcal{H}_{N,\,S_{0}}(d_{N})}\frac{m_{\star}^{\zeta}}{a(\zeta)}<\frac{C}{N^{\alpha}\,d_{N}^{\alpha-1}}\;.
\]
\end{lem}

\begin{proof}
It should be noted that
\[
\mathcal{H}_{N,\,S_{0}}(d_{N})\subseteq\widetilde{\mathcal{H}}_{N,\,S_{0}}(d_{N})\cup\Bigl(\bigcup_{y\in S_{0}\setminus S_{\star}}\mathcal{H}_{N,\,S_{0}}^{y}(d_{N})\Bigr)\;,
\]
where
\begin{align*}
 & \widetilde{\mathcal{H}}_{N,\,S_{0}}(d_{N})=\left\{ \zeta\in\mathcal{H}_{N,\,S_{0}}:\zeta_{x}<N-d_{N}\;\;\text{for all }x\in S_{0}\right\} \;,\\
 & \mathcal{H}_{N,\,S_{0}}^{y}(d_{N})=\left\{ \zeta\in\mathcal{H}_{N,\,S_{0}}:\zeta_{y}\ge N-d_{N}\right\} \;\;;\;y\in S_{0}\setminus S_{\star}\;.
\end{align*}
The set $\widetilde{\mathcal{H}}_{N,\,S_{0}}(d_{N})$ differs from
$\mathcal{H}_{N,\,S_{0}}(d_{N})$ as it is additionally imposed that
$\zeta_{x}<N-d_{N}$ for $x\in S_{0}\setminus S_{\star}$ on this
set. By \cite[Lemma 3.2]{BL3},
\begin{equation}
\sum_{\zeta\in\widetilde{\mathcal{H}}_{N,\,S_{0}}(d_{N})}\frac{m_{\star}^{\zeta}}{a(\zeta)}<\frac{C}{N^{\alpha}\,d_{N}^{\alpha-1}}\;.\label{dec3}
\end{equation}
Let $\widehat{m}_{\star}=\max\left\{ m_{\star}(y):y\in S\setminus S_{\star}\right\} <1$.
As $d_{N}\ll N$, there exists sufficiently large $N_{0}$ such that
for all $N>N_{0}$,
\[
\frac{\widehat{m}_{\star}^{N-d_{N}}}{(N-d_{N})^{\alpha}}<\frac{1}{N^{\alpha}\,d_{N}^{\alpha-1}}
\]
Hence, for $y\in S_{0}\setminus S_{\star}$ and $N>N_{0}$,
\begin{equation}
\sum_{\zeta\in\mathcal{H}_{N,\,S_{0}}^{y}(d_{N})}\frac{m_{\star}^{\zeta}}{a(\zeta)}\le\frac{m_{\star}(y)^{\zeta_{y}}}{a(\zeta_{y})}\le\frac{m_{\star}(y)^{N-d_{N}}}{a(N-d_{N})}<\frac{1}{N^{\alpha\,}d_{N}^{\alpha-1}}\;.\label{dec2}
\end{equation}
The proof is completed by \eqref{dec3} and \eqref{dec2}.
\end{proof}

\subsection{\label{sec71}Enlarged valleys and saddle tubes}

Herein several subsets of $\mathcal{H}_{N}$ are defined that are
suitably designed to capture typical metastable transitions among
valley and in turn play a central role in the construction of approximations
of equilibrium potentials and optimal flows. Figure \ref{fig1} shows
a visualization of the sets that are defined below.
\begin{figure}
\includegraphics[scale=0.073]{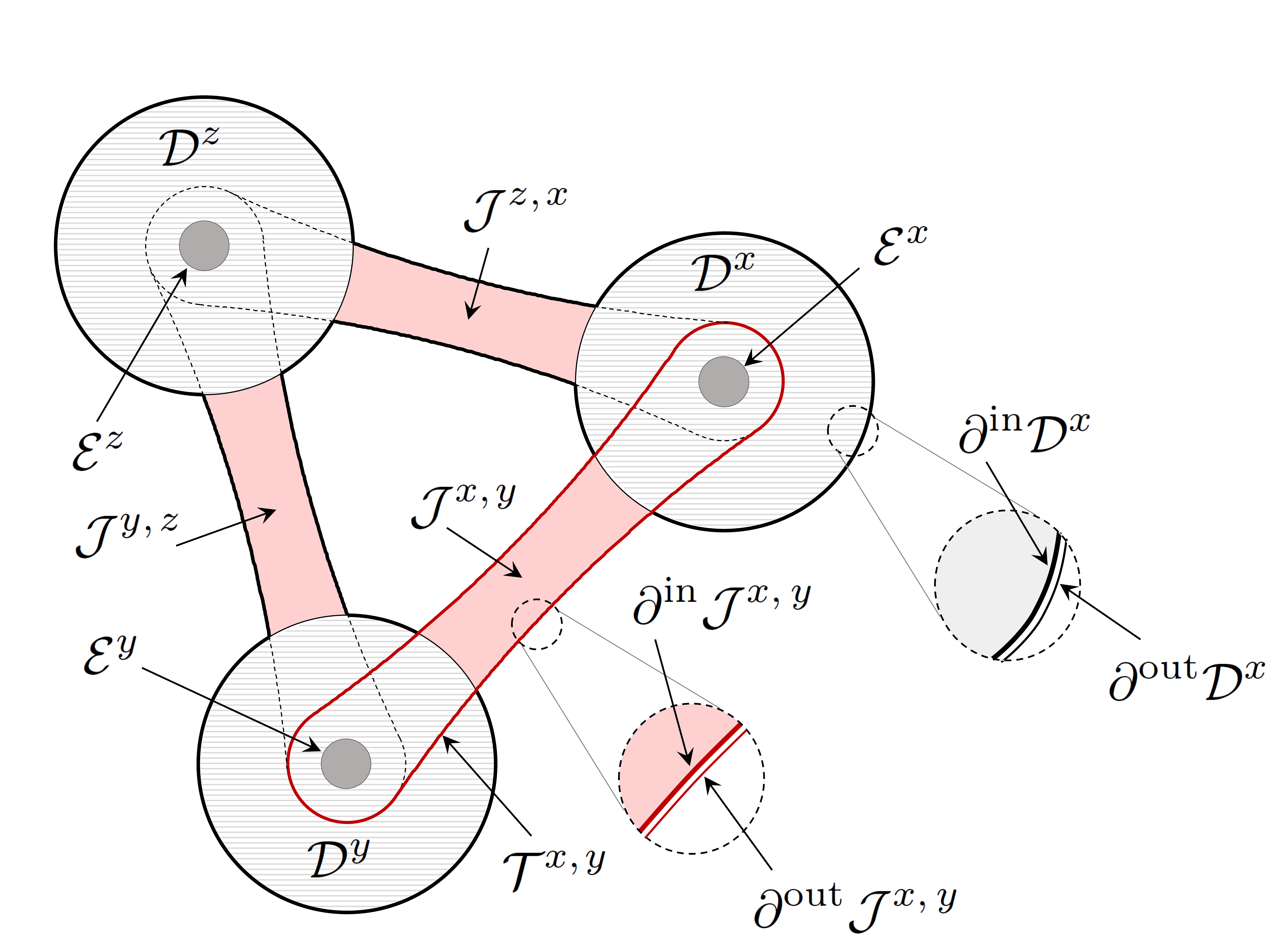}\caption{Structure of metastable valleys, wells, and saddle tubes for the case
$S_{\star}=\{x,\,y,\,z\}.$}
\label{fig1}
\end{figure}

The enlarged valley is defined by
\begin{equation}
\mathcal{D}^{x}=\{\eta\in\mathcal{H}_{N}:\eta_{x}\ge N(1-2\epsilon)\}\;\;;\;x\in S_{\star}\;.\label{dx}
\end{equation}
Thus, $\mathcal{E}^{x}\subset\mathcal{D}^{x}$ for all $x$ as it
is assumed that $N$ is sufficiently large so that $N\epsilon>\ell_{N}$.
The set $\mathcal{D}^{x}$ can be regarded as a metastable well corresponding
to $\mathcal{E}^{x}$ in the sense of \cite{BL1}. For $x,\,y\in S_{\star}$,
the tube between valleys $\mathcal{E}^{x}$ and $\mathcal{E}^{y}$
is defined by
\[
\mathcal{T}^{x,\,y}=\left\{ \eta\in\mathcal{H}_{N}:\eta_{x}+\eta_{y}\ge N-\pi_{N}\right\} \;,
\]
and the saddle tube is defined by
\begin{align*}
\mathcal{J}^{x,\,y} & =\mathcal{T}^{x,\,y}\setminus(\mathcal{D}^{x}\cup\mathcal{D}^{y})\;.\\
 & =\left\{ \eta\in\mathcal{H}_{N}:\eta_{x}+\eta_{y}\ge N-\pi_{N}\mbox{ and}\;\eta_{x},\,\eta_{y}<N(1-2\epsilon)\right\} \;.
\end{align*}
As $\pi_{N}<N\epsilon$, one can observe that
\begin{equation}
\eta_{x},\,\eta_{y}\in[N\epsilon,\,N(1-2\epsilon)]\;\;\text{for all }\eta\in\mathcal{J}^{x,y}\;.\label{obe1}
\end{equation}
We claim that $\mathcal{J}^{x,\,y}\cap\mathcal{J}^{x',\,y'}=\emptyset$
unless $\{x,\,y\}=\{x',\,y'\}$. To prove this claim, it suffices
to verify that for three points $x,\,y,\,z\in S_{\star}$, we have
\[
\mathcal{T}^{x,\,y}\cap\mathcal{T}^{x,\,z}\subset\mathcal{D}^{x}\;.
\]
This is obvious because if $\eta\in\mathcal{T}^{x,\,y}\cap\mathcal{T}^{x,\,z}$,
then
\[
2(N-\pi_{N})\le(\eta_{x}+\eta_{y})+(\eta_{x}+\eta_{z})\le\eta_{x}+N\;,
\]
and thus $\eta_{x}\ge N-2\pi_{N}>N(1-2\epsilon)$.

Let
\[
\mathcal{G}=\Bigl(\bigcup_{x\in S_{\star}}\mathcal{D}^{x}\Bigr)\bigcup\Bigl(\bigcup_{\{x,\,y\}\subset S_{\star}}\mathcal{J}^{x,\,y}\Bigr)\;.
\]
It should be noticed that the unions in the previous definition are
disjoint. The construction of approximating functions and optimal
flows is focused on $\mathcal{G}$, particularly on each of its components.
The definition of approximating objects outside $\mathcal{G}$ hardly
affects the computation except for the discontinuity along the boundary
of $\mathcal{G}$. Hence, the boundary of $\mathcal{G}$ is carefully
analyzed. It can be decomposed into several parts. The inner and outer
boundaries of the saddle tube $\mathcal{J}^{x,\,y}$, $x,\,y\in S_{\star}$,
are defined by
\begin{align*}
 & \partial^{\textrm{\,in}}\mathcal{J}^{x,\,y}=\left\{ \eta\in\mathcal{J}^{x,\,y}:\eta_{x}+\eta_{y}=N-\pi_{N}\right\} \;,\\
 & \partial^{\,\textrm{out}}\mathcal{J}^{x,\,y}=\left\{ \eta\in\mathcal{G}^{c}:\eta_{x}+\eta_{y}=N-\pi_{N}-1\right\} \;,
\end{align*}
respectively. The corresponding boundaries of the enlarged valleys
$\mathcal{D}^{x}$, $x\in S_{\star}$, are defined by
\begin{align*}
 & \partial^{\textrm{\,in}}\mathcal{D}^{x}=\left\{ \eta\in\mathcal{D}^{x}:\eta_{x}=N(1-2\epsilon)\right\} \setminus\Bigl(\bigcup_{y\in S_{\star}\setminus\{x\}}\mathcal{T}^{x,\,y}\Bigr)\;,\\
 & \partial^{\textrm{\,out}}\mathcal{D}^{x}=\left\{ \eta\in\mathcal{G}^{c}:\eta_{x}=N(1-2\epsilon)-1\right\} \setminus\Bigl(\bigcup_{y\in S_{\star}\setminus\{x\}}\mathcal{T}^{x,\,y}\Bigr)\;.
\end{align*}
Finally, the inner and outer boundaries of $\mathcal{G}$ are defined
by \begin{equation} \label{de1} \begin{aligned}
&\partial^{\textrm{in}}\mathcal{G}=\Bigl(\bigcup_{x\in S_{\star}}\partial^{\textrm{in}}\mathcal{D}^{x}\Bigr)\bigcup\Bigl(\bigcup_{\{x,\,y\}\subset S_{\star}}\partial^{\textrm{in}}\mathcal{J}^{x,\,y}\Bigr)\;,\\
&\partial^{\textrm{out}}\mathcal{G}=\Bigl(\bigcup_{x\in S_{\star}}\partial^{\textrm{out}}\mathcal{D}^{x}\Bigr)\bigcup\Bigl(\bigcup_{\{x,\,y\}\subset S_{\star}}\partial^{\textrm{out}}\mathcal{J}^{x,\,y}\Bigr)\;,
\end{aligned} \end{equation} respectively. In addition, the interior of $\mathcal{J}^{x,\,y}$,
$x,\,y\in S_{\star}$, and of $\mathcal{D}^{x}$, $x\in S_{\star}$,
are defined by
\[
\mathcal{J}_{\textrm{int }}^{x,\,y}=\mathcal{J}^{x,\,y}\setminus\partial^{\textrm{\,in}}\mathcal{J}^{x,\,y}\;\;\text{and\;\;\ensuremath{\mathcal{D}_{\textrm{int}}^{x}}=\ensuremath{\mathcal{D}^{x}}\ensuremath{\ensuremath{\setminus\partial^{\textrm{\,in}}\mathcal{D}^{x}}}}\;,
\]
respectively. Thus, the set $\mathcal{G}$ can be further decomposed
as
\begin{equation}
\mathcal{G}=\Bigl(\bigcup_{x\in S_{\star}}\mathcal{D}_{\textrm{int}}^{x}\Bigr)\,\bigcup\,\Bigl(\bigcup_{\{x,\,y\}\subset S_{\star}}\mathcal{J}_{\textrm{int }}^{x,\,y}\Bigr)\bigcup\partial^{\textrm{\,in}}\mathcal{G}\;.\label{decompg}
\end{equation}
The interior of the set $\mathcal{G}^{c}=\mathcal{H}_{N}\setminus\mathcal{G}$
is defined by
\begin{equation}
(\mathcal{G}^{c})_{\textrm{int}}=\mathcal{G}^{c}\setminus\partial^{\textrm{\,out}}\mathcal{G}\;.\label{decompgc}
\end{equation}
To control the discontinuity of approximating objects along the boundary,
the following lemma is required.
\begin{lem}
\label{lem71}We have that
\[
\mu_{N}(\partial^{\textrm{\textup{\,in}}}\mathcal{G})=o_{N}(1)\,N^{-(1+\alpha)}\;\;\text{and\;\;}\mu_{N}(\partial^{\textrm{\textup{\,out}}}\mathcal{G})=o_{N}(1)\,N^{-(1+\alpha)}\;.
\]
\end{lem}

\begin{proof}
Only the first estimate will be proved; the proof of the second is
identical. By \eqref{de1},
\[
\mu_{N}(\partial^{\textrm{\,in}}\mathcal{G})=\sum_{x\in S_{\star}}\mu_{N}(\partial^{\,\textrm{in}}\mathcal{D}^{x})+\sum_{\{x,\,y\}\subset S_{\star}}\mu_{N}(\partial^{\textrm{\,in}}\mathcal{J}^{x,\,y})\;.
\]
For the first summation, let us fix $x\in S_{\star}$ and let us temporarily
denote $\zeta=\zeta(\eta)\in\mathbb{N}^{S\setminus\{x\}}$ the particle
configuration of $\eta$ conditioned on $S\setminus\{x\}$, i.e.,
$\zeta_{y}=\eta_{y}$ for all $y\in S\setminus\{x\}$. Then, for $\eta\in\partial^{\,\textrm{in}}\mathcal{D}^{x}$,
it follows from the definition of $\partial^{\textrm{\,in}}\mathcal{D}^{x}$
that
\[
|\zeta|=\sum_{y\in S\setminus\{x\}}\zeta_{y}=N-\eta_{x}=2N\epsilon\;.
\]
That is, $\zeta\in\mathcal{H}_{2N\epsilon,\,S\setminus\{x\}}$. Moreover,
as $\eta_{x}+\eta_{y}<N-\pi_{N}$ for all $\eta\in\partial^{\,\textrm{in}}\mathcal{D}^{x}$
and $y\in S_{\star}$, we have
\[
\zeta_{y}=\eta_{y}<N-\eta_{x}-\pi_{N}<2N\epsilon-\pi_{N}\;\;;\;y\in S_{\star}\setminus\{x\}\;.
\]
Therefore, $\zeta\in\mathcal{H}_{2N\epsilon,\,S\setminus\{x\}}(\pi_{N}).$
Hence, by Proposition \ref{e21} and Lemma \ref{le73},
\[
\mu_{N}(\partial^{\textrm{\,in}}\mathcal{D}^{x})=\frac{N^{\alpha}}{Z_{N}}\,\frac{1}{a(N-2N\epsilon)}\,\,\sum_{\zeta\in\mathcal{H}_{2N\epsilon,\,S\setminus\{x\}}(\pi_{N})}\frac{m_{\star}^{\zeta}}{a(\zeta)}\le\frac{C}{(2N\epsilon)^{\alpha}(\pi_{N})^{\alpha-1}}\;.
\]
This proves that $\mu_{N}(\partial^{\textrm{\,in}}\mathcal{D}^{x})=o_{N}(1)\,N^{-(1+\alpha)}$,
as $\pi_{N}^{\alpha-1}\gg N$.

For the second estimate of the lemma, it suffices to demonstrate that
$\mu_{N}(\partial^{\textrm{\,in}}\mathcal{J}^{x,\,y})=o_{N}(1)\,N^{-(1+\alpha)}$
for all $x,\,y\in S_{\star}$. Let us fix $x,\,y\in S_{\star}$, and
let us temporarily write $\xi=\xi(\eta)\in\mathbb{N}^{S\setminus\{x,\,y\}}$
the particle configuration of $\eta$ on $S\setminus\{x,\,y\}$. Then,
by the definition of $\partial^{\textrm{\,in}}\mathcal{J}^{x,\,y}$,
we have that
\[
|\xi|=\sum_{z\in S\setminus\{x,\,y\}}\xi_{z}=N-\eta_{x}-\eta_{y}=\pi_{N}\;.
\]
That is, $\xi\in\mathcal{H}_{\pi_{N},\,S\setminus\{x,\,y\}}.$ Hence,
by Proposition \ref{e21} and \eqref{obe1}, we have
\begin{equation}
\mu_{N}(\partial^{\textrm{\,in}}\mathcal{J}^{x,\,y})\le C\,N^{\alpha}\,\sum_{\xi\in\mathcal{H}_{\pi_{N},\,S\setminus\{x,\,y\}}}\frac{m_{\star}^{\xi}}{a(\xi)}\,\sum_{i=N\epsilon}^{N(1-2\epsilon)}\frac{1}{a(i)\,a(N-\pi_{N}-i)}\;.\label{obe21}
\end{equation}
By Lemma \ref{le1}, the first summation is bounded above by $C\,\pi_{N}^{-\alpha}$,
whereas the second summation is bounded by
\begin{equation}
\sum_{i=N\epsilon}^{N(1-2\epsilon)}\frac{1}{a(i)\,a(N-\pi_{N}-i)}\le\frac{N(1-3\epsilon)}{a(N\epsilon)\,a(N\epsilon)}=\frac{C}{N^{2\alpha-1}}\;.\label{obe22}
\end{equation}
Consequently, $\mu_{N}(\partial^{\textrm{\,in}}\mathcal{J}^{x,\,y})\le C\,\pi_{N}^{-\alpha}\,N^{-(\alpha-1)}$.
Thus the proof is completed as $\pi_{N}^{\alpha}\gg N^{2}$.
\end{proof}
The mass of $\mathcal{J}^{x,\,y}$ satisfies the following estimate.
\begin{lem}
\label{lem710}For all $x,\,y\in S_{\star}$, there exists a constant
$C>0$ such that
\[
\mu_{N}(\mathcal{J}^{x,\,y})\le CN^{-(\alpha-1)}\;.
\]
\end{lem}

\begin{proof}
By Proposition \ref{e21}, \eqref{obe1}, and computations as in \eqref{obe21} and \eqref{obe22},
we obtain
\begin{align*}
\mu_{N}(\mathcal{J}^{x,\,y}) & \le CN^{\alpha}\sum_{k=0}^{\pi_{N}}\left[\sum_{i=N\epsilon}^{N(1-2\epsilon)}\frac{1}{a(i)a(N-k-i)}\sum_{\zeta\in\mathcal{H}_{k,\,S\setminus\{x,\,y\}}}\frac{m_{\star}^{\zeta}}{a(\zeta)}\right]\\
 & \le CN^{-(\alpha-1)}\sum_{k=0}^{\pi_{N}}\,\sum_{\zeta\in\mathcal{H}_{k,\,S\setminus\{x,\,y\}}}\frac{m_{\star}^{\zeta}}{a(\zeta)}\;.
\end{align*}
Therefore, the proof is completed by Lemma \ref{le72}.
\end{proof}

\subsection{\label{sec73}Auxiliary functions}

To introduce approximations of equilibrium potentials, an important
function is
\[
V(t)=\frac{1}{I_{\alpha}}\int_{0}^{t}s^{\alpha}(1-s)^{\alpha}
\]
which essentially captures the one-dimensional projection of the equilibrium
potential along the tube. However, this function cannot be used in
its own form, and thus an approximated version is required. We introduce
this object below.

For $\epsilon\in(0,\,1/8)$, a continuous piece-wise linear function
$\widehat{\gamma}_{\epsilon}:\mathbb{R}\rightarrow[0,1]$ is defined
by
\[
\widehat{\gamma}_{\epsilon}(t)=\begin{cases}
0 & \text{if }t\in(-\infty,\,4\epsilon]\\
(t-3\epsilon)/(1-6\epsilon) & \text{if }t\in[4\epsilon,\,1-4\epsilon]\\
1 & \text{if }t\in[1-4\epsilon,\,\infty)\;.
\end{cases}
\]
where $I_{\alpha}$ has been introduced in \eqref{ialpha}. Let $\phi:\mathbb{R}\rightarrow\mathbb{R}$
be a smooth, symmetric, non-negative function supported on $[-1,\,1]$
such that $\int_{-1}^{1}\phi(t)dt=1$, and let $\phi_{\delta}(t)=(1/\delta)\,\phi(t/\delta)$
for $\delta>0$. Namely, $(\phi_{\delta})_{\delta>0}$ is a sequence
of standard smooth mollifiers. Furthermore, let $\gamma_{\epsilon}:\mathbb{R}\rightarrow[0,\,1]$
be defined by $\gamma_{\epsilon}=\widehat{\gamma}_{\epsilon}*\phi_{\epsilon}$.
\begin{lem}
\label{bg}There exists $\epsilon_{0}>0$ such that the following
properties of $\gamma$ hold for all $\epsilon\in(0,\,\epsilon_{0})$:
\begin{enumerate}
\item $\gamma_{\epsilon}$ is a smooth increasing function and satisfies
\[
\gamma_{\epsilon}(t)=\begin{cases}
0 & \text{if }t\in(-\infty,\,3\epsilon]\\
(t-3\epsilon)/(1-6\epsilon) & \text{if }t\in[5\epsilon,\,1-5\epsilon]\\
1 & \text{if }t\in[1-3\epsilon,\,\infty)\;.
\end{cases}
\]
\item $\gamma_{\epsilon}(1-t)=1-\gamma_{\epsilon}(t)$.
\item $\gamma_{\epsilon}'(t)\le1+\epsilon^{1/2}$ for all $t\in[0,\,1]$.
\item $\gamma_{\epsilon}(t)/t\le1+\epsilon^{1/2}$ for all $t\in(0,\,1]$.
\item $\gamma_{\epsilon}(t)/t\ge1-4\epsilon^{1/2}>0$ for all $t\in[\epsilon^{1/2},\,1]$.
\end{enumerate}
\end{lem}

The proof is elementary and left to the reader. Henceforth, we shall
assume that $\epsilon\in(0,\,\epsilon_{0})$, so that the above properties
hold, and the notation $\gamma$ will be used instead of $\gamma_{\epsilon}$.
Let a non-decreasing smooth function $H=H_{\epsilon}:[0,\,1]\rightarrow[0,\,1]$
be defined by
\begin{equation}
H(t)=\frac{1}{I_{\alpha}}\int_{0}^{\gamma(t)}s^{\alpha}(1-s)^{\alpha}ds=V(\gamma(t))\;.\label{ck00}
\end{equation}
This function is an approximation of $V(\cdot)$. By (1) of Lemma
\ref{bg}, $H(\cdot)$ satisfies
\begin{equation}
H(t)=\begin{cases}
0 & \text{if }t\in[0,\,3\epsilon]\\
1 & \text{if }t\in[1-3\epsilon,\,1]\;.
\end{cases}\label{ck0}
\end{equation}
Let
\[
U(t)=(1/I_{\alpha})\,t^{\alpha}(1-t)^{\alpha}=V'(t)\;.
\]
The following basic results will be useful later.
\begin{lem}
\label{lemH}For all $t\in[0,\,1]$, it holds that
\[
U(\gamma(t))\le(1+o_{\epsilon}(1))\,U(t)\;\;\text{and\;}\;H'(t)\le(1+o_{\epsilon}(1))\,U(t)\;.
\]
\end{lem}

\begin{proof}
The first inequality follows from (2) and (3) of Lemma \ref{bg},
namely,
\[
U(\gamma(t))=\gamma(t)^{\alpha}\,\gamma(1-t)^{\alpha}\le\left\{ (1+\epsilon^{1/2})t\right\} ^{\alpha}\left\{ (1+\epsilon^{1/2})(1-t)\right\} ^{\alpha}=(1+o_{\epsilon}(1))\,U(t)\;.
\]
The second inequality is now obvious by this and (3) of Lemma \ref{bg},
since
\begin{equation}
H'(t)=\gamma'(t)\,U(\gamma(t))\;.\label{e733}
\end{equation}
\end{proof}
\begin{lem}
\label{lemH1}For all $t\in[\,\epsilon^{1/2},\,1-\epsilon^{1/2}\,]$,
it holds that
\[
H'(t)=\left(1+o_{\epsilon}(1)\right)U(t)\;.
\]
\end{lem}

\begin{proof}
By (5) of Lemma \ref{bg}, for all $t\in[\,\epsilon^{1/2},\,1-\epsilon^{1/2}\,]$,
we have
\[
U(\gamma(t))=\gamma(t)^{\alpha}\,\gamma(1-t)^{\alpha}\ge\left[(1-4\,\epsilon^{1/2})\,t\right]^{\alpha}\left[(1-4\,\epsilon^{1/2})(1-t)\right]^{\alpha}=\left(1+o_{\epsilon}(1)\right)U(t)\;.
\]
As $\gamma'(t)\ge1$, by the above computation and \eqref{e733},
\[
H'(t)-U(t)\ge U(\gamma(t))-U(t)\ge o_{\epsilon}(1)\,U(t)\;.
\]
For the upper bound of $H'(t)-U(t)$, it suffices to use the second
inequality of Lemma \ref{lemH}.
\end{proof}
\begin{lem}
\label{lem750}For all $x,\,y\in S_{\star}$ and for all $\eta\in\mathcal{T}^{x,\,y}$,
it holds that
\[
0\le U\left(\frac{\eta_{x}}{N}\right)-\frac{a(\eta_{x})\,a(\eta_{y})}{N^{2\alpha}\,I_{\alpha}}\le C\frac{\pi_{N}}{N}\;.
\]
\end{lem}

\begin{proof}
The left inequality is trivial. For the right inequality, by the mean-value
theorem,
\[
U\left(\frac{\eta_{x}}{N}\right)-\frac{a(\eta_{x})\,a(\eta_{y})}{N^{2\alpha}\,I_{\alpha}}=\frac{\eta_{x}^{\alpha}}{N^{\alpha}\,I_{\alpha}}\left[\left(\frac{N-\eta_{x}}{N}\right)^{\alpha}-\left(\frac{\eta_{y}}{N}\right)^{\alpha}\right]\le C\,\frac{(N-\eta_{x})-\eta_{y}}{N}=C\,\frac{\pi_{N}}{N}\;.
\]
\end{proof}
\begin{lem}
\label{lem751}We have that
\[
\sum_{\eta\in\mathcal{J}^{x,\,y}}\mu_{N}(\eta)\,U\left(\frac{\eta_{x}}{N}\right)^{2}\le\left(1+o_{N}(1)\right)N^{-(\alpha-1)}\frac{1}{\kappa_{\star}\,I_{\alpha}\,\Gamma(\alpha)}\;.
\]
\end{lem}

\begin{proof}
By Lemmas \ref{lem710} and \ref{lem750}, we obtain
\[
\left|\sum_{\eta\in\mathcal{J}^{x,\,y}}\mu_{N}(\eta)\left[U\left(\frac{\eta_{x}}{N}\right)^{2}-\frac{a(\eta_{x})^{2}\,a(\eta_{y})^{2}}{N^{4\alpha}\,I_{\alpha}^{2}}\right]\right|\le C\,\frac{\pi_{N}}{N}\sum_{\eta\in\mathcal{J}^{x,\,y}}\mu_{N}(\eta)=o_{N}(1)\,N^{-(\alpha-1)}\;.
\]
By Proposition \ref{e21}, we have
\begin{equation}
\sum_{\eta\in\mathcal{J}^{x,y}}\mu_{N}(\eta)\frac{a(\eta_{x})^{2}\,a(\eta_{y})^{2}}{N^{4\alpha}\,I_{\alpha}^{2}}\le\frac{1+o_{N}(1)}{Z\,I_{\alpha}^{2}\,N^{3\alpha}}\,\sum_{k=0}^{\pi_{N}}\left[\sum_{i=0}^{N-k}i^{\alpha}\,(N-k-i)^{\alpha}\sum_{\zeta\in\mathcal{H}_{k,\,S\setminus\{x,\,y\}}}\frac{m_{\star}^{\zeta}}{a(\zeta)}\right]\;.\label{e7501}
\end{equation}
For $k\le\pi_{N}\ll N$, we have that
\[
\sum_{i=0}^{N-k}i^{\alpha}\,(N-k-i)^{\alpha}=\left(1+o_{N}(1)\right)N^{2\alpha+1}\,\int_{0}^{1}t^{\alpha}(1-t)^{\alpha}dt=\left(1+o_{N}(1)\right)N^{2\alpha+1}I_{\alpha}\;.
\]
Therefore, the right-hand side of \eqref{e7501} is bounded above
by
\[
\left(1+o_{N}(1)\right)\frac{1}{Z\,I_{\alpha}N^{\alpha-1}}\,\sum_{k=0}^{\pi_{N}}\,\sum_{\zeta\in\mathcal{H}_{k,\,S\setminus\{x,\,y\}}}\frac{m_{\star}^{\zeta}}{a(\zeta)}\;.
\]
The proof is completed by Lemma \ref{le72} and the definition \eqref{e20}
of the constant $Z$.
\end{proof}

\subsection{\label{sec74}Construction on tubes}

Throughout this subsection, we fix two points $x,\,y\in S_{\star}$.
Then, we shall define a function $\mathbf{W}_{x,\,y}(\cdot)$ corresponding
to the approximation of the equilibrium potential $\mathbf{h}_{\mathcal{E}^{x},\,\mathcal{E}^{y}}(\cdot)$
on the tube $\mathcal{T}^{x,\,y}$. Indeed, this task has been carried
out in \cite{BL3} for the reversible case, and the definitions as
well as concomitant estimates for the non-reversible case are similar
to those for the reversible case.

Recall that the function $h_{x,\,y}(\cdot)$ represents the equilibrium
potential between two points $x$ and $y$ with respect to the random
walk $X(\cdot)$. Enumerate points of $S$ by $x=z_{1},\,z_{2},\,\cdots,\,z_{\kappa}=y$
in such a manner that
\[
1=h_{x,\,y}(z_{1})\ge h_{x,\,y}(z_{2})\ge\cdots\ge h_{x,\,y}(z_{\kappa})=0\;.
\]
For $\eta\in\mathcal{H}_{N}$ and $1\le i\le\kappa$, define
\[
\eta^{(i)}=\eta_{z_{1}}+\eta_{z_{2}}+\cdots+\eta_{z_{i}}\;.
\]
The function $\mathbf{W}_{x,\,y}=\mathbf{W}_{x,\,y}^{N,\,\epsilon}:\mathcal{T}^{x,\,y}\rightarrow\mathbb{R}$
is defined by
\begin{align*}
\mathbf{W}_{x,\,y}(\eta) & =\sum_{i=1}^{\kappa-1}\left[h_{x,\,y}(z_{i})-h_{x,\,y}(z_{i+1})\right]H\left(\frac{\eta^{(i)}}{N}\right)\;\;;\;\eta\in\mathcal{T}^{x,\,y}\;,
\end{align*}
where $H=H_{\epsilon}$ is the function introduced in \eqref{ck00}.
By \eqref{ck0}, it is obvious that
\begin{equation}
\mathbf{W}_{x,\,y}(\eta)=\begin{cases}
1 & \text{if }\text{\ensuremath{\eta\in\mathcal{D}^{x}\cap\mathcal{T}^{x,\,y}}\;,}\\
0 & \text{if }\eta\in\mathcal{D}^{y}\cap\mathcal{T}^{x,\,y}\;.
\end{cases}\label{bdw1}
\end{equation}
The following lemma is useful.
\begin{lem}
\label{lem75}For $\eta\in\mathcal{J}^{x,\,y}$, we have
\[
\left|H\left(\frac{\eta^{(u)}}{N}\right)-H\left(\frac{\eta^{(u)}\pm1}{N}\right)\right|\le(1+o_{\epsilon}(1))\,\frac{a(\eta_{x})\,a(\eta_{y})}{N^{2\alpha}\,I_{\alpha}}+C\,\frac{\pi_{N}}{N^{2}}\;.
\]
\end{lem}

\begin{proof}
The minus sign on the left-hand side is only considered, as the proof
for the plus sign is the same. By the mean-value theorem, there exists
$\delta\in[0,\,1]$ such that
\[
H\left(\frac{\eta^{(u)}}{N}\right)-H\left(\frac{\eta^{(u)}-1}{N}\right)=\frac{1}{N}H'\left(\frac{\eta^{(u)}-\delta}{N}\right)\;.
\]
By the mean-value theorem again, the fact that
\[
\left|(\eta^{(u)}-\delta)-\eta_{x}\right|\le\max\{\delta,\,\eta_{z_{2}}+\cdots+\eta_{z_{\kappa-1}}\}<\pi_{N}\;,
\]
and Lemma \ref{lemH}, we obtain that
\begin{equation}
\left|H\left(\frac{\eta^{(u)}}{N}\right)-H\left(\frac{\eta^{(u)}-1}{N}\right)\right|\le\frac{1}{N}\,H'\left(\frac{\eta_{x}}{N}\right)+C\,\frac{\pi_{N}}{N^{2}}\le\frac{1+o_{\epsilon}(1)}{N}U\left(\frac{\eta_{x}}{N}\right)+C\frac{\pi_{N}}{N^{2}}\;,\label{eW2}
\end{equation}
where the constant $C$ appeared in this expression is the $L^{\infty}$
norm of $H''$, which depends on $\epsilon$. Finally, by Lemma \ref{lem750}
the term $U(\eta_{x}/N)$ can be replaced with $a(\eta_{x})\,a(\eta_{y})/(N^{2\alpha}\,I_{\alpha})$,
without changing the order of the error term. This completes the proof.
\end{proof}
The neighborhood of a set $\mathcal{A}\subseteq\mathcal{H}_{N}$ is
defined by
\[
\mathcal{\mathcal{\overline{A}}=}\{\eta:\eta=\sigma^{z,\,w}\zeta\;\text{for some\;}\zeta\in\mathcal{A}\;\text{and\;}z,\,w\in S\}\;.
\]
For $\mathbf{f}:\mathcal{\overline{A}\rightarrow\mathbb{R}}$, the
Dirichlet form of $\mathbf{f}$ on $\mathcal{A}$ is defined by
\[
\mathscr{D}_{N}(\mathbf{f};\mathcal{A})=\frac{1}{2}\sum_{\eta\in\mathcal{A}}\sum_{z,\,w\in S}\mu_{N}(\eta)\,g(\eta_{z})\,r(z,\,w)\left[\mathbf{f}(\sigma^{z,\,w}\eta)-\mathbf{f}(\eta)\right]^{2}\;.
\]
The right-hand side can be evaluated since $\mathbf{f}$ is defined
on $\overline{\mathcal{A}}$.
\begin{lem}
\label{lem72}For all $x,\,y\in S_{\star}$, we have
\[
\mathscr{D}_{N}(\mathbf{W}_{x,\,y};\mathcal{\mathcal{J}}_{\textrm{int}}^{x,\,y})\le\left(1+o_{N}(1)+o_{\epsilon}(1)\right)N^{-(1+\alpha)}\,\frac{\textup{cap}_{X}(x,\,y)}{M_{\star}\,\kappa_{\star}\,I_{\alpha}\,\Gamma(\alpha)}\;.
\]
\end{lem}

\begin{proof}
By the definition of $\mathbf{W}_{x,\,y}$, we can write \begin{equation} \label{eW} \begin{aligned}
&\mathbf{W}_{x,\,y}(\eta)-\mathbf{W}_{x,\,y}(\sigma^{z_{i},\,z_{j}}\eta)\\
&\,=\begin{cases}\sum_{u=i}^{j-1}\left[h_{x,\,y}(z_{u})-h_{x,\,y}(z_{u+1})\right]\left[H(\eta^{(u)}/N)-H\bigl((\eta^{(u)}-1)/N\bigr)\right] & \mbox{if }i<j\;,\\\sum_{u=j}^{i-1}\left[h_{x,\,y}(z_{u})-h_{x,\,y}(z_{u+1})\right]\left[H(\eta^{(u)}/N)-H\bigl((\eta^{(u)}-1)/N\bigr)\right] & \mbox{if }i>j\;,\end{cases}\end{aligned} \end{equation}provided that $\eta_{z_{i}}\ge1$. Of course, this quantity is equal
to $0$ if $\eta_{z_{i}}=0$.

The case $i<j$ is first considered. By Lemma \ref{lem75}, and \eqref{eW},
\begin{equation}
\left|\mathbf{W}_{x,\,y}(\eta)-\mathbf{W}_{x,\,y}(\sigma^{z_{i},\,z_{j}}\eta)\right|\le(1+o_{\epsilon}(1))\,\frac{a(\eta_{x})\,a(\eta_{y})}{N^{2\alpha+1}\,I_{\alpha}}(h_{x,\,y}(z_{i})-h_{x,\,y}(z_{j}))+C\frac{\pi_{N}}{N^{2}}\;.\label{eW3}
\end{equation}
Therefore, by Lemma \ref{lem710},\begin{equation} \label{eWW} \begin{aligned}
&\frac{1}{2}\sum_{\eta\in\mathcal{\mathcal{J}}_{\textrm{int }}^{x,\,y}}\mu_{N}(\eta)\,g(\eta_{z_{i}})\,r(z_{i},\,z_{j})\left[\mathbf{W}_{x,\,y}(\eta)-\mathbf{W}_{x,\,y}(\sigma^{z_{i},\,z_{j}}\eta)\right]^{2}\\
&\le\,\frac{1+o_{\epsilon}(1)}{2}\sum_{\eta\in\mathcal{\mathcal{J}}_{\textrm{int}}^{x,\,y},\,\eta_{z_{i}}\ge1}\mu_{N}(\eta)\,g(\eta_{z_{i}})\,r(z_{i},\,z_{j})\left[\frac{a(\eta_{x})\,a(\eta_{y})}{N^{2\alpha+1}\,I_{\alpha}}(h_{x,\,y}(z_{i})-h_{x,\,y}(z_{j}))\right]^{2}\\
&\quad\;+C\sum_{\eta\in\mathcal{\mathcal{J}}_{\textrm{int }}^{x,y}}\mu_{N}(\eta)\frac{\pi_{N}}{N^{3}}\\
&\le\,\frac{1+o_{\epsilon}(1)}{2Z_{N}\,I_{\alpha}^{2}\,N^{3\alpha+2}}\,r(z_{i},\,z_{j})\,(h_{x,\,y}(z_{i})-h_{x,\,y}(z_{j}))^{2}\sum_{\eta\in\mathcal{\mathcal{J}}_{\textrm{int}}^{x,\,y},\,\eta_{z_{i}}\ge1}\frac{m_{\star}^{\eta}}{a(\eta-\omega^{z_{i}})}\,a(\eta_{x})^{2}\,a(\eta_{y})^{2}\\
&\quad\;+o_{N}(1)\,N^{-(\alpha+1)}\;.
\end{aligned} \end{equation}By the change of variable $\eta-\omega^{z_{i}}=\zeta$ and by an argument
similar to that in Lemma \ref{lem751}, it can be verified that
\begin{equation}
\sum_{\eta\in\mathcal{\mathcal{J}}_{\textrm{int}}^{x,\,y},\,\eta_{z_{i}}\ge1}\frac{m_{\star}^{\eta-\omega^{z_{i}}}}{a(\eta-\omega^{z_{i}})}\,a(\eta_{x})^{2}a(\eta_{y})^{2}\le\left(1+o_{N}(1)\right)N^{2\alpha+1}\,I_{\alpha}\,\Gamma(\alpha)^{\kappa_{\star}-2}\prod_{x\in S\setminus S_{\star}}\Gamma_{x}\;.\label{eWW2}
\end{equation}
By inserting \eqref{eWW2} into the line of \eqref{eWW}, and applying
Proposition \ref{e21}, we obtain
\begin{align*}
 & \frac{1}{2}\sum_{\eta\in\mathcal{\mathcal{J}}_{\textrm{int }}^{x,\,y}}\mu_{N}(\eta)\,g(\eta_{z_{i}})\,r(z_{i},\,z_{j})\left[\mathbf{W}_{x,\,y}(\eta)-\mathbf{W}_{x,\,y}(\sigma^{z_{i},\,z_{j}}\eta)\right]^{2}\\
 & \le\left(1+o_{N}(1)+o_{\epsilon}(1)\right)\frac{1}{2M_{\star}\,\kappa_{\star}\,I_{\alpha}\,\Gamma(\alpha)\,N^{\alpha+1}}\,m(z_{i})\,r(z_{i},\,z_{j})\,\left[h_{x,y}(z_{i})-h_{x,y}(z_{j})\right]^{2}\;.
\end{align*}
The case $i>j$ can be similarly treated and the same form of estimate
is obtained. Hence, by summing this result over all $1\le i,\,j\le\kappa$,
and by using
\[
\frac{1}{2}\sum_{i,\,j=1}^{\kappa}m(z_{i})\,r(z_{i},\,z_{j})\,(h_{x,\,y}(z_{i})-h_{x,\,y}(z_{j}))^{2}=\text{\textup{cap}}_{X}(x,\,y)\;,
\]
we complete the proof.
\end{proof}
\begin{rem}[Construction of $\mathbf{W}_{y,\,x}$]
\label{rm74}Suppose that the enumeration $x=z_{1},\,z_{2},\,\cdots,\,z_{\kappa}=y$
is used in the construction of $\mathbf{W}_{x,\,y}$. If all $h_{x,\,y}(z_{i})$,
$1\le i\le\kappa$, are different, then the construction of $\mathbf{W}_{y,\,x}$
is unambiguous as $h_{y,\,x}=1-h_{x,\,y}$, and we obtain $\mathbf{W}_{y,\,x}=1-\mathbf{W}_{x,\,y}$.
By contrast, if $h_{x,\,y}(z_{i})=h_{x,\,y}(z_{i+1})$ for some $i$,
then there are several possibilities in the selection of the enumeration
for the construction of $\mathbf{W}_{y,\,x}$. In this case, the rule
is to select $y=w_{1},\,w_{2},\,\cdots,\,w_{\kappa}=x$ for the enumeration,
where $w_{i}=z_{\kappa+1-i}$, $1\le i\le\kappa$. It can be thereby
verified that $\mathbf{W}_{y,\,x}=1-\mathbf{W}_{x,\,y}$; therefore,
\[
\mathscr{D}_{N}(\mathbf{W}_{x,\,y};\mathcal{\mathcal{J}}_{\textrm{int }}^{x,\,y})=\mathscr{D}_{N}(\mathbf{W}_{y,\,x};\mathcal{\mathcal{J}}_{\textrm{int }}^{y,\,x})\;.
\]
\end{rem}

\subsubsection*{Construction for adjoint dynamics}

The function $\mathbf{W}_{x,\,y}^{*}$ on $\mathcal{T}^{x,\,y}$ is
similarly defined. Recall $h_{x,\,y}^{*}$ the equilibrium potential
between $x$ and $y$ with respect to the adjoint random walk $X^{*}(\cdot)$,
and enumerate points of $S$ by $x=z_{1}^{*},\,\,\cdots,\,z_{\kappa}^{*}=y$
in such a manner that
\[
1=h_{x,\,y}^{*}(z_{1}^{*})\ge h_{x,\,y}^{*}(z_{2}^{*})\ge\cdots\ge h_{x,y\,}^{*}(z_{\kappa}^{*})=0\;.
\]
Then, let
\[
\mathbf{W}_{x,\,y}^{*}(\eta)=\sum_{i=1}^{\kappa-1}\left[h_{x,\,y}^{*}(z_{i}^{*})-h_{x,\,y}^{*}(z_{i+1}^{*})\right]H\left(\frac{\eta^{(i)}}{N}\right)\;.
\]
This function $\mathbf{W}_{x,\,y}^{*}$ also satisfies the property
\eqref{bdw1}, and the following variant of Lemma \ref{lem72}, whose
proof is identical to that of Lemma \ref{lem72}.
\begin{lem}
\label{lem73}For all $x,\,y\in S_{\star}$, we have
\[
\mathscr{D}_{N}(\mathbf{W}_{x,\,y}^{*};\mathcal{\mathcal{J}}_{\textup{\textrm{int}}}^{x,\,y})\le\left(1+o_{N}(1)+o_{\epsilon}(1)\right)N^{-(1+\alpha)}\,\frac{\textup{cap}_{X}(x,\,y)}{M_{\star}\,\kappa_{\star}\,I_{\alpha}\,\Gamma(\alpha)}\;.
\]
\end{lem}

The rule for selecting the enumeration corresponding to $\mathbf{W}_{y,\,x}^{*}$
is the same as that in Remark \ref{rm74}; hence, $\mathbf{W}_{y,\,x}^{*}=1-\mathbf{W}_{x,\,y}^{*}$.

\subsection{\label{sec75}Global construction of $\mathbf{V}_{A,\,B}$, $\mathbf{V}_{A,\,B}^{*}$
and proof of Proposition \ref{p63}}

For two disjoint non-empty subsets $A,\,B$ of $S_{\star}$, let the
function $\mathbf{V}_{A,\,B}:\mathcal{H}_{N}\rightarrow\mathbb{R}$
be defined as follows:
\[
\mathbf{V}_{A,\,B}(\eta)=\begin{cases}
\mathfrak{h}_{A,\,B}(x) & \text{if }\eta\in\mathcal{D}^{x},\,x\in S_{\star}\;,\\
\mathfrak{h}_{A,\,B}(y)+\left[\mathfrak{h}_{A,\,B}(x)-\mathfrak{h}_{A,\,B}(y)\right]\mathbf{W}_{x,\,y}(\eta) & \text{if }\eta\in\mathcal{J}^{x,\,y}\;,x,\,y\in S_{\star}\;,\\
0 & \text{if }\eta\in\mathcal{G}^{c}\;.
\end{cases}
\]
In this expression, the definition on $\mathcal{J}^{x,\,y}$ does
not depend on the order of $x$ and $y$, owing to Remark \ref{rm74},
in the sense that
\[
\mathfrak{h}_{A,\,B}(y)+\left[\mathfrak{h}_{A,\,B}(x)-\mathfrak{h}_{A,\,B}(y)\right]\mathbf{W}_{x,\,y}(\eta)=\mathfrak{h}_{A,\,B}(x)+\left[\mathfrak{h}_{A,\,B}(y)-\mathfrak{h}_{A,\,B}(x)\right]\mathbf{W}_{y,\,x}(\eta)\;.
\]
The function $\mathbf{V}_{A,\,B}^{*}(\cdot)$ is defined by replacing
$\mathbf{W}_{x,\,y}$ in the definition of $\mathbf{V}_{A,\,B}$ with
$\mathbf{W}_{x,\,y}^{*}$. Then, it is immediate from the definition
that $\mathbf{V}_{A,\,B}$ and $\mathbf{V}_{A,\,B}^{*}$ satisfy part
(1) of Proposition \ref{p63}. Hence, to complete the proof of Proposition
\ref{p63}, it suffices to prove part (2). This will be verified only
for the function $\mathbf{V}_{A,\,B}$, as the proof for $\mathbf{V}_{A,\,B}^{*}$
is essentially the same.
\begin{lem}
For two disjoint non-empty subsets $A,\,B$ of $S_{\star}$, we have
\[
\mathscr{D}_{N}(\mathbf{V}_{A,\,B})\le\left(1+o_{N}(1)+o_{\epsilon}(1)\right)N^{-(1+\alpha)}\,\textup{cap}_{Y}(A,\,B)\;.
\]
\end{lem}

\begin{proof}
By \eqref{decompg} and \eqref{decompgc} the Dirichlet form $\mathscr{D}_{N}(\mathbf{V}_{A,\,B})$
can be decomposed as\begin{equation} \label{e741} \begin{aligned}
&\sum_{x\in S_{\star}}\mathscr{D}_{N}(\mathbf{V}_{A,\,B};\mathcal{D}_{\textup{int}}^{x})+\sum_{\{x,\,y\}\subset S_{\star}}\mathscr{D}_{N}(\mathbf{V}_{A,\,B};\mathcal{J}_{\textup{int}}^{x,\,y})+\mathscr{D}_{N}(\mathbf{V}_{A,\,B};\partial^{\textrm{\,in}}\mathcal{G}))\\
&\;+\mathscr{D}_{N}(\mathbf{V}_{A,\,B};(\mathcal{G}^{c})_{\textrm{int}})+\mathscr{D}_{N}(\mathbf{V}_{A,\,B};\partial^{\textrm{\,out}}\mathcal{G}))\;. \end{aligned} \end{equation} It is first observed that for all $z,\,w\in S$, we have $\mathbf{V}_{A,\,B}(\sigma^{z,\,w}\eta)-\mathbf{V}_{A,\,B}(\eta)=0$
for all $\eta\in\mathcal{D}_{\textup{int}}^{x}$, $x\in S_{\star}$,
and for all $\eta\in(\mathcal{G}^{c})_{\textrm{int}}$. Therefore,
\begin{equation}
\sum_{x\in S_{\star}}\mathscr{D}_{N}(\mathbf{V}_{A,\,B};\mathcal{D}_{\textup{int}}^{x})=0\;\;\text{and\;\;}\mathscr{D}_{N}(\mathbf{V}_{A,\,B};(\mathcal{G}^{c})_{\textrm{int}})=0\;.\label{e742}
\end{equation}
Moreover, as $\mathbf{V}_{A,\,B}(\eta)\in[0,\,1]$, and $g(\cdot)$,
$r(\cdot,\,\cdot)$ are bounded, it holds that
\[
\mathscr{D}_{N}(\mathbf{V}_{A,\,B};\partial^{\textrm{\,in}}\mathcal{G}))\le C\,\mu_{N}(\partial^{\textrm{\,in}}\mathcal{G})\;\;\text{and\;\;}\mathscr{D}_{N}(\mathbf{V}_{A,\,B};\partial^{\textrm{\,out}}\mathcal{G}))\le C\,\mu_{N}(\partial^{\textrm{\,out}}\mathcal{G})
\]
for some constant $C$. Hence, by Lemma \ref{lem71},
\begin{equation}
\mathscr{D}_{N}(\mathbf{V}_{A,\,B};\partial^{\textrm{\,in}}\mathcal{G}))+\mathscr{D}_{N}(\mathbf{V}_{A,\,B};\partial^{\textrm{\,out}}\mathcal{G}))=o_{N}(1)\,N^{-(1+\alpha)}\;.\label{e743}
\end{equation}
Finally, by Lemma \ref{lem72}, for $x,\,y\in S_{\star}$,
\begin{align}
 & \mathscr{D}_{N}(\mathbf{V}_{A,\,B};\mathcal{J}_{\textup{int}}^{x,\,y})\nonumber \\
 & =\left[\mathfrak{h}_{A,\,B}(y)-\mathfrak{h}_{A,\,B}(x)\right]^{2}\mathscr{D}_{N}(\mathbf{W}_{x,\,y};\mathcal{J}_{\textup{int}}^{x,\,y})\label{e744}\\
 & \le\left(1+o_{N}(1)+o_{\epsilon}(1)\right)\,N^{-(1+\alpha)}\left[\mathfrak{h}_{A,\,B}(y)-\mathfrak{h}_{A,\,B}(x)\right]^{2}\frac{\textup{cap}_{X}(x,\,y)}{M_{\star}\,\kappa_{\star}\,I_{\alpha}\,\Gamma(\alpha)}\;.\nonumber
\end{align}
The proof is completed by combining \eqref{e741}, \eqref{e742}, \eqref{e743}, \eqref{e744},
and the fact that
\[
\sum_{\{x,\,y\}\subset S_{\star}}\frac{\textup{cap}_{X}(x,\,y)}{M_{\star}\,\kappa_{\star}\,I_{\alpha}\,\Gamma(\alpha)}\left[\mathfrak{h}_{A,\,B}(y)-\mathfrak{h}_{A,\,B}(x)\right]^{2}=\mathfrak{D}_{Y}(\mathfrak{h}_{A,\,B})=\textup{cap}_{Y}(A,\,B)\;.
\]
\end{proof}

\section{\label{sec8}Correction procedure for test flows}

In this section, for two disjoint non-empty subsets $A$ and $B$
of $S_{\star}$, we shall construct suitable approximations of the
optimal flows $\Phi_{\mathbf{h}_{\mathcal{E}(A),\,\mathcal{E}(B)}}^{*}$
and $\Phi_{\mathbf{h}_{\mathcal{E}(A),\,\mathcal{E}(B)}^{*}}$, denoted
by $\Phi_{A,\,B}$ and $\Phi_{A,\,B}^{*}$, respectively. The focus
is only on the former because the construction of the latter is entirely
parallel. In particular, it is demonstrated that $\Phi_{A,\,B}$ satisfies
four conditions presented in Proposition \ref{p64}. Henceforth, we
fix two disjoint non-empty subsets $A$ and $B$ of $S_{\star}$

\subsection{Analysis of the flow $\Phi_{\mathbf{V}_{A,\,B}}^{*}$}

Recall the definition of $\Phi_{\mathbf{V}_{A,\,B}}^{*}$ from \eqref{flow}.
By (1) of Proposition \ref{pflow}, we have
\begin{equation}
(\textup{div}\,\Phi_{\mathbf{V}_{A,\,B}}^{*})(\eta)=\sum_{z,\,w\in S}\mu_{N}(\eta)g(\eta_{z})r(z,\,w)\left[\mathbf{V}_{A,\,B}(\eta)-\mathbf{V}_{A,\,B}(\sigma^{z,\,w}\eta)\right]\;.\label{div}
\end{equation}
This expression can be used to derive several basic facts about the
flow $\Phi_{\mathbf{V}_{A,\,B}}^{*}$.
\begin{prop}
\label{p81}The flow $\Phi_{\mathbf{V}_{A,\,B}}^{*}$ has the following
properties.
\begin{enumerate}
\item The flow is divergence-free on $\mathcal{D}_{\textup{\textrm{int}}}^{x}$,
$x\in S_{\star}$, and on $(\mathcal{G}^{c})_{\textup{\textrm{int}}}$,
i.e.,
\[
(\textup{div }\Phi_{\mathbf{V}_{A,\,B}}^{*})(\eta)=0\text{\;\;for all \;}\eta\in\bigcup_{x\in S_{\star}}\mathcal{D}_{\textup{\textrm{int}}}^{x}\;\;\text{and for all }\eta\in(\mathcal{G}^{c})_{\textrm{\textup{\textrm{int}}}}\;.
\]
\item The divergence on boundaries is negligible in the sense that
\[
\sum_{\eta\in\partial^{\textrm{\textup{\,in}}}\mathcal{G}}\left|(\textup{div }\Phi_{\mathbf{V}_{A,\,B}}^{*})(\eta)\right|+\sum_{\eta\in\partial^{\text{\textup{\,out}}}\mathcal{G}}\left|(\textup{div }\Phi_{\mathbf{V}_{A,\,B}}^{*})(\eta)\right|=o_{N}(1)\,N^{-(1+\alpha)}\;.
\]
\end{enumerate}
\end{prop}

\begin{proof}
Part (1) is obvious, as for all $z,\,w\in S$, we have $\mathbf{V}_{A,\,B}(\eta)=\mathbf{V}_{A,\,B}(\sigma^{z,\,w}\eta)$
for all $\eta$ belonging to $\mathcal{D}_{\textrm{int}}^{x}$, $x\in S_{\star}$,
or $(\mathcal{G}^{c})_{\textrm{int}}$. For part (2), it suffices
to observe from \eqref{div} that $|(\textup{div }\Phi_{\mathbf{V}_{A,\,B}}^{*})(\eta)|\le C\mu_{N}(\eta)$
and use Lemma \ref{lem71}.
\end{proof}
The previous proposition reveals a crucial drawback regarding the
flow $\Phi_{\mathbf{V}_{A,\,B}}^{*}$, namely,
\[
(\textup{div }\Phi_{\mathbf{V}_{A,\,B}}^{*})(\mathcal{E}(A))=(\textup{div }\Phi_{\mathbf{V}_{A,\,B}}^{*})(\mathcal{E}(B))=0\;,
\]
as one of the main requirement of the test flow in the application
of Theorem \ref{t04} is that the total divergence on $\mathcal{E}(A)$
of the test flow is approximately $N^{-(1+\alpha)}\,\textup{cap}_{Y}(A,\,B)$.
In addition, one can easily verify that the divergence of $\Phi_{\mathbf{V}_{A,\,B}}^{*}$
on $\mathcal{J}^{x,\,y}$, $x,\,y\in S_{\star}$, is not negligible,
i.e., not of order $o_{N}(1)\,N^{-(1+\alpha)}$. This is the second
serious problem, as one would hope that the divergence on $\Delta_{N}$
is of order $o_{N}(1)\,N^{-(1+\alpha)}$, but $\mathcal{J}^{x,\,y}\subset\Delta_{N}$.

These two defects are closely related. The non-negligible divergences
on the saddle tube $\mathcal{J}^{x,\,y}$ should be carefully sent
to the valleys $\mathcal{E}^{x}$ and $\mathcal{E}^{y}$, so that
after this correction procedure, the divergence on saddle tubes is
negligible whereas the divergence on valleys $\mathcal{E}(A)$ is
close to the desired values.

Under the assumption \eqref{ass1}, this correction procedure is first
carried out for each tube $\mathcal{T}^{x,\,y}$ in Section \ref{sec82}
and then globally in Section \ref{sec83}. The proof of Proposition
\ref{p64} is also presented in Section \ref{sec83}. Several technical
computations are summarized in Section \ref{sec84}. The special assumption
\eqref{ass1} is imposed to simplify cumbersome notations and does
not affect the validity of the main arguments. In Section \ref{sec85}
the general case without this special assumption is considered.

\subsection{\label{sec82}Correction on tube: special case}

In Sections \ref{sec82}, \ref{sec83}, and \ref{sec84}, it will be
assumed that
\begin{equation}
r(u,\,v)>0\;\;\text{for all\;}u,\,v\in S\;.\label{ass1}
\end{equation}
The general result without this redundant assumption will be explained
in Section \ref{sec85}, as mentioned earlier. We fix two points $x,\,y\in S_{\star}$
throughout this subsection in order to focus on the construction on
the tube $\mathcal{T}^{x,\,y}$.

\subsubsection*{Localization of the correction procedure}

Recall the enumeration $x=z_{1},\,z_{2},\,\cdots,\,z_{\kappa}=y$
and the function $\mathbf{W}_{x,\,y}$ from Section \ref{sec74}.
As the function $\mathbf{W}_{x,\,y}$ is defined only on the tube
$\mathcal{T}^{x,\,y}$, the flow $\Phi_{\mathbf{W}_{x,\,y}}^{*}$
cannot be defined in the usual manner. Hence, let us first extend
$\mathbf{W}_{x,\,y}$ to a function on $\mathcal{H}_{N}$ by
\[
\mathbf{\widetilde{W}}_{x,\,y}(\eta)=\mathbf{W}_{x,\,y}(\eta)\,\mathbf{1}\left\{ \eta\in\mathcal{T}^{x,\,y}\right\} \;.
\]
Then, the flow $\Phi_{\mathbf{\widetilde{W}}_{x,\,y}}^{*}$ can be
defined. With a slight abuse of notation, this flow can be written
as $\Phi_{\mathbf{W}_{x,\,y}}^{*}$. Then, by the definition of $\mathbf{V}_{A,\,B}$,
the flow $\Phi_{\mathbf{V}_{A,\,B}}^{*}$ satisfies
\begin{equation}
\Phi_{\mathbf{V}_{A,\,B}}^{*}(\eta,\,\zeta)=\left[\mathfrak{h}_{A,\,B}(x)-\mathfrak{h}_{A,\,B}(y)\right]\Phi_{\mathbf{W}_{x,\,y}}^{*}(\eta,\,\zeta)\;\;\text{for all\;}\eta,\,\zeta\in\mathcal{T}^{x,\,y}\;.\label{e91}
\end{equation}
It should be noted that this relation is valid not only for $\eta,\,\zeta\in\mathcal{J}^{x,\,y}$
but also for $\eta,\,\zeta\in\mathcal{T}^{x,\,y}$, as both sides
are equal to $0$ if either $\eta\in\mathcal{T}^{x,\,y}$ or $\zeta\in\mathcal{T}^{x,\,y}$
does not belong to $\mathcal{J}^{x,\,y}$. The interior of the tube
is defined by
\[
\mathcal{T}_{\textrm{int}}^{x,\,y}=\{\eta:\eta_{x}+\eta_{y}>N(1-\epsilon)\}\;.
\]
Then, by \eqref{e91},
\begin{equation}
(\textup{div\,}\Phi_{\mathbf{V}_{A,\,B}}^{*})(\eta)=\left[\mathfrak{h}_{A,\,B}(x)-\mathfrak{h}_{A,\,B}(y)\right](\textup{div\,}\Phi_{\mathbf{W}_{x,\,y}}^{*})(\eta)\;\;\text{for all }\eta\in\mathcal{T}_{\textrm{int}}^{x,\,y}\;.\label{num}
\end{equation}
Therefore, the correction procedure for the divergence of the flow
$\Phi_{\mathbf{V}_{A,\,B}}^{*}$ on the tube $\mathcal{T}_{\textrm{int}}^{x,\,y}$
is reduced to that of $\Phi_{\mathbf{W}_{x,\,y}}^{*}$ on $\mathcal{T}_{\textrm{int}}^{x,\,y}$.

The divergence of the flow $\Phi_{\mathbf{W}_{x,\,y}}^{*}$ on $\mathcal{T}_{\textrm{int}}^{x,\,y}$
is now investigated. By \eqref{u1}, the divergence $\textup{div\,}\Phi_{\mathbf{W}_{x,\,y}}^{*}$
at $\eta\in\mathcal{T}_{\textrm{int}}^{x,\,y}$ can be written as\begin{equation} \label{e321} \begin{aligned}
(\textup{div\,}\Phi_{\mathbf{W}_{x,\,y}}^{*})(\eta)&=\sum_{z,\,w\in S}\mu_{N}(\eta)\,g(\eta_{z})\,r(z,\,w)\left[\mathbf{W}_{x,\,y}(\eta)-\mathbf{W}_{x,\,y}(\sigma^{z,\,w}\eta)\right]\;\\&=a_{N}\sum_{z,\,w\in S}\mu_{N-1}(\eta-\omega^{z})\,m(z)\,r(z,\,w)\left[\mathbf{W}_{x,\,y}(\eta)-\mathbf{W}_{x,\,y}(\sigma^{z,\,w}\eta)\right]\;.
\end{aligned} \end{equation}By \eqref{eW}, the last expression can be written as
\begin{equation}
(\textup{div\,}\Phi_{\mathbf{W}_{x,\,y}}^{*})(\eta)=a_{N}\sum_{i=1}^{\kappa}\mu_{N-1}(\eta-\omega^{z_{i}})\,m(z_{i})\mathbf{\,B}(\eta;z_{i})\mathbf{\,1}\{\eta_{z_{i}}\ge1\}\;,\label{ed1}
\end{equation}
where \begin{equation} \label{ed2} \begin{aligned}
\mathbf{B}(\eta;z_{i})=&\,\sum_{j:i<j}r(z_{i},\,z_{j})\sum_{u=i}^{j-1}\left[h_{x,\,y}(z_{u})-h_{x,\,y}(z_{u+1})\right]\left[H\left(\frac{\eta^{(u)}}{N}\right)-H\left(\frac{\eta^{(u)}-1}{N}\right)\right]\\&\;+\sum_{j:i>j}r(z_{i},\,z_{j})\sum_{u=j}^{i-1}\left[h_{x,\,y}(z_{u})-h_{x,\,y}(z_{u+1})\right]\left[H\left(\frac{\eta^{(u)}}{N}\right)-H\left(\frac{\eta^{(u)}+1}{N}\right)\right]\;.\end{aligned} \end{equation} Estimates on $\mathbf{B}(\cdot,\cdot)$ will now be provided.
\begin{lem}
\label{lem84}For $\eta\in\mathcal{T}^{x,\,y}$, there exists a constant
$C\ge0$ such that
\begin{align*}
 & \left|\mathbf{B}(\eta;z_{i})\right|\le C\frac{\pi_{N}}{N^{2}}\;\;;\;2\le i\le\kappa-1\;,\\
 & \left|\mathbf{B}(\eta;z_{1})-\frac{1}{N}\,H'\left(\frac{\eta_{x}}{N}\right)\frac{1}{M_{\star}}\,\textup{cap}_{X}(x,\,y)\right|\le C\,\frac{\pi_{N}}{N^{2}}\;,\;\text{and}\\
 & \left|\mathbf{B}(\eta;z_{\kappa})+\frac{1}{N}\,H'\left(\frac{\eta_{x}}{N}\right)\frac{1}{M_{\star}}\,\textup{cap}_{X}(x,\,y)\right|\le C\,\frac{\pi_{N}}{N^{2}}\;.
\end{align*}
In particular, the constant $C$ can be chosen to be $0$ if $\eta\notin\mathcal{J}^{x,\,y}$.
\end{lem}

\begin{proof}
The argument presented in the proof of Lemma \ref{lem75} based on
the mean-value theorem yields
\[
\left|\left\{ H\left(\frac{\eta^{(u)}}{N}\right)-H\left(\frac{\eta^{(u)}\pm1}{N}\right)\right\} \mp\frac{1}{N}H'\left(\frac{\eta_{x}}{N}\right)\right|\le C\,\frac{\pi_{N}}{N^{2}}\;.
\]
Applying this bound to each term in $\mathbf{B}(\eta;z_{i})$ and
using \eqref{eqp1} and \eqref{eqp11} provide the desired estimates.
For $\eta\in\mathcal{J}^{x,\,y}$, both $\mathbf{B}(\eta;z_{i})$,
$1\le i\le\kappa$, and $H'(\eta_{x}/N)$ are equal to $0$; therefore
we can select $C=0$.
\end{proof}
At first glance, these estimates imply that the right-hand side of
\eqref{ed1} is \textit{small}; hence, the divergence of $\Phi_{\mathbf{W}_{x,\,y}}^{*}$
on $\mathcal{J}^{x,\,y}$ is small. As the flow $\Phi_{\mathbf{h}_{\mathcal{E}^{x},\mathcal{\,E}^{y}}}^{*}$
is divergence-free on $\mathcal{J}^{x,\,y}$, this heuristic observation
supports the claim that $\mathbf{W}_{x,\,y}$ approximates the equilibrium
potential $\mathbf{h}_{\mathcal{E}^{x},\,\mathcal{E}^{y}}$ on the
saddle tube. However, the word \textit{small} used here is not quite
correct in some sense. To be more precise, these estimates along with
the expression \eqref{ed1} imply that $(\textup{div\,}\Phi_{\mathbf{W}_{x,\,y}}^{*})(\eta)$
is of order $\mu_{N}(\eta)\,(\pi_{N}/N^{2})$ for $\eta\in\mathcal{J}^{x,\,y}$.
Therefore, in view of Lemma \ref{lem710}, the divergence on $\mathcal{J}^{x,\,y}$
is not negligible, i.e., not of order $o_{N}(1)\,N^{-(1+\alpha)}$.

The essence of the correction procedure hereafter presented is to
send these small, but not negligible divergences on $\mathcal{J}^{x,\,y}$
to $\mathcal{E}^{x}$ and $\mathcal{E}^{y}$, without excessively
perturbing the flow $\Phi_{\mathbf{W}_{x,\,y}}^{*}$, in the sense
of the flow norm. This procedure is carried out through the correction
flow $\chi_{x,\,y}$ defined below.

\subsubsection*{Correction flow $\chi_{x,\,y}$ and corrected flow $\Phi_{x,\,y}$.}

Two subsets of $\mathcal{H}_{N}$ are now defined by
\[
\mathcal{V}^{x}=\{\eta\in\mathcal{T}_{\textrm{ }}^{x,\,y}:\eta_{y}=0\}\;\;\text{and\;\;}\mathcal{V}^{y}=\{\eta\in\mathcal{T}^{x,\,y}:\eta_{x}=0\}\;,
\]
so that $\mathcal{V}^{x}\subset\mathcal{D}^{x}$ and $\mathcal{V}^{y}\subset\mathcal{D}^{y}$.
For $\eta\in\mathcal{H}_{N}$, let $\widehat{\eta}\in\mathbb{N}^{S\setminus\{x,\,y\}}$
be the configuration on $S\setminus\{x,\,y\}$ obtained from $\eta$
by neglecting two sites $x$ and $y$. Then, let
\begin{equation}
\mathbf{C}(\eta):=\frac{\textup{cap}_{X}(x,\,y)}{N^{\alpha+1\,}Z_{N}\,M_{\star}\,I_{\alpha}}\frac{m_{\star}^{\eta}}{a(\widehat{\eta})}=\frac{\textup{cap}_{X}(x,\,y)\,m_{\star}^{\eta}}{N^{\alpha+1\,}Z_{N}\,M_{\star}\,I_{\alpha}}\frac{a(\eta_{x})\,a(\eta_{y})}{a(\eta)}\;.\label{sj3}
\end{equation}
The correction flow $\chi_{x,y}$ will now be defined. Recall $\mathbf{B}(\eta;z_{i})$
from \eqref{ed2}.

We first define a flow $\chi_{x,\,y}^{(1)}$. If $\eta\in\mathcal{T}^{x,\,y}$
satisfies $\eta_{z_{i}}\ge1$

for some $2\le i\le\kappa-1$ and $\zeta=\sigma^{z_{i},\,x}\eta$
or $\sigma^{z_{i},\,y}\eta$ , then
\[
\chi_{x,\,y}^{(1)}(\eta,\,\zeta)=-\chi_{x,\,y}^{(1)}(\zeta,\,\eta)=-\frac{1}{2}a_{N}\mu_{N}(\eta-\omega^{z_{i}})\,m(z_{i})\,\mathbf{B}(\eta;z_{i})\;.
\]
 Otherwise, $\chi_{x,\,y}^{(1)}(\eta,\,\zeta)=0$.

Now we define a flow $\chi_{x,\,y}^{(2)}$. If $\eta\in\mathcal{T}^{x,\,y}$
satisfies $\eta_{y}\ge1$ and $\zeta=\sigma^{y,\,x}\eta$, then
\begin{align*}
 & \chi_{x,\,y}^{(2)}(\eta,\,\zeta)=-\chi_{x,\,y}^{(2)}(\zeta,\,\eta)\\
 & \;=\frac{1}{2}a_{N}\,\mu_{N-1}(\eta-\omega^{z_{\kappa}})\left[m(x)\,\mathbf{B}(\sigma^{z_{\kappa},\,z_{1}}\eta;z_{1})-m(y)\,\mathbf{B}(\eta;z_{\kappa})\right]-\mathbf{C}(\eta)\;.
\end{align*}
Otherwise,t $\chi_{x,\,y}^{(2)}(\eta,\,\zeta)=0$. Finally,
\[
\chi_{x,\,y}=\chi_{x,\,y}^{(1)}+\chi_{x,\,y}^{(2)}\;.
\]
It should be noted here that $\chi_{x,\,y}(\eta,\,\zeta)=0$ unless
$\eta,\,\zeta\in\mathcal{T}^{x,\,y}$.
\begin{prop}
\label{p82}The flow $\chi_{x,\,y}$ is negligible in the sense that
\[
||\chi_{x,\,y}||^{2}=\left(o_{N}(1)+o_{\epsilon}(1)\right)N^{-(1+\alpha)}\;.
\]
\end{prop}

\begin{proof}
The estimate for $\chi_{x,\,y}^{(1)}$ is not complicated. By Lemma
\ref{lem84}, it is seen that
\[
|\chi_{x,\,y}^{(1)}(\eta,\,\zeta)|\le C\,\mu_{N}(\eta)\,\frac{\pi_{N}}{N^{2}}\;\;\text{for all }\eta\in\mathcal{J}^{x,\,y}\;,
\]
and that $\chi_{x,\,y}^{(1)}(\eta,\,\zeta)=0$ for all $\eta\notin\mathcal{J}^{x,\,y}$.
As $c_{N}^{s}(\eta,\,\zeta)\ge C\mu_{N}(\eta)$ for all $(\eta,\,\zeta)\in\mathcal{H}_{N}^{\otimes}$,
we have
\begin{equation}
\left\Vert \chi_{x,\,y}^{(1)}\right\Vert ^{2}=\frac{1}{2}\sum_{\eta\in\mathcal{H}_{N}}\,\sum_{\zeta:\zeta\sim\eta}\frac{\chi_{x,\,y}^{(1)}(\eta,\,\zeta)^{2}}{c_{N}^{s}(\eta,\,\zeta)}\le C\sum_{\eta\in\mathcal{J}^{x,\,y}}\mu_{N}(\eta)\,\frac{\pi_{N}^{2}}{N^{4}}\;.\label{e111}
\end{equation}
Hence, by Lemma \ref{lem710},
\begin{equation}
\left\Vert \chi_{x,\,y}^{(1)}\right\Vert ^{2}=o_{N}(1)\,N^{-(1+\alpha)}\;.\label{ecc1}
\end{equation}

The estimate for $\chi_{x,\,y}^{(2)}$ is rather complicated. For
$\eta\in\mathcal{T}^{x,\,y}$, by Lemma \ref{lem84},\begin{equation} \label{sj1} \begin{aligned}
&\chi_{x,\,y}^{(2)}(\eta,\,\sigma^{y,\,x}\eta)\\
&\;=\mu_{N}(\eta)\,\frac{1}{N}\frac{\textup{cap}_{X}(x,\,y)}{M_{\star}}\left[H'\left(\frac{\eta_{x}}{N}\right)\,g(\eta_{x})-\frac{a(\eta_{x})\,a(\eta_{y})}{N^{2\alpha}\,I_{\alpha}}\right]+\mu_{N}(\eta)\,\frac{o_{N}(1)}{N}\mathbf{\,1}\{\eta\in\mathcal{J}^{x,y}\}\;.
\end{aligned} \end{equation}Since $H'(\eta_{x}/N)=0$ for $\eta\notin\mathcal{J}^{x,\,y}$, and
since $g(\eta_{x})=1+o_{N}(1)$ for $\eta\in\mathcal{J}^{x,\,y}$,
we have
\[
H'\left(\frac{\eta_{x}}{N}\right)\,g(\eta_{x})=H'\left(\frac{\eta_{x}}{N}\right)\,+o_{N}(1)\mathbf{\,1}\{\eta\in\mathcal{J}^{x,y}\}\;.
\]
Therefore, the term $g(\eta_{x})$ in \eqref{sj1} can be replaced
with $1$, without changing the type of the error term. Thus, for
$\eta\in\mathcal{T}^{x,\,y}$, we have \begin{equation} \label{e234} \begin{aligned}
&\chi_{x,\,y}^{(2)}(\eta,\,\sigma^{y,\,x}\eta)\\
&\;=\mu_{N}(\eta)\,\frac{1}{N}\frac{\textup{cap}_{X}(x,\,y)}{M_{\star}}\left[H'\left(\frac{\eta_{x}}{N}\right)-\frac{a(\eta_{x})\,a(\eta_{y})}{N^{2\alpha}\,I_{\alpha}}\right]+\mu_{N}(\eta)\,\frac{\beta_{N}}{N}\mathbf{1}\{\eta\in\mathcal{J}^{x,\,y}\}\;,
\end{aligned} \end{equation} where $\beta_{N}=o_{N}(1)$.

The flow norm of $\chi_{x,\,y}^{(2)}$ is now considered by decomposing
it into three flows. The first flow is defined by
\[
\chi_{x,\,y}^{(2\textrm{a})}(\eta,\,\sigma^{y,\,x}\eta)=\chi_{x,\,y}^{(2)}(\eta,\,\sigma^{y,\,x}\eta)\,\mathbf{1}\left\{ \eta\in\mathcal{J}^{x,\,y}\;,\eta_{x}\in[\,N\epsilon^{1/2},\,N(1-\epsilon^{1/2})]\right\} \;.
\]
If $\eta\in\mathcal{T}^{x,\,y}$ and $\eta_{x}\in[\,N\epsilon^{1/2},\,N(1-\epsilon^{1/2})]$,
then by Lemmas \ref{lem750} and \ref{lemH1}
\begin{align*}
\left|\chi_{x,\,y}^{(2\textrm{a})}(\eta,\,\sigma^{y,\,x}\eta)\right| & \le\mu_{N}(\eta)\,\frac{1}{N}\frac{\textup{cap}_{X}(x,\,y)}{M_{\star}}\left[H'\left(\frac{\eta_{x}}{N}\right)-U\left(\frac{\eta_{x}}{N}\right)\right]+\mu_{N}(\eta)\,\frac{o_{N}(1)}{N}\\
 & \le\frac{1}{N}\left(o_{N}(1)+o_{\epsilon}(1)\,U\left(\frac{\eta_{x}}{N}\right)\,\right)\,\mu_{N}(\eta)\;.
\end{align*}
Thus, by the same argument as in \eqref{e111},
\[
\left\Vert \chi_{x,\,y}^{(2\textrm{a})}\right\Vert ^{2}\le\,\frac{1}{N^{2}}\sum_{\eta\in\mathcal{J}^{x,\,y}}\left(o_{N}(1)+o_{\epsilon}(1)U\left(\frac{\eta_{x}}{N}\right)^{2}\,\right)\mu_{N}(\eta).
\]
Hence, by Lemmas \ref{lem710} and \ref{lem751},
\begin{equation}
\left\Vert \chi_{x,\,y}^{(2\textrm{a})}\right\Vert ^{2}\le\left(o_{N}(1)+o_{\epsilon}(1)\right)\,N^{-(\alpha+1)}\;.\label{ecc2}
\end{equation}
For $\eta\in\mathcal{T}^{x,\,y}$ with $\eta_{x}\notin[\,N\epsilon^{1/2},\,N(1-\epsilon^{1/2})]$,
the second flow is defined by,
\[
\chi_{x,\,y}^{(2\textrm{b})}(\eta,\,\sigma^{y,\,x}\eta)=\mu_{N}(\eta)\,\frac{1}{N}\frac{\textup{cap}_{X}(x,\,y)}{M_{\star}}\left[H'\left(\frac{\eta_{x}}{N}\right)-\frac{a(\eta_{x})\,a(\eta_{y})}{N^{2\alpha}\,I_{\alpha}}\right]\;,
\]
and $\chi_{x,\,y}^{(2\textrm{b})}\equiv0$ otherwise. By the trivial
bound
\[
\left|H'\left(\frac{\eta_{x}}{N}\right)-\frac{a(\eta_{x})\,a(\eta_{y})}{N^{2\alpha}\,I_{\alpha}}\right|\le H'\left(\frac{\eta_{x}}{N}\right)+\frac{a(\eta_{x})\,a(\eta_{y})}{N^{2\alpha}\,I_{\alpha}}\le H'\left(\frac{\eta_{x}}{N}\right)+U\left(\frac{\eta_{x}}{N}\right)
\]
and by Lemma \ref{lemH}, we obtain
\[
\left|\chi_{x,\,y}^{(2\textrm{b})}(\eta,\,\zeta)\right|\le C\,\frac{\mu_{N}(\eta)}{N}\,U\left(\frac{\eta_{x}}{N}\right)\;.
\]
Therefore, by the same argument as in \eqref{e111} and in the proof
of Lemma \ref{lem751},
\begin{align*}
\left\Vert \chi_{x,\,y}^{(2\textrm{b})}\right\Vert ^{2} & \le\frac{C}{N^{2}}\sum_{\eta\in\mathcal{T}_{\textrm{ }}^{x,\,y},\,\eta_{x}\notin[N\epsilon^{1/2},\,N(1-\epsilon^{1/2})]}\mu_{N}(\eta)\,U^{2}\left(\frac{\eta_{x}}{N}\right)\\
 & \le\frac{C}{N^{\alpha+1}}\left[\int_{0}^{\epsilon^{1/2}}t^{\alpha}(1-t)^{\alpha}dt+\int_{1-\epsilon^{1/2}}^{1}t^{\alpha}(1-t)^{\alpha}dt\right]\;.
\end{align*}
It should be noticed that, in the last bound, the constant $C$ can
be chosen to be the one independent of $\epsilon$. Consequently,
\begin{equation}
\left\Vert \chi_{x,\,y}^{(2\textrm{b})}\right\Vert ^{2}\le o_{\epsilon}(1)\,N^{-(\alpha+1)}\;.\label{ecc3}
\end{equation}
Finally, the third flow is defined by
\[
\chi_{x,\,y}^{(2\textrm{c})}(\eta,\,\sigma^{y,\,x}\eta)=\mu_{N}(\eta)\,\frac{\beta_{N}}{N}\,\mathbf{1}\{\eta\in\mathcal{J}^{x,\,y},\,\eta_{x}\notin[N\epsilon^{1/2},\,N(1-\epsilon^{1/2})]\}\;.
\]
Then, by the same computations as before,
\begin{equation}
\left\Vert \chi_{x,\,y}^{(2\textrm{c})}\right\Vert ^{2}\le\frac{o_{N}(1)}{N^{2}}\,\mu_{N}(\mathcal{J}^{x,\,y})=o_{N}(1)\,N^{-(\alpha+1)}\;.\label{ecc4}
\end{equation}
Since
\[
\chi_{x,\,y}^{(2)}=\chi_{x,\,y}^{(2\textrm{a})}+\chi_{x,\,y}^{(2\textrm{b})}+\chi_{x,\,y}^{(2\textrm{c})}\;,
\]
the estimate of the flow norm of $\chi_{x,\,y}^{(2)}$ can be completed
by combining \eqref{ecc2}, \eqref{ecc3}, and \eqref{ecc4}.
\end{proof}
The divergence of $\chi_{x,\,y}$ is now considered.
\begin{prop}
\label{p84}The correction flow $\chi_{x,\,y}$ has the followings
properties:
\begin{enumerate}
\item The flow $\chi_{x,\,y}$ is divergence-free on $\text{\ensuremath{\left(\mathcal{V}^{x}\cup\mathcal{V}^{y}\cup\mathcal{J}^{x,\,y}\right)}}^{c}$.
\item The divergence of $\chi_{x,\,y}$ on $\partial^{\textup{\textrm{\,in}}}\mathcal{J}^{x,\,y}$,
$\mathcal{V}^{x}\setminus\mathcal{E}^{x}$, and $\mathcal{V}^{y}\setminus\mathcal{E}^{y}$
is negligible in the sense that
\[
\Bigl(\sum_{\eta\in\partial^{\textup{\textrm{\,in}}}\mathcal{J}^{x,\,y}}+\sum_{\eta\in\mathcal{V}^{x}\setminus\mathcal{E}^{x}}+\sum_{\eta\in\mathcal{V}^{y}\setminus\mathcal{E}^{y}}\Bigr)\bigl|(\textup{div }\chi_{x,\,y})(\eta)\bigr|=o_{N}(1)\,N^{-(1+\alpha)}\;.
\]
\item The divergence of $\chi_{x,\,y}$ on $\mathcal{E}^{x}$ and $\mathcal{\mathcal{E}}^{y}$
satisfies
\begin{align*}
 & (\textup{div }\chi_{x,\,y})(\mathcal{E}^{x})=\left(1+o_{N}(1)\right)\,N^{-(1+\alpha)}\,\frac{\textup{cap}_{X}(x,\,y)}{M_{\star}\,\Gamma(\alpha)\,I_{\alpha}}\;\;\text{and}\\
 & (\textup{div }\chi_{x,\,y})(\mathcal{E}^{y})=-\left(1+o_{N}(1)\right)\,N^{-(1+\alpha)}\,\frac{\textup{cap}_{X}(x,\,y)}{M_{\star}\,\Gamma(\alpha)\,I_{\alpha}}\;.
\end{align*}
In addition, $(\textup{div }\chi_{x,\,y})(\eta)>0$ for all $\eta\in\mathcal{E}^{x}$
and $(\textup{div }\chi_{x,\,y})(\eta)<0$ for all $\eta\in\mathcal{E}^{y}$.
\item The divergence of $\chi_{x,\,y}$ on $\mathcal{J}^{x,\,y}$ satisfies
\[
(\textup{div }\chi_{x,\,y})(\eta)=-(\textup{div }\Phi_{\mathbf{W}_{x,\,y}}^{*})(\eta)\;\;\text{for all }\eta\in\mathcal{J}_{\textup{\textrm{int}}}^{x,\,y}\;.
\]
\end{enumerate}
\end{prop}

The proof is postponed to Section \ref{sec84}. The corrected flow
is defined by
\[
\Phi_{x,\,y}=\Phi_{\mathbf{W}_{x,\,y}}^{*}+\chi_{x,\,y}\;.
\]
An interpretation of the previous proposition is now given in terms
of the correction procedure. By (3) of Proposition \ref{p84} and
the fact that $\Phi_{\mathbf{W}_{x,\,y}}^{*}$ is divergence-free
on $\mathcal{E}^{x}$ and $\mathcal{E}^{y}$, it follows that \begin{equation} \label{corr1} \begin{aligned}
&(\textup{div }\Phi_{x,\,y})(\mathcal{E}^{x})=\left(1+o_{N}(1)\right)N^{-(1+\alpha)}\,\frac{\textup{cap}_{X}(x,\,y)}{\kappa_{\star}\,M_{\star}\,\Gamma(\alpha)\,I_{\alpha}}\;\;\text{and}\\
&(\textup{div }\Phi_{x,\,y})(\mathcal{E}^{y})=-\left(1+o_{N}(1)\right)N^{-(1+\alpha)}\,\frac{\textup{cap}_{X}(x,\,y)}{\kappa_{\star}\,M_{\star}\,\Gamma(\alpha)\,I_{\alpha}}\;.
 \end{aligned} \end{equation} Moreover, by (4) of Proposition \ref{p84}, the flow $\Phi_{x,\,y}$
is divergence-free on $\mathcal{J}^{x,\,y}$. Hence, the divergence
of $\Phi_{\mathbf{W}_{x,\,y}}^{*}$ on $\mathcal{J}^{x,\,y}$ was
cleaned out by sending it to $\mathcal{E}^{x}$ and $\mathcal{E}^{y}$.
By (1) and (2) of Proposition \ref{p84}, this procedure has negligible
effect on the divergence of the remaining part, and by Proposition
\ref{p82} it does not essentially change the flow norm.

\subsection{\label{sec83}Global correction of $\Phi_{\mathbf{V}_{A,\,B}}^{*}$
and proof of Proposition \ref{p64}: special case}

Herein, the global correction for the flow $\Phi_{\mathbf{V}_{A,\,B}}^{*}$
is carried out. This procedure relies on the correction flows $\{\chi_{x,\,y}:x,\,y\in S_{\star}\}$
defined in the previous subsection.

By following the rule stated in Remark \ref{rm74}, it can be verified
that $\chi_{x,\,y}=-\chi_{y,\,x}$; hence, the following summation
is well-defined:
\begin{equation}
\chi_{A,\,B}=\sum_{\{x,\,y\}\in S_{\star}}\left[\mathfrak{h}_{A,\,B}(x)-\mathfrak{h}_{A,\,B}(y)\right]\chi_{x,\,y}\;.\label{chi}
\end{equation}
The test flow is finally defined by
\begin{equation}
\Phi_{A,\,B}=\Phi_{\mathbf{V}_{A,\,B}}^{*}+\chi_{A,\,B}\;.\label{pab}
\end{equation}
The following lemma is believed to hold in typical metastability situations.
Recall from Section \ref{sec61} the notation $\xi_{N}^{x}\in\mathcal{E}_{x}$,
$x\in S_{\star}$, which indicates the configuration for which all
the particles are located at site $x$.
\begin{lem}
\label{lem85}Suppose that $A$ and $B$ are disjoint non-empty subsets
of $S_{\star}$ and that $x\in S_{\star}\setminus(A\cup B)$. Then,
it holds that
\[
\lim_{N\rightarrow\infty}\sup_{\eta\in\mathcal{E}^{x}}\left|\mathbf{\mathbf{h}}_{\mathcal{E}(A),\,\mathcal{E}(B)}(\eta)-\mathbf{\mathbf{h}}_{\mathcal{E}(A),\,\mathcal{E}(B)}(\xi_{N}^{x})\right|=0\;.
\]
\end{lem}

\begin{proof}
Fix $A,\,B$, and $x\in S_{\star}\setminus(A\cup B)$. For $\eta\in\mathcal{E}^{x}\setminus\{\xi_{N}^{x}\}$,
define
\[
q_{N}(\eta)=\mathbb{P}_{\eta}^{N}\left[\tau_{\mathcal{E}(A\cup B)}<\tau_{\xi_{N}^{x}}\right]\;\;\text{and}\;\;p_{N}(\eta)=\mathbb{P}_{\eta}^{N}\left[\tau_{\mathcal{E}(A)}<\tau_{\mathcal{E}(B)}|\tau_{\mathcal{E}(A\cup B)}<\tau_{\xi_{N}^{x}}\right]\;.
\]
As $\mathbf{\mathbf{h}}_{\mathcal{E}(A),\mathcal{\,E}(B)}(\eta)=\mathbb{P}_{\eta}^{N}[\tau_{\mathcal{E}(A)}<\tau_{\mathcal{E}(B)}]$,
by the Markov property,
\begin{align*}
\mathbf{\mathbf{h}}_{\mathcal{E}(A),\mathcal{\,E}(B)}(\eta)=\, & \mathbb{P}_{\eta}^{N}\left[\tau_{\mathcal{E}(A)}<\tau_{\mathcal{E}(B)}|\tau_{\xi_{N}^{x}}<\tau_{\mathcal{E}(A\cup B)}\right](1-q_{N}(\eta))+p_{N}(\eta)\,q_{N}(\eta)\\
=\, & \mathbf{\mathbf{h}}_{\mathcal{E}(A),\,\mathcal{E}(B)}(\xi_{N}^{x})\,(1-q_{N}(\eta))+p_{N}(\eta)\,q_{N}(\eta)\;.
\end{align*}
Therefore,
\begin{equation}
\left|\mathbf{\mathbf{h}}_{\mathcal{E}(A),\mathcal{\,E}(B)}(\eta)-\mathbf{\mathbf{h}}_{\mathcal{E}(A),\mathcal{\,E}(B)}(\xi_{N}^{x})\right|=q_{N}(\eta)\left|p_{N}(\eta)-\mathbf{\mathbf{h}}_{\mathcal{E}(A),\mathcal{\,E}(B)}(\xi_{N}^{x})\right|\le q_{N}(\eta)\;.\label{sicn}
\end{equation}
It is well known (for instance, \cite[display (3.2)]{LL}) that
\begin{equation}
q_{N}(\eta)\le\frac{\textup{cap}_{N}(\eta,\,\mathcal{E}(A\cup B))}{\textup{cap}_{N}(\eta,\,\xi_{N}^{x})}\;.\label{sj2}
\end{equation}
By the monotonicity of the capacity, we have $\textup{cap}_{N}(\eta,\,\mathcal{E}(A\cup B))\le\textup{cap}_{N}(\mathcal{E}^{x},\,\breve{\mathcal{E}}^{x})$.
Hence, by \eqref{H1} and \eqref{sj2},
\[
\lim_{N\rightarrow\infty}\sup_{\eta\in\mathcal{E}^{x}}q_{N}(\eta)=0\;.
\]
In view of \eqref{sicn}, the last estimate completes the proof.
\end{proof}
Now we are ready to prove Proposition \ref{p64}.
\begin{proof}[Proof of Proposition \ref{p64}]
We claim that $\Phi_{A,\,B}$ defined in \eqref{pab} fulfills all
the requirements presented in the statement of the proposition. By
\eqref{chi} and the Cauchy-Schwarz inequality,
\[
||\chi_{A,\,B}||^{2}\le\frac{\kappa_{\star}(\kappa_{\star}-1)}{2}\sum_{\{x,\,y\}\subset S_{\star}}\left[\mathfrak{h}_{A,\,B}(x)-\mathfrak{h}_{A,\,B}(y)\right]^{2}||\chi_{x,\,y}||^{2}\;.
\]
Hence, by \eqref{pab} and Proposition \ref{p81},
\[
\left\Vert \Phi_{A,\,B}-\Phi_{\mathbf{V}_{A,\,B}}^{*}\right\Vert ^{2}=||\chi_{A,\,B}||^{2}=\left(o_{N}(1)+o_{\epsilon}(1)\right)N^{-(1+\alpha)}\;.
\]
This proves part (1).

Now we consider part (2). The set $\Delta_{N}$ can be decomposed
as
\[
(\mathcal{G}^{c})_{\textrm{int}}\cup\partial^{\text{\,out}}\mathcal{G}\cup\partial^{\textrm{\,in}}\mathcal{G}\cup\Bigl(\bigcup_{x,\,y\in S_{\star}}\mathcal{J}_{\text{int}}^{x,\,y}\Bigr)\cup\Bigl(\bigcup_{x\in S_{\star}}(\mathcal{D}_{\textrm{int}}^{x}\setminus\mathcal{E}^{x})\Bigr)\;.
\]
On $(\mathcal{G}^{c})_{\textrm{int}}$, both $\text{\ensuremath{\Phi_{\mathbf{V}_{A,\,B}}^{*}} and \ensuremath{\chi_{A,\,B}}}$
are divergence-free, by (1) of Proposition \ref{p81} and (1) of Proposition
\ref{p84}, respectively. Thus,
\begin{equation}
(\mbox{div }\Phi_{A,\,B})(\eta)=0\;\;\text{for all\;}\eta\in(\mathcal{G}^{c})_{\textrm{int}}\;.\label{pp21}
\end{equation}
On $\partial^{\text{\,out}}\mathcal{G}$, the flow $\chi_{A,\,B}$
is divergence-free by (1) of Proposition \ref{p84}; hence,
\[
(\mbox{div }\Phi_{A,\,B})(\eta)=(\mbox{div }\Phi_{\mathbf{V}_{A,\,B}}^{*})(\eta)\text{ for all }\eta\in\partial^{\text{out}}\mathcal{G}\;.
\]
Thus, by (2) of Proposition \ref{p84}
\begin{equation}
\sum_{\eta\in\partial^{\text{out}}\mathcal{G}}\left|(\textup{div }\Phi_{A,\,B})(\eta)\right|=o_{N}(1)\,N^{-(1+\alpha)}\;.\label{pp22}
\end{equation}
For $\partial^{\textrm{\,in}}\mathcal{G}$, by the triangle inequality,
by (2) of Proposition \ref{p81}, and by (1), (2) of Proposition \ref{p84},
we have
\begin{equation}
\sum_{\eta\in\partial^{\,\text{in}}\mathcal{G}}\left|(\textup{div }\Phi_{A,\,B})(\eta)\right|\le\sum_{\eta\in\partial^{\text{\,in}}\mathcal{G}}\left|(\textup{div }\Phi_{\mathbf{V}_{A,\,B}}^{*})(\eta)\right|+\sum_{\eta\in\partial^{\text{\,in}}\mathcal{G}}\left|(\textup{div }\chi_{x,\,y})(\eta)\right|=o_{N}(1)\,N^{-(1+\alpha)}\;.\label{pp23}
\end{equation}
On $\mathcal{J}_{\text{int}}^{x,y}$, by \eqref{e91}, the flow $\Phi_{A,B}$
can be written as
\[
\Phi_{A,\,B}(\eta,\,\zeta)=\left[\mathfrak{h}_{A,\,B}(y)-\mathfrak{h}_{A,\,B}(x)\right]\left(\Phi_{\mathbf{W}_{x,\,y}}^{*}+\chi_{x,\,y}\right)(\eta,\,\zeta)\;\;;\;\eta\in\mathcal{J}_{\textrm{int }}^{x,\,y}\;.
\]
Hence, by (4) of Proposition \ref{p84},
\begin{equation}
(\mbox{div }\Phi_{A,\,B})(\eta)=0\;\;\text{for all\;}\eta\in\mathcal{J}_{\textrm{int }}^{x,\,y},\;x,\,y\in S_{\star}\;.\label{pp24}
\end{equation}
Finally, on $\mathcal{D}_{\textrm{int}}^{x}\setminus\mathcal{E}^{x}$,
the flow $\Phi_{\mathbf{V}_{A,\,B}}^{*}$ is divergence-free by (1)
of Proposition \ref{p81}. Hence, by (1) and (2) of Proposition \ref{p84},
for all $x\in S_{\star}$, we have
\begin{equation}
\sum_{\eta\in\mathcal{D}_{\textrm{int}}^{x}\setminus\mathcal{E}^{x}}\left|(\textup{div }\Phi_{A,\,B})(\eta)\right|=\sum_{\eta\in\mathcal{V}^{x}\setminus\mathcal{E}^{x}}\left|(\textup{div }\chi_{A,\,B})(\eta)\right|=o_{N}(1)\,N^{-(1+\alpha)}\;.\label{pp25}
\end{equation}
Combining \eqref{pp21}--\eqref{pp25} yields the proof of part (2).

Part (3) is now considered. As $\Phi_{\mathbf{V}_{A,\,B}}^{*}$ is
divergence-free on $\mathcal{E}(S_{\star})$, it follows from (1)
and (3) of Proposition \ref{p84} that for $x\in S_{\star}\setminus(A\cup B)$,
\begin{equation} \label{sim} \begin{aligned}
(\textup{div }\Phi_{A,\,B})(\mathcal{E}^{x})&=\sum_{y\in S_{\star}\setminus\{x\}}\left(\mathfrak{h}_{A,\,B}(x)-\mathfrak{h}_{A,\,B}(y)\right)(\textup{div }\chi_{x,\,y})(\mathcal{E}^{x})\\&=N^{-(1+\alpha)}\left[o_{N}(1)+\sum_{y\in S_{\star}\setminus\{x\}}\left(\mathfrak{h}_{A,\,B}(x)-\mathfrak{h}_{A,\,B}(y)\right)\frac{\textup{cap}_{X}(x,\,y)}{\kappa_{\star}\,M_{\star}\,\Gamma(\alpha)\,I_{\alpha}}\right]\\&=N^{-(1+\alpha)}\left[o_{N}(1)-\mu(x)\,(\mathfrak{L}_{Y}\mathfrak{h}_{A,\,B})(x)\right]=o_{N}(1)\,N^{-(1+\alpha)}\;,
\end{aligned} \end{equation}where the last equality follows from the fact that $\mathfrak{L}_{Y}\mathfrak{h}_{A,B}\equiv0$
on $S_{\star}\setminus(A\cup B)$. This proves the first identity.
To prove the second identity, it is first claimed that
\begin{equation}
\sum_{\eta\in\mathcal{E}^{x}}\left|(\textup{div }\Phi_{A,\,B})(\eta)\right|\le C\,N^{-(1+\alpha)}\;.\label{cl1}
\end{equation}
To prove this, based on the trivial bound $|\mathfrak{h}_{A,\,B}(y)-\mathfrak{h}_{A,\,B}(x)|\le1$,
we have
\[
\sum_{\eta\in\mathcal{E}^{x}}\left|(\textup{div }\Phi_{A,\,B})(\eta)\right|\le\sum_{y\in S_{\star}\setminus\{x\}}\,\sum_{\eta\in\mathcal{E}^{x}}\left|(\textup{div }\chi_{x,\,y})(\eta)\right|\;.
\]
It should be noted that $(\textup{div }\chi_{x,\,y})(\eta)$, $y\in S_{\star}\setminus\{x\}$,
is positive for all $\eta\in\mathcal{E}^{x}$ by (3) of Proposition
\ref{p84}. Thus, the bound \eqref{cl1} is a direct consequence of
the estimate in (3) of Proposition \ref{p84}. Now we prove part (3).
By the triangle inequality and Lemma \ref{lem85},
\begin{align*}
 & \Bigl|\sum_{\eta\in\mathcal{E}_{N}^{x}}\mathbf{\mathbf{h}}_{\mathcal{E}(A),\mathcal{\,E}(B)}(\eta)\,(\mbox{div }\Phi_{A,\,B})(\eta)\Bigr|\\
 & \le o_{N}(1)\sum_{\eta\in\mathcal{E}_{N}^{x}}\left|(\mbox{div }\Phi_{A,\,B})(\eta)\right|+\mathbf{\mathbf{h}}_{\mathcal{E}(A),\mathcal{\,E}(B)}(\xi_{N}^{x})\Bigl|\sum_{\eta\in\mathcal{E}^{x}}(\mbox{div }\Phi_{A,\,B})(\eta)\Bigr|\;.
\end{align*}
The first term of the right-hand side is $o_{N}(1)\,N^{-(1+\alpha)}$
by \eqref{cl1}, whereas the second term is $o_{N}(1)\,N^{-(1+\alpha)}$
by \eqref{sim}. This verifies the second identity of part (3).

For part (4), by a computation as in \eqref{sim},
\begin{align*}
(\textup{div }\Phi_{A,B})(\mathcal{E}(A)) & =\sum_{x\in A}\,\sum_{y\in S_{\star}\setminus\{x\}}\left[\mathfrak{h}_{A,B}(x)-\mathfrak{h}_{A,B}(y)\right](\textup{div }\chi_{x,\,y})(\mathcal{E}^{x})\\
 & =\left(1+o_{N}(1)\right)N^{-(1+\alpha)}\,\sum_{x\in A}\mu(x)(-\mathfrak{L}_{Y}\mathfrak{h}_{A,\,B})(x)\\
 & =\left(1+o_{N}(1)\right)N^{-(1+\alpha)}\,\textup{cap}_{Y}(A,\,B)\;,
\end{align*}
where the last line follows from elementary properties of the capacity.
The proof of the estimate for $(\textup{div }\Phi_{A,\,B})(\mathcal{E}(B))$
is identical.
\end{proof}
It should be remarked that, the approximation $\Phi_{A,\,B}^{*}$
of $\Phi_{\mathbf{h}_{\mathcal{E}^{x},\,\mathcal{E}^{y}}^{*}}$ can
be constructed by an identical procedure and it can be verified that
this flow enjoys all the corresponding properties in Proposition \ref{p64}.

\subsection{\label{sec84}Proof of Proposition \ref{p84}}

The proof of Proposition \ref{p84} is divided into a series of lemmas.
The correspondence between these lemmas and Proposition \ref{p84}
will be explained at the end of the this subsection.
\begin{lem}
\label{lem1}The flow $\chi_{x,\,y}$ is divergence-free on $\text{\ensuremath{\left(\mathcal{V}^{x}\cup\mathcal{V}^{y}\cup\mathcal{J}^{x,\,y}\right)}}^{c}$.
\end{lem}

\begin{proof}
In view of the definition of $\chi_{x,\,y}$, the only part that should
be verified is $(\mathcal{\mathcal{T}}_{\textrm{int }}^{x,\,y}\cap\mathcal{D}^{x})\setminus\mathcal{V}^{x}$
and $(\mathcal{\mathcal{T}}_{\textrm{int }}^{x,\,y}\cap\mathcal{D}^{y})\setminus\mathcal{V}^{y}$,
as $\mathbf{B}(\eta;z_{i})=0$ for $\eta\in\mathcal{D}^{x}$ or $\eta\in\mathcal{D}^{y}$
for all $1\le i\le\kappa$. Hence, for $\eta\in(\mathcal{\mathcal{T}}_{\textrm{int }}^{x,\,y}\cap\mathcal{D}^{x})\setminus\mathcal{V}^{x}$
or $(\mathcal{\mathcal{T}}_{\textrm{int }}^{x,\,y}\cap\mathcal{D}^{y})\setminus\mathcal{V}^{y}$,
we have
\[
(\textup{div }\chi_{x,\,y})(\eta)=\chi_{x,\,y}(\eta,\,\sigma^{x,\,y}\eta)+\chi_{x,\,y}(\eta,\,\sigma^{y,\,x}\eta)=\mathbf{C}(\eta)-\mathbf{C}(\sigma^{y,\,x}\eta)\;.
\]
The last expression equals to $0$ since ,it the definition \eqref{sj3},
$\mathbf{C}(\cdot)$ is a function of $\widehat{\eta}$ only.
\end{proof}
\begin{lem}
\label{lem2-1}It holds that
\[
\sum_{\eta\in\partial^{\textup{\textrm{\,in}}}\mathcal{J}^{x,\,y}}\left|(\textup{div }\chi_{x,\,y})(\eta)\right|=o_{N}(1)N^{-(1+\alpha)}\;.
\]
\end{lem}

\begin{proof}
From the definition of $\chi_{x,\,y}$, it is immediate that $\left|(\textup{div }\chi_{x,\,y})(\eta)\right|\le C\mu_{N}(\eta)$
for some $C>0$. Hence, the lemma is a direct consequence of Lemma
\ref{lem71}.
\end{proof}
\begin{lem}
\label{lem2-2}It holds that
\[
\sum_{\eta\in\mathcal{V}^{x}\setminus\mathcal{E}^{x}}\left|(\textup{div }\chi_{x,\,y})(\eta)\right|=o_{N}(1)\,N^{-(1+\alpha)}\;\;\text{and\;\;}\sum_{\eta\in\mathcal{V}^{y}\setminus\mathcal{E}^{y}}\left|(\textup{div }\chi_{x,\,y})(\eta)\right|=o_{N}(1)\,N^{-(1+\alpha)}\;.
\]
\end{lem}

\begin{proof}
By the same reasoning as in the proof of Lemma \ref{lem1},
\begin{equation}
(\textup{div }\chi_{x,\,y})(\eta)=\begin{cases}
\chi_{x,\,y}(\eta,\,\sigma^{x,\,y}\eta)=\mathbf{C}(\eta) & \text{if }\eta\in\mathcal{V}_{x}\;,\\
\chi_{x,\,y}(\eta,\,\sigma^{y,\,x}\eta)=-\mathbf{C}(\sigma^{y,\,x}\eta)=-\mathbf{C}(\eta) & \text{if }\eta\in\mathcal{V}_{y}\;.
\end{cases}\label{sim2}
\end{equation}
Thus, by the definition of $\mathbf{C}(\cdot)$,
\begin{equation}
\sum_{\eta\in\mathcal{V}^{x}\setminus\mathcal{E}^{x}}\left|(\textup{div }\chi_{x,\,y})(\eta)\right|\le C\,N^{-(\alpha+1)}\sum_{\eta\in\mathcal{V}^{x}\setminus\mathcal{E}^{x}}\frac{m_{\star}^{\eta}}{a(\widehat{\eta})}\;.\label{mn1}
\end{equation}
For the configurations $\eta\in\mathcal{V}^{x}$, we have that $\eta_{y}=0$;
hence, by Lemma \ref{le71},
\begin{equation}
\sum_{\eta\in\mathcal{V}^{x}\setminus\mathcal{E}^{x}}\frac{m_{\star}^{\eta}}{a(\widehat{\eta})}=\sum_{k=\ell_{N}+1}^{\pi_{N}}\,\,\sum_{\zeta\in\mathcal{H}_{k,\,S\setminus\{x,y\}}}\frac{m_{\star}^{\zeta}}{a(\zeta)}\le C\sum_{k=\ell_{N}+1}^{\pi_{N}}\frac{1}{k^{\alpha}}\le\frac{C}{\ell_{N}^{\alpha-1}}=o_{N}(1)\label{mn2}
\end{equation}
It should be noted that the fact that $m_{\star}(x)=m_{\star}(y)=1$
was used at the first equality of \eqref{mn2}. Thus, the first estimate
of the lemma is obtained from \eqref{mn1} and \eqref{mn2}. The proof
for the second estimate is identical.
\end{proof}
\begin{lem}
\label{lem3}We have that
\begin{align*}
 & (\textup{div }\chi_{x,\,y})(\mathcal{E}^{x})=\left(1+o_{N}(1)\right)N^{-(1+\alpha)}\,\frac{\textup{cap}_{X}(x,\,y)}{\kappa_{\star}\,M_{\star}\,\Gamma(\alpha)\,I_{\alpha}}\;\;\text{and}\\
 & (\textup{div }\chi_{x,\,y})(\mathcal{E}^{y})=-\left(1+o_{N}(1)\right)N^{-(1+\alpha)}\frac{\textup{cap}_{X}(x,\,y)}{\kappa_{\star}\,M_{\star}\,\Gamma(\alpha)\,I_{\alpha}}\;.
\end{align*}
\end{lem}

\begin{proof}
By \eqref{sim2} and an argument similar to that in the previous lemma,
\begin{align*}
(\textup{div }\chi_{x,\,y})(\mathcal{E}^{x}) & =\sum_{\eta\in\mathcal{E}^{x}}\mathbf{C}(\eta)=\frac{\textup{cap}_{X}(x,\,y)}{N^{\alpha+1}\,Z_{N}\,M_{\star}\,I_{\alpha}}\sum_{k=0}^{\ell_{N}}\,\sum_{\zeta\in\mathcal{H}_{k,S\setminus\{x,y\}}}\frac{m_{\star}^{\zeta}}{a(\zeta)}\;.
\end{align*}
Hence, by Proposition \ref{e21} and Lemma \ref{lem72}, the first
identity of the lemma is proved. The proof for the second identity
is identical.
\end{proof}
An identity for the function $\mathbf{B}(\cdot;\cdot)$ is now established.
\begin{lem}
\label{le1}For all $\eta\in\mathcal{H}_{N}$ such that $\eta_{x}\ge1$,
\[
\sum_{i=1}^{\kappa}m(z_{i})\mathbf{\,B}(\sigma^{x,\,z_{i}}\eta;z_{i})=0\;.
\]
\end{lem}

\begin{proof}
For $1\le u<\kappa$, let
\[
\mathbf{Q}_{u}(\eta)=\left[h_{x,\,y}(z_{u})-h_{x,\,y}(z_{u+1})\right]\left[H\left(\frac{\eta^{(u)}}{N}\right)-H\left(\frac{\eta^{(u)}-1}{N}\right)\right]\;.
\]
Then,
\[
\mathbf{B}(\sigma^{x,\,z_{i}}\eta;z_{i})=\sum_{j:j>i}r(z_{i},\,z_{j})\sum_{u=i}^{j-1}\mathbf{Q}_{u}(\eta)-\sum_{j:j<i}r(z_{i},\,z_{j})\sum_{u=j}^{i-1}\mathbf{Q}_{u}(\eta)\;.
\]
Therefore,
\begin{align*}
 & \sum_{i=1}^{\kappa}m(z_{i})\,\mathbf{B}(\sigma^{x,\,z_{i}}\eta;z_{i})\\
 & =\sum_{i=1}^{\kappa}\sum_{j:j>i}m(z_{i})\,r(z_{i},\,z_{j})\sum_{u=i}^{j-1}\mathbf{Q}_{u}(\eta)-\sum_{i=1}^{\kappa}\sum_{j:j<i}m(z_{i})\,r(z_{i},\,z_{j})\sum_{u=j}^{i-1}\mathbf{Q}_{u}(\eta)\\
 & =\sum_{u=1}^{\kappa-1}\mathbf{Q}_{u}(\eta)\sum_{i=1}^{u}\sum_{j=u+1}^{\kappa}\left[m(z_{i})\,r(z_{i},\,z_{j})-m(z_{j})\,r(z_{j},\,z_{i})\right]\\
 & =\sum_{u=1}^{\kappa-1}\mathbf{Q}_{u}(\eta)\sum_{i=1}^{u}\sum_{j=1}^{\kappa}\left[m(z_{i})\,r(z_{i},\,z_{j})-m(z_{j})\,r(z_{j},\,z_{i})\right]=0\;.
\end{align*}
It should be remarked that the last equality is a consequence of \eqref{inv},
whereas the third equality holds owing to the following which in turn
holds by the symmetry of the summation:
\[
\sum_{i=1}^{u}\sum_{j=1}^{u}\left[m(z_{i})\,r(z_{i},\,z_{j})-m(z_{j})\,r(z_{j},\,z_{i})\right]=0\;\;;\;1\le u\le\kappa\;.
\]
\end{proof}
\begin{lem}
\label{lem4}The flow $\Phi_{x,\,y}$ is divergence-free on $\mathcal{J}_{\textrm{int }}^{x,\,y}$.
\end{lem}

\begin{proof}
We fix a configuration $\eta\in\mathcal{J}_{\textrm{int }}^{x,\,y}$
so that $\eta_{x},\,\eta_{y}\ge1$. By the definition of $\chi_{x,\,y}$
we have\begin{equation} \label{ed3} \begin{aligned} (\textup{div }\chi_{x,\,y})(\eta)=&\,\sum_{i=2}^{\kappa-1}\left\{ \chi_{x,\,y}(\eta,\,\sigma^{z_{i},\,x}\eta)+\chi_{x,\,y}(\eta,\,\sigma^{z_{i},\,y}\eta)\right\} \\&+\sum_{i=2}^{\kappa-1}\chi_{x,\,y}(\eta,\,\sigma^{x,\,z_{i}}\eta)+\sum_{i=2}^{\kappa-1}\chi_{x,\,y}(\eta,\,\sigma^{y,\,z_{k}}\eta)\\&+\chi_{x,\,y}(\eta,\,\sigma^{y,\,x}\eta)+\chi_{x,\,y}(\eta,\,\sigma^{x,\,y}\eta)\;.\end{aligned} \end{equation}
The right-hand side can be computed term by term. The terms $\chi_{x,\,y}(\eta,\,\sigma^{z_{i},\,x}\eta)$,
$\chi_{x,\,y}(\eta,\,\sigma^{z_{i},\,y}\eta)$ and $\chi_{x,\,y}(\eta,\,\sigma^{z_{\kappa},\,z_{1}}\eta)$
are immediate from the definition, whereas the term $\chi_{x,\,y}(\eta,\,\sigma^{x,\,z_{i}}\eta)$,
$2\le i\le\kappa-1$, can be evaluated as
\begin{align*}
\chi_{x,\,y}(\eta,\,\sigma^{x,\,z_{i}}\eta) & =-\chi_{x,\,y}(\sigma^{x,\,z_{i}}\eta,\,\eta)=-\chi_{x,\,y}(\sigma^{x,\,z_{i}}\eta,\,\sigma^{z_{i},\,x}(\sigma^{x,\,z_{i}}\eta))\\
 & =\frac{1}{2}\sum_{i=2}^{\kappa-1}a_{N}\,\mu_{N-1}(\sigma^{x,\,z_{i}}\eta-\omega^{z_{i}})\,m(z_{i})\,\mathbf{B}(\sigma^{x,\,z_{i}}\eta;z_{i})\\
 & =\frac{1}{2}\sum_{i=2}^{\kappa-1}a_{N}\,\mu_{N-1}(\eta-\omega^{x})\,m(z_{i})\,\mathbf{B}(\sigma^{x,\,z_{i}}\eta;z_{i})\;.
\end{align*}
It should be noted that the fact that $\eta_{x}\ge1$ was implicitly
used in this computation. Similarly, we can express $\chi_{x,\,y}(\eta,\,\sigma^{y,\,z_{i}}\eta)$
and $\chi_{x,\,y}(\eta,\,\sigma^{x,\,y}\eta)$. By inserting these
results into \eqref{ed3}, we obtain
\begin{align*}
(\textup{div }\chi_{x,\,y})(\eta)= & -\sum_{i=2}^{\kappa-1}a_{N}\,\mu_{N-1}(\eta-\omega^{z_{i}})\,m(z_{i})\,\mathbf{B}(\eta;z_{i})\mathbf{1}\{\eta_{z_{i}}\ge1\}\\
 & +\frac{1}{2}\sum_{i=2}^{\kappa-1}a_{N}\,\mu_{N-1}(\eta-\omega^{x})\,m(z_{i})\,\mathbf{B}(\sigma^{x,\,z_{i}}\eta;z_{i})\\
 & +\frac{1}{2}\sum_{i=2}^{\kappa-1}a_{N}\,\mu_{N-1}(\eta-\omega^{y})\,m(z_{i})\,\mathbf{B}(\sigma^{y,\,z_{i}}\eta;z_{i})\\
 & +\frac{1}{2}a_{N}\,\mu_{N-1}(\eta-\omega^{y})\left[m(z_{1})\,\mathbf{B}(\sigma^{y,\,x}\eta;z_{1})-m(z_{\kappa})\,\mathbf{B}(\eta;z_{\kappa})\right]-\mathbf{C}(\eta)\\
 & -\left\{ \frac{1}{2}a_{N}\,\mu_{N-1}(\eta-\omega^{x})\left[m(z_{1})\,\mathbf{B}(\eta;z_{1})-m(z_{\kappa})\,\mathbf{B}(\sigma^{x,\,y}\eta;z_{\kappa})\right]-\mathbf{C}(\eta)\right\} \;.
\end{align*}
Hence, by \eqref{ed1},
\begin{align*}
(\textup{div }\Phi_{x,\,y})(\eta)=\, & (\textup{div }\Phi_{\mathbf{W}_{x,\,y}}^{*})(\eta)+(\textup{div }\chi_{x,\,y})(\eta)\\
=\, & \frac{1}{2}a_{N}\,\mu_{N-1}(\eta-\mathfrak{\omega}^{x})\,\sum_{i=1}^{\kappa}m(z_{i})\,\mathbf{B}(\sigma^{x,\,z_{i}}\eta;z_{1})\\
 & +\frac{1}{2}a_{N}\,\mu_{N-1}(\eta-\omega^{y})\,\sum_{i=1}^{\kappa}m(z_{i})\,\mathbf{B}(\sigma^{y,\,z_{i}}\eta;z_{1})\;.
\end{align*}
The last expression is equal to $0$ by Lemma \ref{le1}. In particular,
the second summation can be rewritten as
\[
\sum_{i=1}^{\kappa}m(z_{i})\,\mathbf{B}\left(\sigma^{x,\,z_{i}}\left(\sigma^{y,\,x}\eta\right);z_{1}\right)\;.
\]
Thus, Lemma \ref{le1} can be applied by replacing $\eta$ with $\sigma^{y,\,x}\eta$
to verify that this summation is $0$.
\end{proof}
The proof of Proposition \ref{p84} can be completed by combining
the results obtained above.
\begin{proof}[Proof of Proposition \ref{p84}]
Part (1) is proven in Lemma \ref{lem1}, and part (2) is a direct
consequence of Lemmas \ref{lem2-1} and \ref{lem2-2}. Part (3) is
immediate from Lemma \ref{lem3} and \eqref{sim2}. Part (4) has been
verified in Lemma \ref{lem4}.
\end{proof}

\subsection{\label{sec85}Proof in the general case}

In the previous proof for the special case, the assumption \eqref{ass1}
was used only in the construction of the correction flow $\chi_{x,\,y}$.
Namely, this assumption allowed the definition of $\chi_{x,\,y}(\eta,\,\sigma^{u,\,v}\eta)$
for any $u,\,v\in S$ without any restriction. In the general case,
if $r(u,\,v)=0$, we cannot define $\chi_{x,\,y}(\eta,\,\sigma^{u,\,v}\eta)$;
therefore, the cleaning of divergence on $\eta$ is not immediate.
This can be resolved by the canonical path introduced in Section \ref{sec42}.
This argument is now explained in detail.

Let us fix two points $x,\,y\in S_{\star}$. Previously, for $\eta\in\mathcal{T}_{\textrm{int}}^{x,\,y}$
such that $\eta_{z_{i}}\ge1$, we defined
\begin{equation}
\chi_{x,\,y}(\eta,\,\sigma^{z_{i},\,x}\eta)=-\frac{1}{2}a_{N}\,\mu_{N}(\eta-\omega^{z_{i}})\,m(z_{i})\,\mathbf{B}(\eta;z_{i})\;.\label{ef1}
\end{equation}
As mentioned earlier, this object is meaningless if $r(z_{i},\,x)=0$.
Hence, the canonical path $z_{i}=w_{1},\,w_{2},\,\cdots,\,w_{k}=x$
between $z_{i}$ and $x$ should be invoked. By the property of the
canonical path, we have
\begin{equation}
r(w_{i},\,w_{i+1})>0\;\;\text{for all }1\le i\le k-1\;\;\text{and\;\;}k\le\kappa\;.\label{ef0}
\end{equation}
For each $\eta$ and $1\le i\le k$, let $\eta^{i}=\sigma^{w_{1},\,w_{i}}\eta$,
so that $\eta^{1}=z_{i}$ and $\eta^{k}=\sigma^{z_{i},\,x}\eta$.
With these notations, for $\eta\in\mathcal{T}_{\textrm{int}}^{x,\,y}$
with $\eta_{z_{i}}\ge1$, the previous definition \eqref{ef1} of
the correction flow can be replaced with
\begin{equation}
\widehat{\chi}_{x,\,y}(\eta^{1},\,\eta^{2})=\cdots=\widehat{\chi}_{x,\,y}(\eta^{k-1},\,\eta^{k})=-\frac{1}{2}a_{N}\,\mu_{N}(\eta-\omega^{z_{i}})\,m(z_{i})\,\mathbf{B}(\eta;z_{i})\;.\label{ef2}
\end{equation}
Of course, $\widehat{\chi}_{x,\,y}(\eta^{i+1},\,\eta^{i})$ is defined
by $-\widehat{\chi}_{x,\,y}(\eta^{i},\,\eta^{i+1})$ for all $i$.
The crucial observation here is that
\[
\eta^{i+1}=\sigma^{w_{i},\,w_{i+1}}\eta^{i}\text{ for all }1\le i\le k-1\;\;\text{and\;\;\ensuremath{\eta_{w_{i}}^{i}\ge1\;,}}
\]
so that $\widehat{\chi}_{x,\,y}(\eta^{i},\,\eta^{i+1})$ is meaningful
by \eqref{ef0}. Furthermore, the construction \eqref{ef2} changes
the divergence of $\eta^{1}=z_{i}$ and $\eta^{k}=\sigma^{z_{i}\,,x}\eta$
only; the divergence at $\eta^{i}$, $2\le i\le k-1$, is not affected
by this construction. Thus, as far as divergence is concerned, this
new object plays the exact same role as \eqref{ef1}. The flow $\widehat{\chi}_{x,\,y}(\eta,\,\sigma^{y,\,x}\eta)$
corresponding to $\chi_{x,\,y}(\eta,\,\sigma^{y,\,x}\eta)$ can be
constructed by a similar argument. The construction of the correction
flow $\widehat{\chi}_{x,\,y}$ can be thereby completed in the general
case.

The validity of Propositions \ref{p82} and \ref{p84} should now be
verified for this general object. For Proposition \ref{p82}, the
same argument can be used; the only difference is that the same edge
is used by several canonical paths. Here, the Cauchy-Schwarz inequality
can be used for obtaining an upper bound, as the number of canonical
paths using a certain edge is bounded by a uniform constant. In particular,
the uniform bound on the length of canonical paths, i.e., the condition
$k\le\kappa$ in \eqref{ef0}, is crucially used here. Moreover, by
the observation on the divergence in the previous paragraph, the validity
of Proposition \ref{p84} is immediate.
\begin{acknowledgement*}
I. Seo was supported by the National Research Foundation of Korea
(NRF) grant funded by the Korea government (MSIT) (No. 2018R1C1B6006896
and No. 2017R1A5A1015626). I. Seo wishes to thank Claudio Landim, and
Fraydoun Rezakhanlou for providing valuable ideas through numerous
discussions. Part of this work was done during the author's stay at
the IMPA for the conference ``XXI Escola Brasileira de Probabilidade''.
The author thanks IMPA for the hospitality and support for this visit.
\end{acknowledgement*}

\end{document}